\documentclass[12pt]{article}

\usepackage{amsfonts}
\usepackage{amsmath}
\usepackage{amsthm}
\usepackage[margin=2.5cm]{geometry}
\usepackage{epsfig}
\usepackage{graphicx}
\usepackage{color}

\newtheorem{thm}{Theorem}[section]
\newtheorem{lemma}[thm]{Lemma}
\newtheorem{prop}[thm]{Proposition}
\newtheorem{cor}[thm]{Corollary}

\newcommand{\abs}[1]{\left| #1 \right|}

\newcommand{\wh}{\widehat}
\newcommand{\wt}{\widetilde}

\newcommand{\Reals}{\mathbb{R}}

\newcommand{\Complex}{\mathbb{C}}
\newcommand{\Integers}{\mathbb{Z}}
\newcommand{\Halfplane}{\mathbb{H}}
\newcommand{\Disc}{\mathbb{D}}
\renewcommand{\Re}{\operatorname{Re}}
\renewcommand{\Im}{\operatorname{Im}}

\newcommand{\dist}{\operatorname{dist}}
\newcommand{\SLE}{\operatorname{SLE}}
\newcommand{\Arg}{\operatorname{Arg}}
\newcommand{\Es}{\operatorname{Es}}
\newcommand{\Gr}{\operatorname{Gr}}
\newcommand{\ind}{\textbf{1}}

\newcommand{\conj}[1]{\overline{#1}}

\newcommand{\Pro}[2]{\mathbf{P}^{#1} \left\{ #2 \right\}}
\newcommand{\Prob}[2]{\mathbf{P}^{#1} \left( #2 \right)}
\newcommand{\Exp}[2]{\mathbf{E}^{#1} \left[ #2 \right]}
\newcommand{\whPro}[2]{\widehat{\mathbf{P}}^{#1} \left\{ #2 \right \}}
\newcommand{\whExp}[2]{\widehat{\mathbf{E}}^{#1} \left[ #2 \right]}

\newcommand{\condPro}[3]{\mathbf{P}^{#1} \left\{ #2 \hskip5pt \vline \hskip5pt #3 \right\}}
\newcommand{\condExp}[3]{\mathbf{E}^{#1} \left[ #2 \hskip5pt \vline \hskip5pt #3 \right]}
\newcommand{\condProb}[3]{\mathbf{P}^{#1} \left( #2 \hskip5pt \vline \hskip5pt #3 \right)}
\newcommand{\condwhPro}[3]{\widehat{\mathbf{P}}^{#1} \left\{ #2 \hskip5pt \vline \hskip5pt #3 \right\}}

\newcommand{\qedd}{\nobreak \ifvmode \relax \else
      \ifdim\lastskip<1.5em \hskip-\lastskip
      \hskip1.5em plus0em minus0.5em \fi \nobreak
      \vrule height0.75em width0.5em depth0.25em\fi}

\newenvironment{remark}[1][Remark]{\begin{trivlist}
\item[\hskip \labelsep {\bfseries #1}]}{\end{trivlist}}

\title{The growth exponent for planar loop-erased random walk}
\author{Robert Masson \\ Department of Mathematics, University of British Columbia \\ 1984 Mathematics Road \\ Vancouver, BC V6T 1Z2, Canada. \\rmasson@math.ubc.ca}
\date{}

\begin{document}

\maketitle

\abstract
We give a new proof of a result of Kenyon that the growth exponent for loop-erased random walks in two dimensions is $5/4$. The proof uses the convergence of LERW to Schramm-Loewner evolution with parameter $2$, and is valid for irreducible bounded symmetric random walks on any discrete lattice of $\Reals^2$.

\vskip20pt


\vskip10pt


\vskip10pt

\noindent Submitted to EJP on August 8, 2008, final version accepted March 31, 2009.

\newpage

\tableofcontents

\section{Introduction}

\subsection{Overview}

Let $S$ be a random walk on a discrete lattice $\Lambda \subset \Reals^d$, started at the origin.  The loop-erased random walk (LERW) $\wh{S}^n$ is obtained by running $S$ up to the first exit time of the ball of radius $n$ and then chronologically erasing its loops. 

The LERW was introduced by Lawler \cite{Law80} in order to study the self-avoiding walk, but it was soon found that the two processes are in different universality classes. Nevertheless, LERW is extensively studied in statistical physics for two reasons. First of all, LERW is a model that exhibits many similarities to other interesting models: there is a critical dimension above which its behavior is trivial, it satisfies a domain Markov property, and it has a conformally invariant scaling limit. Furthermore, LERWs are often easier to analyze than these other models because properties of LERWs can often be deduced from facts about random walks. The other reason why LERWs are studied is that they are closely related to certain models in statistical physics like the uniform spanning tree (through Wilson's algorithm which allows one to generate uniform spanning trees from LERWs \cite{Wil96}), the abelian sandpile model \cite{Dha99} and the $b$-Laplacian random walk  \cite{Law91} (LERW is the case $b=1$).

Let $\Gr(n)$ be the expected number of steps of a $d$-dimensional LERW $\wh{S}^n$. Then the $d$ dimensional growth exponent $\alpha_d$ is defined to be such that
$$ \Gr(n) \approx n^{\alpha_d} $$
where
$f(n) \approx g(n)$ if
$$ \lim_{n \to \infty} \frac{\log f(n)}{\log g(n)} = 1. $$
For $d \geq 4$, it was shown by Lawler \cite{Law91, Law95} that $\alpha_d = 2$ (roughly speaking, in these dimensions, random walks do not produce many loops and LERWs have the same growth exponent as random walks). For $d = 3$, numerical simulations suggest that $\alpha_3$ is approximately $1.62$ \cite{AD01} but neither the existence of $\alpha_3$, nor its exact value has been determined rigorously (it is not expected to be a rational number). In the two dimensional case, it was shown by Kenyon \cite{Ken00} that $\alpha_2$ exists for simple random walk on the integer lattice $\Integers^2$ and is equal to $5/4$. His proof uses domino tilings to compute asymptotics for the number of uniform spanning trees of rectilinear regions of $\Reals^2$ and then uses the relation between uniform spanning trees and LERW to conclude that $\alpha_2 = 5/4$.

In this paper, we give a substantially different proof that $\alpha_2 = 5/4$. Namely, we prove 
\begin{thm} \label{growth} Let $S$ be an irreducible bounded symmetric random walk on a two-dimensional discrete lattice started at the origin and let $\sigma_n$ be the first exit time of the ball of radius $n$. Let $\wh{S}^n$ be the loop-erasure of $S[0,\sigma_n]$ and $\Gr(n)$ be the expected number of steps of $\wh{S}^n$. Then
$$ \Gr(n) \approx n^{5/4}.$$
\end{thm}
The proof of Theorem \ref{growth} uses the fact that LERW has a conformally invariant scaling limit called radial $\SLE_2$. Radial Schramm-Loewner evolution with parameter $\kappa \geq 0$ is a continuous random process from the unit circle to the origin in $\overline{\Disc}$. It was introduced by Schramm \cite{Sch00} as a candidate for the scaling limit of various discrete models from statistical physics. Indeed, he showed that if LERW has a conformally invariant scaling limit, then that limit must be $\SLE_2$. In the later paper by Lawler, Schramm and Werner \cite{LSW04}, the convergence of LERW to $\SLE_2$ was proved. Other models known to scale to SLE include the uniform spanning tree Peano curve ($\kappa = 8$, Lawler, Schramm and Werner \cite{LSW04}), the interface of the Ising model at criticality ($\kappa = 16/3$, Smirnov \cite{Smi07a}), the harmonic explorer ($\kappa = 4$, Schramm and Sheffield \cite{SS05}), the interface of the discrete Gaussian free field ($\kappa = 4$, Schramm and Sheffield \cite{SS06}), and the interface of critical percolation on the triangular lattice ($\kappa = 6$, Smirnov \cite{Smi01} and Camia and Newman \cite{CN06, CN07}). There is also strong evidence to suggest that the self-avoiding walk converges to $\SLE_{8/3}$, but so far, attempts to prove this have been unsuccessful \cite{LSW04b}.

One of the reasons to show convergence of discrete models to SLE is that properties and exponents for SLE are usually easier to derive than those for the corresponding discrete model. It is also widely believed that the discrete model will share the exponents of its corresponding SLE scaling limit. However, the equivalence of exponents between the discrete models and their scaling limits is not immediate. For instance, Lawler and Puckette \cite{LP00} showed that the exponent associated to the non-intersection of two random walks is the same as that for the non-intersection of two Brownian motions. In the case of discrete models converging to SLE, different techniques must be used, since the convergence is weaker than the convergence of random walks to Brownian motion. To the author's knowledge, the derivation of  arm exponents for critical percolation from disconnection exponents for $\SLE_6$ by Lawler, Schramm and Werner \cite{LSW02} and Smirnov and Werner \cite{SW01} is the only other example of exponents for a discrete model being derived from those for its SLE scaling limit.

There are three main reasons for giving a new proof that $\alpha_2 = 5/4$. The first is to give another example where an exponent for a discrete model is derived from its corresponding SLE scaling limit. The second reason is that the convergence of LERW to $\SLE_2$ holds for a general class of random walks on a broad set of lattices. This allows us to establish the exponent $5/4$ for irreducible bounded symmetric random walks on discrete lattices of $\Reals^2$, and thereby generalize Kenyon's result which holds only for simple random walks on $\Integers^2$. Finally, in the course of the proof we establish some facts about LERWs that are of interest on their own. Indeed, in a forthcoming paper with Martin Barlow \cite{BM09}, we use a number of the intermediary results in this paper to obtain second moment estimates for the growth exponent.

There are two properties of $\SLE_2$ that suggest that $\alpha_2 = 5/4$. The first is that the Hausdorff dimension of the SLE curves was established by Beffara \cite{Bef08}, and is equal to $5/4$ for $\SLE_2$. However, we have not found a proof that uses this fact directly. Instead, we use the fact that the probability that a complex Brownian motion from the origin to the unit circle does not intersect an independent $\SLE_2$ curve from the unit circle to the circle of radius $0 < r < 1$ is comparable to $r^{3/4}$. This and other exponents for SLE were established by Lawler, Schramm and Werner \cite{LSW01}. We use this fact to show that the probability that a random walk and an independent LERW started at the origin and stopped at the first exit time of the ball of radius $n$ do not intersect is logarithmically asymptotic to $n^{-3/4}$. We then relate this intersection exponent $3/4$ to the growth exponent $\alpha_2$ and show that $\alpha_2 = 5/4$.

\subsection{Outline of the proof of Theorem \ref{growth}}

While many of the details are quite technical, the main steps in the proof are fairly straightforward. Let $\Es(n)$ be the probability that a LERW and an independent random walk started at the origin do not intersect each other up to leaving $B_n$, the ball of radius $n$. As we mentioned in the previous section, the fact that $\Gr(n) \approx n^{5/4}$ follows from the fact that $\Es(n) \approx n^{-3/4}$. Intuitively, this is not difficult to see. Let $z$ be a point in $B_n$ that is not too close to the origin or the boundary. In order for $z$ to be on the LERW path, it must first be on the random walk path; the expected number of times the random walk path goes through $z$ is of order $1$. Then, in order for $z$ to be on the LERW path, it cannot be part of a loop that gets erased; this occurs if and only if the random walk path from $z$ to $\partial B_n$ does not intersect the loop-erasure of the random walk path from $0$ to $z$. This is comparable to $\Es(n)$. Therefore, since there are on the order of $n^2$ points in $B_n$, $ \Gr(n)$ is comparable to $n^2 \Es(n)$,
and so it suffices to show that $\Es(n) \approx n^{-3/4}$. The above heuristic does not work for points close to the origin or to the circle of radius $n$, and so the actual details are a bit more complicated.

Given $l \leq m \leq n$, decompose the LERW path $\wh{S}^n$ as
$$ \wh{S}^n = \eta^1 \oplus \eta^* \oplus \eta^2$$
(see Figure \ref{decomp}).
\begin{figure}[htp]
\centerline{
\epsfxsize=3.5in
\epsfysize=3.5in
\epsfbox{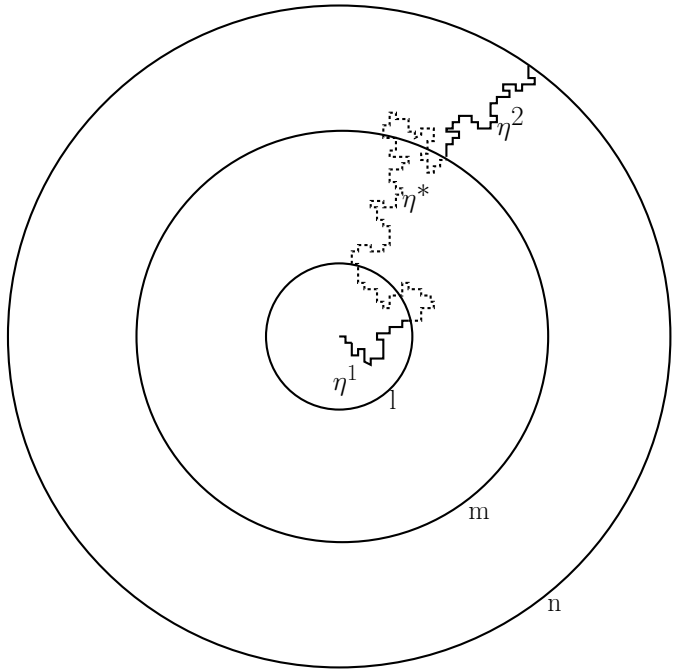}
}
\caption{Decomposition of a LERW path into $\eta^1$, $\eta^2$ and $\eta^*$}
 \label{decomp}
\end{figure}
Define $\Es(m,n)$ to be the probability that a random walk started at the origin leaves the ball $B_n$ before intersecting $\eta^2$. Notice that $\Es(m,n)$ is the discrete analog of the probability that a Brownian motion from the origin to the unit circle does not intersect an independent $\SLE_2$ curve from the unit circle to the circle of radius $m/n$. As mentioned in the previous section, the latter probability is comparable to $(m/n)^{3/4}$ \cite{LSW01}. Therefore, using the convergence of LERW to $\SLE_2$ and the strong approximation of Brownian motion by random walks one can show that there exists $C < \infty$ such that the following holds (Theorem \ref{conv}).  For all $0 < r < 1$, there exists $N$ such that for all $n > N$,
\begin{eqnarray} \label{agt}
\frac{1}{C} r^{3/4} \leq \Es(rn,n) \leq C r^{3/4}.
\end{eqnarray}

Unfortunately, $N$ in the previous statement depends on $r$, so one cannot simply take $r \to 0$ to recover $\Es(n)$. Therefore, one has to relate $\Es(n)$ to $\Es(m,n)$. This is not as easy as it sounds because the probability that a random walk avoids a LERW is highly dependent on the behavior of the LERW near the origin. Nevertheless, we show (Propositions \ref{d} and \ref{e}) that there exists $C < \infty$ such that
\begin{eqnarray} \label{qwd}
C^{-1} \Es(m) \Es(m,n) \leq \Es(n) \leq C \Es(m) \Es(m,n).
\end{eqnarray}
 It is then straightforward to combine (\ref{agt}) and (\ref{qwd}) to deduce that $\Es(n) \approx n^{-3/4}$ (Theorem \ref{escape}).

To prove (\ref{qwd}), we let $l = m/4$ in the decomposition given in Figure \ref{decomp}. Then in order for a random walk $S$ and a LERW $\wh{S}^n$ not to intersect up to leaving $B_n$, they must first reach the circle of radius $l$ without intersecting; this is $\Es(l)$. Next, we show that with probability bounded below by a constant, $\eta^*$ is contained in a fixed half-wedge (Corollary \ref{wedgecond}). We then use a separation lemma (Theorem \ref{sep}) which states that on the event $\Es(l)$, $S$ and $\wh{S}^n$ are at least a distance $cl$ apart at the circle of radius $l$. This allows us to conclude that, conditioned on the event $\Es(l)$, with a probability bounded below by a constant, $S$ will not intersect $\eta^*$. Finally, we use the fact that $\eta^1$ and $\eta^2$ are ``independent up to constants'' (Proposition \ref{indep}) to deduce that
$$ \frac{1}{C} \Es(l)\Es(m,n) \leq \Es(n) \leq C \Es(l) \Es(m,n).$$
Formula (\ref{qwd}) then follows because $m = 4l$ and thus $\Es(l)$ is comparable to $\Es(m)$.

\subsection{Structure of the paper}

In Chapter \ref{prelim}, we give precise definitions of random walks, LERWs and SLE and state some of the basic facts and properties that we require. 

In Chapter \ref{resultsrw}, we prove some technical lemmas about random walks. Section \ref{shitprob} establishes some estimates about Green's functions and the probability of a random walk hitting a set $K_1$ before another set $K_2$. Section \ref{srwavoid} examines the behavior of random walks conditioned to avoid certain sets. Finally, in Section \ref{rwapproxsec} we prove Proposition \ref{a} which states the following. For a fixed continuous curve $\alpha$ in the unit disc $\Disc$, the probability that a continuous random walk on the lattice $\delta \Lambda$ exits $\Disc$ before hitting $\alpha$ tends to the probability that a Brownian motion exits $\Disc$ before hitting $\alpha$. Furthermore, if one fixes $r$, then the convergence is uniform over all curves whose diameter is larger than $r$. 

Chapter \ref{LERWsec} is devoted to proving two results for LERW that are central to the main proof of the paper. The first is Proposition \ref{indep} which states that if $4l \leq m \leq n$ then $\eta^1$ and $\eta^2$ are independent up to a multiplicative constant (see Figure \ref{decomp}). The second result is a separation lemma for LERW. This key lemma states the following intuitive fact about LERW: there exist positive constants $c_1$ and $c_2$ so that, conditioned on the event that a random walk and a LERW do not intersect up to leaving the ball $B_n$, the probability that the random walk and the LERW are at least distance $c_1 n$ apart when they exit the ball $B_n$ is bounded below by $c_2$. Separation lemmas like this one are often quite useful in establishing exponents; a separation lemma was used in \cite{Law96} to establish the existence of the intersection exponent for two Brownian motions and in \cite{SW01} to derive arm exponents for critical percolation.

In Chapter \ref{intersection}, we prove that the growth exponent $\alpha_2 = 5/4$. To do this, we first relate the non-intersection of a random walk and a LERW to the non-intersection of a Brownian motion and an $\SLE_2$. Using the fact that the exponent for the latter is $3/4$, we deduce the same result for the former (Theorem \ref{escape}). Finally, we show how this implies that the growth exponent $\alpha_2$ for LERW is $5/4$ (Theorem \ref{growth}). 

\subsection{Acknowledgements}
I would like to thank Wendelin Werner for suggesting this problem to me. This work was done while I was a graduate student at the University of Chicago and I am very grateful to my advisors Steve Lalley and Greg Lawler for all their patient help and guidance. 

\section{Definitions and background} \label{prelim}

\subsection{Irreducible bounded symmetric random walks}

Throughout this paper, $\Lambda$ will be a two-dimensional discrete lattice of $\Reals^2$. In other words, $\Lambda$ is an additive subgroup of $\Reals^2$ not generated by a single element such that there exists an open neighborhood of the origin whose intersection with $\Lambda$ is just the origin. It can be shown (see for example \cite[Proposition 1.3.1]{LL08}) that $\Lambda$ is isomorphic as a group to $\Integers^2$.    

Now suppose that $V \subset \Lambda \setminus \{ 0 \}$ is a finite generating set for $\Lambda$ with the property that the first nonzero component of every $x \in V$ is positive. Suppose that $\kappa: V \to (0,1)$ is such that
$$ \sum_{x \in V} \kappa(x) \leq 1.$$
Let $p(x) = p(-x) = \kappa(x)/2$ for $x \in V$ and $p(0) = 1 - \sum_{x \in V} \kappa(x)$. Define the random walk $S$ with distribution $p$ to be
$$ S_n = X_1 + X_2 + \cdots + X_n.$$ 
where the random variables $X_k$ are independent with distribution $p$. Then $S$ is a symmetric, irreducible random walk  with bounded increments. It is a Markov chain with transition probabilities $p(x,y) = p(y-x)$.  

If $X = (X^1, X^2)$ has distribution $p$, then 
$$ \Gamma_{i,j} = \Exp{}{X^i X^j} \hskip25pt i,j=1,2$$
is the covariance matrix associated to $S$.  There exists a unique symmetric positive definite matrix $A$ such that $\Gamma = A^2$. Therefore, if $\wt{S}_j = A^{-1}S_j$, then $\wt{S}$ is a random walk on the discrete lattice $A^{-1} \Lambda$ with covariance matrix the identity. Since a linear transformation of a circle is an ellipse, it is clear that if we can show that the growth exponent $\alpha_2$ is $5/4$ for random walks whose covariance matrix is the identity, then $\alpha_2$ will be $5/4$ for random walks with arbitrary covariance matrix. Therefore, to simplify notation and proofs, throughout the paper $S$ will denote a symmetric, irreducible random walk on a discrete lattice $\Lambda$ with bounded increments and \emph{covariance matrix equal to the identity}. 

\subsection{A note about constants}

For the entirety of the paper, we will use the letters $c$ and $C$ to denote constants that may change from line to line but will only depend on the random walk $S$ (which will be fixed throughout).

Given two functions $f(n)$ and $g(n)$, we write $f(n) \approx g(n)$ if
$$ \lim_{n \to \infty} \frac{ \log f(n)}{\log g(n)} = 1,$$
and $f(n) \asymp g(n)$ if there exists $0 < C < \infty$ such that for all $n$
$$ \frac{1}{C} g(n) \leq f(n) \leq C g(n).$$
If $f(n) \to \infty$ and $g(n) \to \infty$ then $f(n) \asymp g(n)$ implies that $f(n) \approx g(n)$, but the converse does not hold.

\subsection{Subsets of $\Complex$ and $\Lambda$} \label{sets}

Recall that our discrete lattice $\Lambda$ and our random walk $S$ with distribution $p$ are fixed throughout.
     
Given $z \in \Complex$, let
$$ D_r(z) = D(z,r) = \{w \in \Complex : \abs{w - z} < r \}$$ be the open disk of radius $r$ centered at $z$ in $\Complex$, and
$$ B_n(z) = B(z,n) = D(z,n) \cap \Lambda$$
be the ball of radius $n$ centered at $z$ in $\Lambda$. We write $D_r$ for $D_r(0)$, $B_n$ for $B_n(0)$ and let $\Disc = D_1$ be the unit disk in $\Complex$.

We use the symbol $\partial$ to denote both the usual boundary of subsets of $\Complex$ and the outer boundary of subsets of $\Lambda$, where the outer boundary of a set $K \subset \Lambda$ (with respect to the distribution $p$) is 
$$ \partial K = \{ x \in \Lambda \setminus K : \text{ there exists $y \in K$ such that $p(x,y) > 0$} \}.$$ The context will make it clear whether we are considering a given set as a subset of $\Complex$ or of $\Lambda$. We will also sometimes consider the inner boundary
$$ \partial_i K = \{ x \in K : \text{ there exists $y \in \Lambda \setminus K$ such that $p(x,y) > 0$} \}.$$ We let $\conj{K} = K \cup \partial K$ and $K^{\circ} = K \setminus \partial_i K$. 

A path with respect to the distribution $p$ is a sequence of points 
$$\omega = [\omega_0, \omega_1, \ldots, \omega_k] \subset \Lambda$$ 
such that
$$ p(\omega) := \Pro{}{S_i = \omega_i: i=0, \ldots, k} = \prod_{i=1}^{k} p(\omega_{i-1},\omega_{i}) > 0.$$
We say that a set $K \subset \Lambda$ is connected (with respect to the distribution $p$) if for any pair of points $x, y \in K$, there exists a path $\omega \subset K$ connecting $x$ and $y$.  

Given $l \leq m \leq n$, let $\Omega_l$ be the set of paths $\omega = [0, \omega_1, \ldots, \omega_k] \subset \Lambda$ such that $\omega_j \in B_l$, $j=1, \ldots, k-1$ and $\omega_k \in \partial B_l$. Let $\widetilde{\Omega}_{m,n}$ be the set of paths $\lambda = [\lambda_0, \lambda_1, \ldots, \lambda_{k'}]$ such that $\lambda_0 \in \partial B_m$, $\lambda_j \in A_{m,n}$, $j=0, 1, \ldots, k' - 1$ and $\lambda_{k'} \in \partial B_{n}$, where $A_{m,n}$ denotes the annulus $B_n \setminus B_m$.

Suppose that $l \leq m \leq n$ and that $\eta = [0, \eta_1, \ldots, \eta_k] \in \Omega_n$. Let
$$ k_1 = \min \{j \geq 1 : \eta_j \notin B_l \} \hskip25pt k_2 = \max \{j \geq 1 : \eta_j \in B_m \}.$$
Then (see Figure \ref{decomp}), $\eta$ can be decomposed as $\eta = \eta^1 \oplus \eta^* \oplus \eta^2$ where
\begin{eqnarray*}
\eta^1 &=& \eta^1_l(\eta) = [0, \ldots, \eta_{k_1}] \in \Omega_l \\
\eta^2 &=& \eta^2_{m,n}(\eta)  = [\eta_{k_2 + 1}, \ldots, \eta_k] \in \widetilde{\Omega}_{m,n} \\
\eta^* &=& \eta^*_{l,m,n}(\eta) = [\eta_{k_1 + 1}, \ldots, \eta_{k_2}].
\end{eqnarray*}

\subsection{Basic facts about Brownian motion and random walks}
Throughout this paper, $W_t$, $t \geq 0$ will denote a standard complex Brownian motion. Given a set $K \subset \Lambda$, let
$$ \sigma_K = \min \{j \geq 1 : S_j \notin K \} \hskip20pt \overline{\sigma}_K = \min \{j \geq 0 : S_j \notin K \} $$
be first exit times of the set $K$. We also let
$$ \xi_K = \min \{j \geq 1 : S_j \in K \} \hskip20pt \overline{\xi}_K = \min \{j \geq 0 : S_j \in K \} $$
be the first hitting times of the set $K$. We let $\sigma_n = \sigma_{B_n}$ and use a similar convention for $\overline{\sigma}_n$, $\xi_n$ and $\overline{\xi}_n$. We also define the following stopping times for Brownian motion: given a set $D \subset \Complex$, let
$$ \tau_D = \min \{t \geq 0 : W_t \in \partial D \}.$$
Depending on whether the Brownian motion is started inside or outside $D$, $\tau_D$ will be either an exit time or a hitting time.

Suppose that $X$ is a Markov chain on $\Lambda$ and that $K \subset \Lambda$. Let
$$ \sigma^X_K = \min \{ j \geq 1: X_j \notin K \}.$$ 
For $x,y \in K$, we let
$$G_K^X(x,y) = \Exp{x}{\sum_{j=0}^{\sigma^X_K - 1} \ind \{X_j = y \}}$$
denote the Green's function for $X$ in $K$. We will sometimes write $G^X(x,y; K)$ for $G^X_K(x,y)$ and also abbreviate $G^X_K(x)$ for $G_K^X(x,x)$. When $X = S$ is a random walk, we will omit the superscript $S$.

Recall that a function $f$ defined on $\overline{K} \subset \Lambda$ is discrete harmonic (with respect to the distribution $p$) if for all $z \in K$,
$$ \mathcal{L}f(z) := -f(z) + \sum_{x \in \Lambda} p(x - z)f(x) = 0.$$
For any two disjoint subsets $K_1$ and $K_2$ of $\Lambda$, it is easy to verify that that the function
$$ h(z) = \Pro{z}{\overline{\xi}_{K_1} < \overline{\xi}_{K_2}}$$
is discrete harmonic on $\Lambda \setminus (K_1 \cup K_2)$.
The following important theorem concerning discrete harmonic functions will be used repeatedly in the sequel \cite[Theorem 6.3.9]{LL08}.

\begin{thm} [Discrete Harnack Principle] Let $U$ be a connected open subset of $\Complex$ and
$A$ a compact subset of $U$. Then there exists a constant $C(U,A)$ such that for all $n$ and all positive harmonic functions $f$ on $nU \cap \Lambda$ 
$$f(x) \leq C(U,A) f(y)$$ for all $x, y \in nA \cap \Lambda$.
\end{thm}

Suppose that $X$ is a Markov chain with hitting times 
$$ \overline{\xi}^X_K = \min \{ j \geq 0 : X \in K \}. $$ 
Given two disjoint subsets $K_1$ and $K_2$ of $\Lambda$, let $Y$ be $X$ conditioned to hit $K_1$ before $K_2$ (as long as this event has positive probability). Then if we let $h(z) = \Pro{z}{\overline{\xi}^X_{K_1} < \overline{\xi}^X_{K_2}}$, $Y$ is a Markov chain with transition probabilities
$$ p^Y(x,y) = \frac{h(y)}{h(x)}p^X(x,y).$$ 
Therefore, if $\omega = [\omega_0, \ldots, \omega_k]$ is a path with respect to $p^X$ in $\Lambda \setminus (K_1 \cup K_2)$,
\begin{eqnarray}
p^Y(\omega) = \frac{h(\omega_k)}{h(\omega_0)} p^X(\omega). \label{htransform}
\end{eqnarray}
Using this fact, the following lemma follows readily.

\begin{lemma} \label{greencondit} Suppose that $X$ is a Markov chain and let $Y$ be $X$ conditioned to hit $K_1$ before $K_2$. Suppose that $K \subset \Lambda \setminus (K_1 \cup K_2)$. Then for any $x,y \in K$,
$$ G^Y_K(x, y) = \frac{h(y)}{h(x)} G^X_K(x,y).$$
In particular, $G^Y_K(x) = G^X_K(x)$.
\end{lemma}
Finally, we recall an important theorem concerning the intersections of random walks and Brownian motion with continuous curves.

\begin{thm} [Beurling estimates] \label{beurling} 

\noindent \begin{enumerate}
\item There exists a constant $C < \infty$ such that the following holds. Suppose that $\alpha: [0, t_\alpha] \to \Complex$ is a continuous curve such that $\alpha(0) = 0$ and $\alpha(t_\alpha) \in \partial D_r$. Then if $z \in D_r$,
$$ \Pro{z}{W[0,\tau_r] \cap \alpha[0, t_\alpha] = \emptyset} \leq C \left( \frac{\abs{z}}{r} \right)^{1/2}.$$
\item There exists a constant $C < \infty$ such that the following holds. Suppose that $\omega$ is a path from the origin to $\partial B_n$. Then if $z \in B_n$,
$$ \Pro{z}{S[0,\sigma_n] \cap \omega = \emptyset} \leq C \left( \frac{\abs{z}}{n} \right)^{1/2}.$$
\end{enumerate}
\end{thm}

\begin{proof}
The statement about Brownian motion can be found, for example, in \cite[Theorem 3.76]{Law05}. The statement about random walks was originally proved in \cite{Kes87}; a formulation that is closer to the one given above can be found in \cite{LL04}.
\end{proof}

\subsection{Loop-erased random walk}

We now describe the loop-erasing procedure and various definitions of the loop-erased random walk (LERW). Given a path $\lambda = [\lambda_0, \ldots, \lambda_m]$ in $\Lambda$, we let $\operatorname{L}(\lambda) = [\hat{\lambda}_0, \ldots, \hat{\lambda}_{n}]$ denote its chronological loop-erasure. More precisely, we let
$$ s_0 = \sup \{ j : \lambda(j) = \lambda(0) \},$$
and for $i > 0$,
$$ s_i = \sup \{ j: \lambda(j) = \lambda(s_{i-1} + 1) \}.$$
Let $$n = \inf \{ i: s_i = m \}.$$
Then $$ \operatorname{L} (\lambda) = [\lambda(s_0), \lambda(s_1), \ldots, \lambda(s_n)].$$
Note that one may obtain a different result if one performs the loop-erasing procedure backwards instead of forwards. In other words, if we let $\lambda^R = [\lambda_m, \ldots, \lambda_0]$, then in general, $\operatorname{L}(\lambda^R) \neq \operatorname{L}(\lambda)^R$. However, if $\lambda$ has the distribution of a random walk, then $\operatorname{L}(\lambda^R)$ has the same distribution as $\operatorname{L}(\lambda)^R$ \cite[Lemma 7.2.1]{Law91}.

Now suppose that $S$ is a random walk on $\Lambda$ and $K$ is a proper subset of $\Lambda$. We define the LERW $\widehat{S}^K$ to be the process
$$ \widehat{S}^K = \operatorname{L}(S[0, \sigma_K]).$$
In other words, we run $S$ up to the first exit time of $K$ and then erase loops. We write $\widehat{S}^n$ for $\widehat{S}^{B_n}$. We also define the following stopping times. Given $A \subset K$, we let
$$ \widehat{\sigma}^K_A = \min \{j \geq 1 : \widehat{S}^K_j \notin A \}.$$ If either $A$ or $K$ is a ball $B_n$, we replace $A$ or $K$ by $n$ in the subscript or superscript.

Different sets $K$ will produce different LERWs $\wh{S}^K$, but one can define an ``infinite LERW" as follows. 
For $\omega \in \Omega_l$, and $n > l$ let
$$ \mu_{l, n}(\omega) = \Pro{}{\wh{S}[0, \wh{\sigma}^n_l] = \omega}.$$
Then one can show \cite[Proposition 7.4.2]{Law91} that there exists a limiting measure $\mu_l$ such that
$$ \lim_{n \to \infty} \mu_{l,n}(\omega) = \mu_l(\omega).$$
The $\mu_l$ are consistent and therefore there exists a measure $\mu$ on infinite self-avoiding paths. We call the associated process the infinite LERW and denote it by $\wh{S}$. In this paper, we will consider both the infinite LERW $\wh{S}$, and LERWs $\wh{S}^K$ obtained by stopping a random walk at the first exit time of $K$ and then erasing loops.

Suppose that $X$ is a Markov chain and $\omega = [\omega_0, \ldots, \omega_k]$ is a path in $\Lambda$ with respect to $p^X$. One can write down an exact formula for the probability that the first $k$ steps of the loop-erased process $\wh{X}^K$ are equal to $\omega$. Letting $A_j = \{\omega_0, \ldots, \omega_j\}$, $j=0, \ldots, k$, $A_{-1} = \emptyset$, and $G^X(.;.)$ be the Green's function for $X$, we define
\begin{eqnarray}
 G_K^X(\omega) = \prod_{j=0}^k G^X(\omega_j; K \setminus A_{j-1}). \label{GK}
\end{eqnarray}
 Then \cite{Law99},
\begin{eqnarray}
\Pro{}{\wh{X}^K[0,k] = \omega} = p^X(\omega) G^X_K(\omega) \Pro{\omega_k}{\sigma^X_K < \xi^X_\omega}. \label{lerwf}
\end{eqnarray}

We can use the previous formula to show that while LERW is certainly not a Markov chain, it does satisfy the following ``domain Markov property'': for any Markov chain $X$, if we condition the initial part of $\wh{X}$ to be equal to $\omega$, the rest of $\wh{X}$ can be obtained by running $X$ conditioned to avoid $\omega$ and then loop-erasing.

\begin{lemma} [Domain Markov Property] \label{condit}
Let $X$ be a Markov chain,  $K \subset \Lambda$ and $\omega = [\omega_0, \omega_1, \ldots, \omega_k]$ be a path in $K$ (with respect to $p^X$). Define a new Markov chain $Y$ to be $X$ started at $\omega_k$ conditioned on the event that $X[1,\sigma^X_K] \cap \omega = \emptyset$. Suppose that $\omega' = [\omega'_0 \ldots, \omega'_{k'}]$ is such that $\omega \oplus \omega'$ is a path from $\omega_0$ to $\partial K$. Then,
$$ \condPro{}{\widehat{X}^K[0,\widehat{\sigma}^X_K] = \omega \oplus \omega'}{\widehat{X}^K[0, k] = \omega}  = \Pro{}{\widehat{Y}^K[1,\widehat{\sigma}^Y_K] = \omega'}.$$
\end{lemma}

\begin{proof}
Let $G^X(.;.)$ and $G^Y(.;.)$ be the Green's functions for $X$ and $Y$ respectively. Then by formula (\ref{lerwf}),
\begin{eqnarray*}
\Pro{}{\widehat{X}^K[0,\widehat{\sigma}^X_K] = \omega \oplus \omega'} &=& p^X(\omega \oplus \omega') G^X_K(\omega) G^X_{K \setminus \omega}(\omega') ;\\
\Pro{}{\widehat{X}^K[0, k] = \omega} &=& p^X(\omega) G^X_K(\omega) \Pro{\omega_k}{\sigma^X_K < \xi^X_\omega};\\
\Pro{}{\widehat{Y}^K[0,\widehat{\sigma}^Y_K] = \omega'} &=& p^Y(\omega') G^Y_{K}(\omega').
\end{eqnarray*}
However, 
$$ p^X(\omega \oplus \omega') = p^X(\omega) p^X(\omega'),$$
$$  p^Y(\omega') = \frac{p^X(\omega')}{\Pro{\omega_k}{\sigma^X_K < \xi^X_\omega}},$$ 
and by Lemma \ref{greencondit},
$$ G^Y_K(\omega') = G^Y_{K \setminus \omega}(\omega') = G^X_{K \setminus \omega}(\omega').$$
\end{proof}

\subsection{Schramm-Loewner evolution}

\label{sSLE}

In this subsection, we give a brief description of Schramm-Loewner evolution. For a much more thorough introduction to SLE, see for instance \cite{Law05} or \cite{Wer04}.

Suppose that $\gamma:[0, \infty] \to \overline{\Disc}$ is a simple continuous curve such that $\gamma(0) \in \partial \Disc$, $\gamma(0, \infty] \subset \Disc$ and $\gamma(\infty) = 0$. Then by the Riemann mapping theorem, for each $t \geq 0$, there exists a unique conformal map $g_t : \Disc \setminus \gamma(0,t] \to \Disc$ such that $g_t(0) = 0$ and $g_t'(0) > 0$. The quantity $\log g_t'(0)$ is called the capacity of $\Disc \setminus \gamma(0,t]$ from $0$. By the Schwarz Lemma, $g_t'(0)$ is increasing in $t$ and therefore, one can reparametrize $\gamma$ so that $g_t'(0) = e^t$; this is the capacity parametrization of $\gamma$. For each $t \geq 0$, one can verify that
$$ U_t := \lim_{z \to \gamma(t)} g_t(z)$$
exists and is continuous as a function of $t$. Also, $g_t$ and $U_t$ satisfy Loewner's equation
\begin{eqnarray}
\dot{g}_t(z) = g_t(z) \frac{U_t + g_t(z)}{U_t - g_t(z)}, \hskip30pt g_0(z) = z. \label{loewner}
\end{eqnarray}
Therefore, given a simple curve $\gamma$ as above, one produces a curve $U_t$ on the unit circle satisfying (\ref{loewner}). One calls $U_t$ the driving function of $\gamma$.

The idea behind the Schramm-Loewner evolution is to start with a driving function $U_t$ and use that to generate the curve $\gamma$. Indeed, given a continuous curve $U:[0,\infty] \to \partial \Disc$ and $z \in \Disc$, one can solve the ODE (\ref{loewner}) up to the first time $T_z$ that $g_t(z) = U_t$. If we let $K_t = \{ z \in \overline{\Disc} : T_z \leq t \}$ then one can show that $g_t$ is a conformal map from $\Disc \setminus K_t$ onto $\Disc$ such that $g_t(0) = 0$ and $g_t'(0) = e^t$. We note that there does not necessarily exist a curve $\gamma$ such that $K_t = \gamma[0,t]$ as was the case above.

The radial Schramm-Loewner evolution arises as a special choice of the driving function $U_t$. For each $\kappa > 0$, we let $U_t = e^{i\sqrt{\kappa} B_t}$ where $B_t$ is a standard one dimensional Brownian motion. Then the resulting random maps $g_t$ and sets $K_t$ are called radial $\SLE_\kappa$. It is possible to show that with probability $1$, there exists a curve $\gamma$ such that $\Disc \setminus K_t$ is the connected component of $\Disc \setminus \gamma[0,t]$ containing $0$ (see \cite{RS05} for the case $\kappa \neq 8$ and \cite{LSW04} for $\kappa = 8$). In \cite{RS05} it was shown that if $\kappa \leq 4$ then $\gamma$ is a.s. a simple curve and if $\kappa > 4$, $\gamma$ is a.s. not a simple curve. One refers to $\gamma$ as the radial $\SLE_\kappa$ curve.

One defines radial $\SLE_\kappa$ in other simply connected domains to be such that $\SLE_\kappa$ is conformally invariant. Given a simply connected domain $D \neq \Complex$, $z \in D$ and $w \in \partial D$, there exists a unique conformal map $f:\Disc \to D$ such that $f(0) = z$ and $f(1) = w$. Then $\SLE_\kappa$ in $D$ from $w$ to $z$ is defined to be the image under $f$ of radial $\SLE_\kappa$ in $\Disc$ from $1$ to $0$. 

We will focus on the case $\kappa = 2$, and throughout $\gamma:[0, \infty] \to \overline{\Disc}$ will denote radial $\SLE_2$ in $\Disc$ started uniformly on $\partial \Disc$. If $D \subset \Disc$, we let 
$$ \wh{\tau}_D = \inf \{t \geq 0 : \gamma(t) \in \partial D \}.$$

We conclude this section with precise statements of the two facts about $\SLE_2$ that were mentioned in the introduction: the intersection exponent for $\SLE_2$ and the weak convergence of LERW to $\SLE_2$.

\begin{thm}[Lawler, Schramm, Werner \cite{LSW01}] \label{harmeasure} Let $\gamma$ be radial $\SLE_\kappa$ from $1$ to $0$ in $\Disc$ and for $0 < r <1$, let $\wh{\tau}_r$ be the first time $\gamma$ enters the disk of radius $r$. Let $W$ be an independent complex Brownian motion started at $0$. Then
$$ \Pro{}{W[0,\tau_\Disc] \cap \gamma[0,\wh{\tau}_r] = \emptyset} \asymp r^{\nu},$$
where 
$$ \nu(\kappa) = \frac{\kappa + 4}{8}.$$
In particular,  $\nu = 3/4$ for $\SLE_2$.
\end{thm}

In order to state the convergence of LERW to $\SLE_2$ we require some notation. Let $\Gamma$ denote the set of continuous curves $\alpha:[0,t_\alpha] \to \overline{\Disc}$ (we allow $t_\alpha$ to be $\infty$) such that $\alpha(0) \in \partial \Disc$, $\alpha(0,t_\alpha] \subset \Disc$ and $\alpha(t_\alpha) = 0$. We can make $\Gamma$ into a metric space as follows. If $\alpha, \beta \in \Gamma$, we let
$$ d(\alpha, \beta) =  \inf \sup_{0 \leq t \leq t_{\alpha}}  \abs{\alpha(t) - \beta(\theta(t))},$$
where the infimum is taken over all continuous, increasing bijections $\theta:[0,t_{\alpha}] \to [0,t_{\beta}]$. Note that $d$ is a pseudo-metric on $\Gamma$, and is a metric if we consider two curves to be equivalent if they are the same up to reparametrization.

Let $f$ be a continuous function on $\Gamma$, $\gamma$ be radial $\SLE_2$, and extend $\wh{S}^n$ to a continuous curve by linear interpolation (so that the time reversal of $n^{-1} \wh{S}^n$  is in $\Gamma$), then

\begin{thm} [Lawler, Schramm, Werner \cite{LSW04}] \label{LERWtoSLE}   
$$ \lim_{n \to \infty} \Exp{}{f( n^{-1} \wh{S}^n)} = \Exp{}{f(\gamma)}.$$
\end{thm}

\section{Some results for random walks} \label{resultsrw}

In this section we establish some technical lemmas concerning random walks that will be used repeatedly in the sequel.

\subsection{Hitting probabilities and Green's function estimates} \label{shitprob}

Recall that $\xi_K$ is the first hitting time of the set $K$ and $G(.;\Lambda \setminus K)$ is the Green's function in the set $\Lambda \setminus K$.

\begin{lemma} \label{rwdecomp}
Let $K_1, K_2 \subset \Lambda$ be disjoint and $z \in \Lambda \setminus (K_1 \cup K_2)$. Then,
\begin{eqnarray*}
& & \Pro{z}{\xi_{K_1} < \xi_{K_2}} \\
&=& \frac{G\left(z; \Lambda \setminus (K_1 \cup K_2)\right)}{G(z; \Lambda \setminus K_1)} \sum_{y \in \partial_i K_1} \condPro{y}{\xi_z < \xi_{K_2}}{\xi_z < \xi_{K_1}} \Pro{z}{S(\xi_{K_1}) = y}.
\end{eqnarray*}
\end{lemma}

\begin{proof}
We begin by showing that for any $K \subset \Lambda$, $z \in \Lambda \setminus K$ and $y \in \partial_i K$,
$$ \Pro{z}{S(\xi_K) = y} = G(z ; \Lambda \setminus K) \Pro{z}{S(\xi_K \wedge \xi_{z}) = y}.$$
To prove this, we proceed as in the proof of \cite[Lemma 2.1.1]{Law91}. Let 
$$\tau = \sup \{ j < \xi_K : S_j = z \}.$$ 
Note that $\tau$ is not a stopping time. However, since $\tau < \xi_K$,
\begin{eqnarray*}
& & \Pro{z}{S(\xi_K) = y} \\
&=& \sum_{k=1}^\infty \Pro{z}{\xi_K = k; S_k = y} \\
&=& \sum_{k=1}^\infty \sum_{j=0}^{k-1} \Pro{z}{\xi_K = k; S_k = y; \tau = j} \\
&=& \sum_{j=0}^\infty \sum_{k=j+1}^\infty \Pro{z}{S_j = z; S_k = y; S_i \notin K, 0 \leq i \leq j;  S_i \notin K \cup \{z\}, j+1 \leq i < k} \\
&=& \sum_{j=0}^\infty \Pro{z}{S_j = z; S_i \notin K, 0 \leq i \leq j} \sum_{k=1}^\infty \Pro{z}{S_k = y, S_i \notin K \cup \{z\}, 1 \leq i < k} \\
&=& G_{\Lambda \setminus K}(z) \Pro{z}{S(\xi_K \wedge \xi_{z}) = y}
\end{eqnarray*}
Applying the previous equality to $K = K_1 \cup K_2$, we get that
\begin{eqnarray*}
\Pro{z}{\xi_{K_1} < \xi_{K_2}} &=& \sum_{y \in \partial_i K_1} \Pro{z}{S(\xi_{K_1 \cup K_2}) = y} \\
&=& G(z; \Lambda \setminus (K_1 \cup K_2)) \sum_{y \in \partial_i K_1} \Pro{z}{S(\xi_{K_1} \wedge \xi_{K_2} \wedge \xi_{z}) = y}
\end{eqnarray*}
By reversing paths, one sees that 
$$\Pro{z}{S(\xi_{K_1} \wedge \xi_{K_2} \wedge \xi_{z}) = y} = \Pro{y}{S(\xi_{K_1} \wedge \xi_{K_2} \wedge \xi_{z}) = z}.$$ 
Thus,
\begin{eqnarray*}
& & 
\Pro{z}{\xi_{K_1} < \xi_{K_2}} \\
&=& G(z; \Lambda \setminus (K_1 \cup K_2)) \sum_{y \in \partial_i K_1} \Pro{y}{S(\xi_{K_1} \wedge \xi_{K_2} \wedge \xi_{z}) = z} \\
&=& G(z; \Lambda \setminus (K_1 \cup K_2)) \sum_{y \in \partial_i K_1} \condPro{y}{\xi_{z } < \xi_{K_2}}{\xi_{z} < \xi_{K_1}} \Pro{y}{\xi_{z} < \xi_{K_1}} 
\end{eqnarray*}
However, by reversing paths yet again, 
$$\Pro{y}{\xi_{z} < \xi_{K_1}} = \Pro{z}{S(\xi_{K_1} \wedge \xi_{z}) = y} =  \frac{\Pro{z}{S(\xi_{K_1}) = y}}{G_{\Lambda \setminus K_1}(z)},$$ 
which completes the proof of the lemma.
\end{proof}

\begin{lemma} \label{disc}

\noindent \begin{enumerate} 
\item  There exists $c > 0$ and $N$ such that for all $l \geq N$ the following holds. Suppose that $K \subset \Lambda$ contains a path connecting $0$ to $\partial B_l$. Then for any $x \in B_l$,  
$$ \Pro{x}{\xi_{K} < \sigma_{2l}} \geq c.$$  

\item There exists $c > 0$ and $N$ such that for all $N \leq 2l < n$, the following holds. Suppose that $K \subset \Lambda$ contains a path connecting $\partial B_{2l}$ to $\partial B_n$. Then for any $x \in \partial B_{2l}$,
$$ \Pro{x}{\xi_{K} \wedge \sigma_n < \xi_{l}} \geq c.$$
\end{enumerate}
\end{lemma}

\begin{proof}

\noindent \textit{Proof of (1):} We assume that $N$ is sufficiently large so that for all $l \geq N$, each of the steps below works. 

First of all, we may assume that $z \in B_{l/4}$ since if $z \in B_l$,
$$ \Pro{z}{\xi_{l/4} < \sigma_{2l}} > c.$$

If $p$ is the distribution of the random walk $S$, let 
$$m = \max \{ \abs{x} : p(x) > 0 \}.$$
Since $K$ connects $0$ to $\partial B_l$, there exists a subset $K'$ of $K$ such that for each $i = 1, \ldots, \lfloor l/m \rfloor$, there is exactly one point $x \in K'$ such that $(i-1)m \leq \abs{x} < im$. It is clear that if the lemma holds for $K'$ then it will hold for $K$. Therefore, we assume that $K$ has this property.

By \cite[Proposition 6.3.5]{LL08}, there exists a constant $C$ such that if $z \in B_l$,
$$ G_{B_{l}}(0,z) = C\left[ \log l - \log \abs{z} \right] + O(\abs{z}^{-1}).$$
Therefore, if $y,z \in B_l$ with $\abs{z - y} < l/2$, and $l$ is large enough, 
\begin{eqnarray*} 
G_{B_{2l}}(z,y) &\geq& G_{B_l}(0,y - z) \\
&\geq& C \left[ \log l - \log \abs{z - y} \right] + O(\abs{z - y}^{-1}) \\
&\geq& c_1 > 0.
\end{eqnarray*}
Similarly, if $z,y \in B_l$, 
$$ G_{B_{2l}}(z,y) \leq G_{B_{4l}}(0, y - z) \leq C \left[ \log l - \log \abs{z - y} \right] + C'.$$

Let $V$ be the number of visits to $K$ before leaving $B_{2l}$. Then for any $z \in B_{l/4}$, since there are at least $l/(4m)$ points within distance $l/2$ from $z$, 
$$ \Exp{z}{V} = \sum_{y \in K} G_{B_{2l}}(z,y) \geq \frac{c_1 l}{4m}.$$
Also, since there are at most $2j/m$ points in $K$ within distance $j$ from $z \in B_l$,
$$ \Exp{z}{V} \leq C \left[ \frac{l \log l}{m} -2 \sum_{j=1}^{l/(2m)} \log j \right]+ C'\frac{l}{m} \leq C_2 \frac{l}{m}.$$

Therefore, for any $x \in B_l$, 
$$\Pro{x}{\xi_{K} < \sigma_{2l}} = \frac{\Exp{x}{V}}{\condExp{x}{V}{\xi_{K} < \sigma_{2l}}} \geq \frac{c_1}{4C_2}.$$   
 
\noindent \textit{Proof of (2):} We again let $N$ be large enough so that if $l \geq N$ the following steps work. For $x \in \partial B_{2l}$, there exists $c > 0$ such that for all $l$ large enough, 
$$ \Pro{x}{\sigma_{4l} < \xi_l} \geq c.$$
Therefore, we may assume that $n > 4l$. We will show that if $K \subset \Lambda$ contains a path connecting $\partial B_{2l}$ to $\partial B_{4l}$, then 
$$ \Pro{x}{\xi_{K} < \xi_l} \geq c.$$

It suffices to show that for all $z, y \in B_{4l} \setminus B_{2l}$,  
\begin{eqnarray} \label{afn}
c_1 \leq G(z,y; B_l^c) \leq C_2 \left[ \log l - \log \abs{z - y} + C' \right].
\end{eqnarray}
For if we can show (\ref{afn}), then we can proceed as in the proof of (1).

To prove the left inequality, we note that for $z \in \partial B_{l/4}(y)$,
$$ G(z,y; B_l^c) \geq G_{B_{l/2}(y)}(z,y) \geq c $$
by the estimate in (1). Therefore, for any $z, y \in B_{4l} \setminus B_{2l}$, 
$$ G(z,y; B_l^c) \geq c \Pro{z}{\xi_{B_{l/4}(y)} < \xi_l},$$
and by approximation by Brownian motion, one can bound the latter from below by a uniform constant. 

We now prove the right inequality in (\ref{afn}). By the monotone convergence theorem, 
$$ G(z,y; B_l^c) = \lim_{m \to \infty} G(z,y; B_{m} \setminus B_l).$$
However, since $B_m \setminus B_l$ is a finite set, we can apply \cite[Proposition 4.6.2]{LL08} which states that
$$ G(z,y; B_{m} \setminus B_l) = \Exp{z}{a(S(\sigma_{B_m \setminus B_l}) - y)} - a(z - y),$$
where $a$ denotes the potential kernel. By \cite[Theorem 4.4.3]{LL08}, 
$$ a(z) = C^*\log \abs{z} + C' + O(\abs{z}^{-2}).$$
Therefore,
\begin{eqnarray*}
& & G(z,y; B_{m} \setminus B_l) \\
&\leq& [C^* \log 5l] \Pro{z}{\xi_l < \sigma_{m}} + [C^* \log (m+4l)]\Pro{z}{\sigma_m < \xi_l} - C^*\log \abs{z - y} + C''.
\end{eqnarray*}
However, because $\abs{z} < 4l$, a standard estimate \cite[Proposition 6.4.1]{LL08} shows that
$$ \Pro{z}{\sigma_{m} < \xi_l} \leq \frac{\log(4l) - \log l + C}{\log m - \log l} \leq \frac{C}{\log m - \log l}.$$
Therefore,
\begin{eqnarray*} 
G(z,y; B_l^c) &=& \lim_{m \to \infty} G(z,y; B_{m \setminus B_l}) \\
&\leq& \lim_{m \to \infty} C^* \log 5l + C \frac{\log (m + 4l)}{\log m - \log l} - C^* \log \abs{z - y} + C'' \\
&=& C^* \left[ \log l - \log \abs{z - y} + C'' \right].
\end{eqnarray*}  

\end{proof}

\begin{lemma} \label{green}
There exists $C < \infty$ and $N$ such that for all $N \leq 2l \leq n$, the following holds. Suppose that $K \subset \Lambda$ contains a path connecting $\partial B_{2l}$ to $\partial B_n$. Then for any $z \in B_l$,
$$ G(z ; B_n \setminus K) \leq C G(z ; B_{2l}).$$
\end{lemma}

\begin{proof}
Without loss of generality, we may assume that $K \subset \Lambda \setminus B_{2l}$. In that case, $\sigma_{2l} < \xi_{K} \wedge \sigma_n$ for all walks started in $B_l$ and therefore,
\begin{eqnarray*}
G(z ; B_n \setminus K) &=&  \Pro{z}{\xi_{K} \wedge \sigma_n < \xi_{z}}^{-1} \\
&=& \left( \sum_{w \in \partial B_{2l}} \Pro{w}{\xi_{K} \wedge \sigma_n < \xi_{z}} \Pro{z}{S(\sigma_{2l}) = w ; \sigma_{2l} < \xi_{z}} \right)^{-1}.
\end{eqnarray*}
However, by Lemma \ref{disc}, for any $w \in \partial B_{2l}$,
$$  \Pro{w}{\xi_{K} \wedge \sigma_n < \xi_{z}} \geq c > 0.$$
Therefore,
$$ G(z ; B_n \setminus K) \leq C \Pro{z}{\sigma_{2l} < \xi_z}^{-1} = C G(z; B_{2l}).$$
\end{proof}

\begin{lemma} \label{rebv}
There exists $c >0$ and $N$ such that for $N \leq 2l \leq n$ the following holds. Suppose $K \subset \Lambda \setminus B_{2l}$ contains a path connecting $\partial B_{2l}$ to $\partial B_n$. Then for $z \in B_l$,
$$ \condPro{z}{\xi_{0} < \sigma_{2l}} {\xi_{0} < \xi_{K} \wedge \sigma_n} \geq c.$$
\end{lemma}

\begin{proof}
To begin with, we claim that it suffices to show that
\begin{eqnarray} \label{adhx} 
\Pro{z}{\xi_{0} < \xi_{K} \wedge \sigma_n} \leq C \Pro{z}{\xi_{0} < \sigma_{2l}}
\end{eqnarray} for $z \in \partial B_l$ such that
$$\Pro{z}{\xi_{0} < \xi_{K} \wedge \sigma_n} = \max_{y \in \partial B_l} \Pro{y}{\xi_{0} < \xi_{K} \wedge \sigma_n}.$$
To see this, note that
$$ \condPro{z}{\xi_{0} < \sigma_{2l}}{\xi_{0} < \xi_{K} \wedge \sigma_n} = \frac{\Pro{z}{\xi_{0} < \sigma_{2l}}}{\Pro{z}{\xi_{0} < \xi_{K} \wedge \sigma_n}}.$$
Therefore it suffices to show that for all $z \in B_l$,
$$ \Pro{z}{\xi_{0} < \xi_{K} \wedge \sigma_n} \leq C \Pro{z}{\xi_{0} < \sigma_{2l}}.$$
However, for $z \in B_l$,
$$ \Pro{z}{\xi_{0} < \xi_{K} \wedge \sigma_n} = \Pro{z}{\xi_0 < \sigma_l} + \sum_{w \in \partial B_l} \Pro{w}{\xi_{0} < \xi_{K} \wedge \sigma_n}\Pro{z}{S(\xi_0 \wedge \sigma_l) = w} $$
and 
$$ \Pro{z}{\xi_{0} < \sigma_{2l}} = \Pro{z}{\xi_0 < \sigma_l} + \sum_{w \in \partial B_l} \Pro{w}{\xi_{0} < \sigma_{2l}}\Pro{z}{S(\xi_0 \wedge \sigma_l) = w}.$$
Furthermore, by the discrete Harnack inequality, for any $y, y' \in \partial B_l$,
$$ \Pro{y}{\xi_{0} < \xi_{K} \wedge \sigma_n} \asymp \Pro{y'}{\xi_{0} < \xi_{K} \wedge \sigma_n}$$
and
$$  \Pro{y}{\xi_{0} < \sigma_{2l}} \asymp  \Pro{y'}{\xi_{0} < \sigma_{2l}}.$$
Therefore, the lemma will follow once we prove (\ref{adhx}).

Let $z \in \partial B_l$ be such that
$$\Pro{z}{\xi_{0} < \xi_{K} \wedge \sigma_n} = \max_{y \in \partial B_l} \Pro{y}{\xi_{0} < \xi_{K} \wedge \sigma_n}.$$
Then,
\begin{eqnarray*}
\Pro{z}{\xi_{0} < \xi_{K} \wedge \sigma_n} &=& \Pro{z}{\xi_{0} < \sigma_{2l}} + \Pro{z}{\sigma_{2l} < \xi_{0}; \xi_{0} < \xi_{K} \wedge \sigma_n}.
\end{eqnarray*}
Now,
\begin{eqnarray*}
& &
\Pro{z}{\sigma_{2l} < \xi_{0}; \xi_{0} < \xi_{K} \wedge \sigma_n} \\
&=& \sum_{w \in \partial B_{2l}} \Pro{w}{\xi_{0} < \xi_{K} \wedge \sigma_n} \Pro{z}{S(\sigma_{2l}) = w;  \sigma_{2l} < \xi_{0}}.
\end{eqnarray*}
By Lemma \ref{disc}, for any $w \in \partial B_{2l}$,
\begin{eqnarray*}
\Pro{w}{\xi_{0} < \xi_{K} \wedge \sigma_n} &=& \sum_{y \in \partial B_l}\Pro{w}{S(\xi_{l}) = y; \xi_{l} < \xi_{K} \wedge \sigma_n}  \Pro{y}{\xi_{0} < \xi_{K} \wedge \sigma_n} \\
&\leq& \Pro{w}{\xi_{l} < \xi_{K} \wedge \sigma_n}  \Pro{z}{\xi_{0} < \xi_{K} \wedge \sigma_n} \\
&\leq& (1-c) \Pro{z}{\xi_{0} < \xi_{K} \wedge \sigma_n}.
\end{eqnarray*}
Thus,
\begin{eqnarray*}
\Pro{z}{\xi_{0} < \xi_{K} \wedge \sigma_n} &\leq& \Pro{z}{\xi_{0} < \sigma_{2l}} + (1-c) \Pro{z}{\xi_{0} < \xi_{K} \wedge \sigma_n} \Pro{z}{\sigma_{2l} < \xi_{0}} \\
&\leq& \Pro{z}{\xi_{0} < \sigma_{2l}} + (1-c) \Pro{z}{\xi_{0} < \xi_{K} \wedge \sigma_n},
\end{eqnarray*}
and therefore
$$ \Pro{z}{\xi_{0} < \xi_{K} \wedge \sigma_n} \leq \frac{1}{c} \Pro{z}{\xi_{0} < \sigma_{2l}},$$
which completes the proof.
\end{proof}

\subsection{Random walks conditioned to avoid certain sets} \label{srwavoid}

\begin{prop} \label{dirichlet}
There exist constants $N$ and $c > 0$ such that for all $n \geq N$ the following holds. Suppose that $K \subset \Lambda \setminus B_n(n,0)$ where $B_n(n,0)$ denotes the ball of radius $n$ centered at $(n,0)$ (see Figure \ref{figdirichlet}). Then,
$$ \condPro{0}{\arg(S(\sigma_n)) \in [-\frac{\pi}{4}, \frac{\pi}{4} ]}{\sigma_n < \xi_{K}} \geq c.$$
\end{prop}

\begin{figure}[htp]
\centerline{
\epsfxsize=3in
\epsfysize=3in
\epsfbox{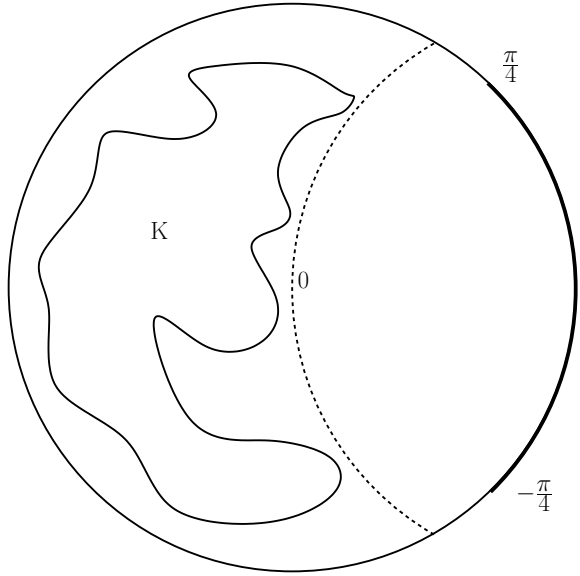}
}
\caption{The setup for Proposition \ref{dirichlet}}
\label{figdirichlet}
\end{figure}

\begin{proof}
For $z \in \Disc$, let
$$ h(z) = \Pro{z}{\arg [W(\tau_\Disc)] \in [-\frac{\pi}{4}, \frac{\pi}{4}]}$$
where $W$ denotes standard two-dimensional Brownian motion. Then $h$ is the solution to the Dirichlet problem with boundary value $\ind_{[-\pi/4, \pi/4]}$. Therefore, we can express $h$ as
$$ h(z) = \frac{1}{2\pi} \int_{-\frac{\pi}{4}}^{\frac{\pi}{4}} H_\Disc(z,e^{i\theta}) \,d\theta$$
where
$$ H_\Disc(z,e^{i\theta}) = \frac{1 - \abs{z}^2}{\abs{e^{i\theta} - z}^2}$$
is the Poisson kernel for the unit disk.

One can compute $h$ (it is easier to consider the problem on $\Halfplane$ and then map back via a conformal transformation):
\begin{eqnarray*} 
h(z) &=& \frac{1}{\pi}\left[ \arctan \left( \frac{(\sqrt{2} - 1)\abs{1 + z}^2 + 2\Im(z)}{1 - \abs{z}^2} \right) \right. \\
&+& \left. \arctan \left(\frac{(\sqrt{2} - 1)\abs{1 + z}^2 - 2 \Im(z)}{1 - \abs{z}^2} \right) \right].
\end{eqnarray*}

We now establish three basic facts about $h$ that we will use below.
\begin{enumerate}
\item Let $D_1(1)$ be the disk of radius $1$ centered at the point $1$. We claim that for all $z \in \Disc \setminus D_1(1)$, $h(z) \leq h(0)$. By the maximal principle for harmonic functions, $h$ restricted to $\overline{\Disc} \setminus D_1(1)$ takes its maximal value on $\partial D_1(1) \cap \Disc$ (since $h$ vanishes on $\partial \Disc \setminus D_1(1)$). Thus, to prove the claim, it suffices to show that 
$$f(t) = h(1 + \cos t, \sin t) $$
takes its maximal value at $t = \pi$ for $2\pi/3 \leq t \leq 4\pi/3$. Since one has an explicit formula for $h$, this is left as an exercise for the reader or the reader's Calculus students.

\item Next, fixing $w = e^{i \theta}$, it is a basic calculation to show that
$$ \frac{\partial H_\Disc(.,w)}{\partial x}(0) = 2 \cos \theta, \hskip25pt \frac{\partial H_\Disc(.,w)}{\partial y}(0) = 2 \sin \theta.$$  
Therefore,
$$ \frac{\partial h}{\partial x}(0) = \frac{1}{2\pi} \int_{-\frac{\pi}{4}}^{\frac{\pi}{4}} (2 \cos \theta) \,d\theta = \frac{\sqrt{2}}{\pi},$$
and
$$ \frac{\partial h}{\partial y}(0) = \frac{1}{2\pi} \int_{-\frac{\pi}{4}}^{\frac{\pi}{4}} (2 \sin \theta) \,d\theta = 0.$$
These results can also be obtained from the explicit formula for $h$.

\item Finally, there exists $0 < r < 1$ such that for all $r < \abs{z} < 1$, and $\abs{\arg(z)} > \pi/3$, $h(z) < 1/8$. This follows from the fact that if $\abs{z} > r$ and $|\arg(z)| > \pi/3$ then for $w \in \partial \Disc$ with $|\arg(w)| < \pi/4$,
$$ H_\Disc(z,w) = \frac{1 - \abs{z}^2}{\abs{w - z}^2} \leq c(1 - r^2),$$
which can be made to be arbitrarily close to $0$.
\end{enumerate}

\vskip15pt

Assume that $n$ is large enough so that $\conj{B_{rn}} \subset n \Disc$ where $r$ is as in the previous paragraph. We let $h_n(z) = h(z/n)$ which is harmonic in $n \Disc$. Then for $z \in \overline{B_{rn}}$, define
$$ \wt{h}_n(z) = \Exp{z}{ h_n(S(\overline{\sigma}_{rn}))}.$$
Then $\wt{h}_n$ is discrete harmonic in $ B_{rn}$ and agrees with $h_n$ on $\partial B_{rn}$.  

A natural question to ask is how close does the discrete harmonic solution $\wt{h}_n$ approximate the continuous harmonic solution $h_n$? By \cite[Corollary 6.2.4]{LL08}, for all $z \in B_{rn}$, 
$$h_n(z) = \wt{h}_n(z) - \Exp{z}{\sum_{j=0}^{\overline{\sigma}_{rn} - 1} \mathcal{L}h_n(S_j)},$$
where $\mathcal{L}f(x) = -f(x) + \sum_{y \in \Lambda} p(y)f(x+y)$ is the generator for $S$. Consider the associated operator 
$$\wt{\mathcal{L}}f(x) = \frac{1}{2} \sum_{y \in \Lambda} \kappa(y) \frac{\partial^2 f}{\partial y^2}(x).$$
By Taylor's theorem, for any $C^4$ function $f$ and $z \in \Lambda$,
$$ \abs{\mathcal{L}f(z) - \wt{\mathcal{L}}f(z)} \leq C R^4 M_4(f)(z),$$
where $R$ is the range of the walk $S$ and $M_4(f)$ is the $L^\infty$ norm of the sum of the fourth derivatives of $f$ in the disk $D_R(z)$.

Since the random walk $S$ has covariance matrix the identity (we have been assuming that $S$ has this property but this is the first place we use it), one can show that $\wt{\mathcal{L}}$ is actually a multiple of the continuous Laplacian. Thus, $\wt{\mathcal{L}} h_n = 0$. Furthermore, since the fourth derivatives of $h$ are bounded on $r\Disc$, $M_4(h_n)$ is bounded by $Cn^{-4}$ in $B_{rn}$. Therefore, combining all the previous remarks (and letting $CR^4 = C$ since $R$ depends only on the random walk $S$ which we've fixed), we obtain that for $z \in B_{rn}$, 
\begin{eqnarray*}
\abs{h_n(z) - \wt{h}_n(z)} &\leq& \Exp{z}{\sum_{j=0}^{\overline{\sigma}_{rn} - 1} C R^4 M_4(h_n)(z)} \\
&\leq& Cn^{-4} \Exp{z}{\overline{\sigma}_{rn} - 1} \\
&=& Cn^{-4} \sum_{x \in B_n} G_n(z,x) \\
&\leq& Cn^{-4} \sum_{x \in B_{2n}} G_{2n}(0,x) \\
&\leq& Cn^{-4} \sum_{x \in B_{2n}} \left[\log 2n - \log\abs{x} + C'\right] \\&\leq& C n^{-2}.
\end{eqnarray*}

We now have all the pieces we need to prove the proposition. Let $z$ be any point in $B_{rn} \setminus B_n(n,0)$, and fix $x \in \Lambda$ such that $\Re(x) > 0$. Then by Taylor's theorem and our previous observations about $h$, if $n$ is large enough so that $x$ is in $B_{rn}$,
\begin{eqnarray*}
& &
\wt{h}_n(x) - \wt{h}_n(z) \\
&=& [\wt{h}_n(x) - h_n(x)] + [h_n(x) - h_n(0)] + [h_n(0) - h_n(z)] + [h_n(z) - \wt{h}_n(z)] \\
&\geq& - Cn^{-2} + \left[n^{-1}\frac{\partial h}{\partial x}(0) \Re(x) - n^{-2} M_2(h) \abs{x}^2 \right] + 0 - Cn^{-2} 
\end{eqnarray*}
where $M_2(h)$ is the $L^\infty$ norm of the sum of the second order derivatives of $h$ in $r\Disc$.
Since $\frac{\partial h}{\partial x}(0) > 0$, it is clear that for $n$ sufficiently large,
$$  \wt{h}_n(x) \geq \wt{h}_n(z),$$
for all $z \in B_{rn} \setminus B_n(n,0)$. 

Since $K \subset \Lambda \setminus B_n(n,0)$,
\begin{eqnarray*}
\condExp{x}{h_n(S(\sigma_{rn}))}{\xi_{K} < \sigma_{rn}} &\leq& \max_{z \in K \cap B_{rn}} \Exp{z}{h_n(S(\sigma_{rn}))} \\
&=& \max_{z \in K \cap B_{rn}} \wt{h}_n(z) \\
&\leq& \wt{h}_n(x) \\
&=& \Exp{x}{h_n(S(\sigma_{rn}))}.
\end{eqnarray*}
Thus,
\begin{eqnarray*}
\Exp{x}{h_n(S(\sigma_{rn}))} 
&=& \condExp{x}{h_n(S(\sigma_{rn}))}{\xi_{K} < \sigma_{rn}} \Pro{x}{\xi_{K} < \sigma_{rn}} \\
&+& \condExp{x}{h_n(S(\sigma_{rn}))}{\sigma_{rn} < \xi_{K}} \Pro{x}{\sigma_{rn} < \xi_{K}} \\
&\leq& \Exp{x}{h_n(S(\sigma_{rn}))} \Pro{x}{\xi_{K} < \sigma_{rn}} \\
&+& \condExp{x}{h_n(S(\sigma_{rn}))}{\sigma_{rn} < \xi_{K}} \Pro{x}{\sigma_{rn} < \xi_{K}}.
\end{eqnarray*}
This implies that
\begin{eqnarray*} 
\condExp{x}{h_n(S(\sigma_{rn}))}{\sigma_{rn} < \xi_{K}} &\geq& \Exp{x}{h_n(S(\sigma_{rn}))} \\
&=& \wt{h}_n(x) \\
&\geq& \wt{h}_n(0) \\
&\geq& h(0) - Cn^{-2} \geq \frac{1}{5}
\end{eqnarray*}
for $n$ sufficiently large, since $h(0) = 1/4$.

Recall that $r$ was defined so that for all $z$ such that $r < \abs{z} < 1$, and $\abs{\arg(z)} > \pi/3$, $h(z) < 1/8$. Therefore,
\begin{eqnarray*} 
\frac{1}{5} &\leq& \condExp{x}{h_n(S(\sigma_{rn}))}{\sigma_{rn} < \xi_{K}} \\
&=& \sum_{|\arg(z)| > \pi/3} h_n(z) \condPro{x}{S(\sigma_{rn}) = z}{\sigma_{rn} < \xi_{K} } \\
&+& \sum_{|\arg(z)| \leq \pi/3} h_n(z) \condPro{x}{S(\sigma_{rn}) = z}{\sigma_{rn} < \xi_{K}} \\
&\leq& \frac{1}{8} \left( 1 - \condPro{x}{\abs{\arg S(\sigma_{rn})} \leq \frac{\pi}{3}}{\sigma_{rn} < \xi_{K} } \right) \\
&+& \condPro{x}{\abs{\arg S(\sigma_{rn})} \leq \frac{\pi}{3}}{\sigma_{rn} < \xi_{K}}.
\end{eqnarray*}
Thus,
$$ \condPro{x}{\abs{\arg S(\sigma_{rn})} \leq \frac{\pi}{3} }{ \sigma_{rn} < \xi_{K} } \geq c > 0.$$

Since $x$ is independent of $K$ and $n$,
$$ \Pro{0}{\xi_{x} < \xi_{K} \wedge \sigma_{rn}} \geq c,$$
and hence,
$$ \condPro{0}{\abs{\arg S(\sigma_{rn})} \leq \frac{\pi}{3}}{\sigma_{rn} < \xi_{K} } \geq c > 0.$$

Finally,
\begin{eqnarray*}
\condPro{0}{\abs{\arg S(\sigma_{n})} \leq \frac{\pi}{4}}{\sigma_{n} < \xi_{K} } &\geq& \condPro{0}{\abs{\arg S(\sigma_{n})} \leq \frac{\pi}{4}; \sigma_n < \xi_K}{ \sigma_{rn} < \sigma_K} \\
&\geq& c \condPro{0}{\abs{\arg S(\sigma_{rn})} \leq \frac{\pi}{3}}{\sigma_{rn} < \xi_{K} } \geq c.
\end{eqnarray*}

\end{proof}

\begin{lemma} \label{wedge}
For $0 < \theta < \pi$, there exist $c(\theta)$, $N(\theta)$ and $\alpha(\theta)$ such that the following holds. For $n > N$, and $z \in \Lambda$ with $N < \abs{z} < n$, let $W$ be the wedge
$$ W = \{ x \in \Lambda :  0 \leq \abs{x} \leq n, \abs{\arg(x) - \arg(z)} \leq \theta \}.$$
Then,
$$\Pro{z}{\sigma_W = \sigma_n} \geq c \left( \frac{\abs{z}}{n} \right)^{\alpha}.$$
\end{lemma}

\begin{remark}
By comparison with Brownian motion, one expects that $\alpha(\theta) = \pi/\theta$ would be the optimal constant. However, in this paper we will only need the existence of $\alpha$ and not its exact value.
\end{remark}

\begin{proof}
It is clear that we can make $\alpha(\theta)$ non-increasing in $\theta$, therefore, without loss of generality, take $\theta < \pi/2$. Also, without loss of generality, assume $\arg(z) = 0$. 

Let $\wt{W}$ be the cone
$$ \wt{W} = \{  x \in \Lambda : \abs{\arg(x)} \leq \theta \}.$$
 
We define a random sequence of points $\{ z_k \} \subset \wt{W}$ as follows. We let $z_0 = z$. Then, given $z_k$, we let $B_k$ be the largest ball centered at $z_k$ such that $\overline{B_k} \subset \wt{W}$, $r_k$ be the radius of $B_k$ and let $z_{k+1} = S(\sigma_{B_k})$ where $S$ is a random walk starting at $z_k$.

We note that $z_j = z_k$ for all $j \geq k$ if and only if $z_k \in \partial_i \wt{W}$. We make $N(\theta)$ large enough to ensure that if $\abs{z} > N$ then $z \notin \partial_i \wt{W}$. In this case, there exists $c'(\theta) > 0$ such that $r_0 \geq c'(\theta) \abs{z}$.  

Let $E_k$ denote the event that $z_{k+1} \neq z_k$ and that
$$ \abs{ \arg (z_{k+1} - z_k)} \leq \frac{\theta}{2}.$$ 
On the event $E_k$,
$$r_{k+1} \geq (1 + 2 \sin(\theta/4))r_k = \wt{c}(\theta)r_k,$$
and
$$\abs{z_{k+1}}^2 \geq \abs{z_k}^2 + r_k^2$$
(we use the fact that $\theta < \pi/2$ for the second assertion). Therefore, if $E_0, \ldots, E_j$ all hold, then 
$$ r_k \geq \wt{c}^k r_0 \geq \wt{c}^k c' \abs{z}$$ 
for $k = 1, \ldots, j$. Therefore,
\begin{eqnarray*} 
\abs{z_{j+1}}^2 &\geq& \abs{z}^2 + \sum_{k=0}^j r_k^2 \\
&\geq& \abs{z}^2 + \sum_{k=0}^j (c')^2 \wt{c}^{2k}  \abs{z}^2 \\
&\geq& (c')^2\wt{c}^{2j} \abs{z}^2.
\end{eqnarray*}
Since $\wt{c} > 1$, it follows that if we let $j$ be the smallest integer such that 
$$ j \geq \frac{\log(n/c'\abs{z})}{\log(\wt{c})},$$ then
$$ \bigcap_{k=0}^j E_j \subset \{ \sigma_W = \sigma_n \}.$$
Finally, by the invariance principle, there exists a constant $c''(\theta)$ and $N$ such that for $n \geq N$, $\Prob{}{E_{k}} \geq c''$ for all $k$. Therefore,
$$ \Pro{z}{\sigma_W =\sigma_n} \geq \Prob{}{\bigcap_{k=0}^{j} E_k} = \prod_{k=0}^{j} \Prob{}{E_k} \geq (c'')^{j+1} \geq c(\theta) \left( \frac{\abs{z}}{n} \right)^{\alpha(\theta)},$$
where $\alpha = - \log(c'')/\log(\wt{c})$ and $c = c''(c')^\alpha$.
\end{proof}

\begin{cor} \label{dfr}
Fix $\theta_1, \theta_2 \in (0, \pi/2)$. There exist $N$, $\alpha$ and $c>0$ depending only on $\theta_1 + \theta_2$ such that the following holds. Let $N \leq l < m < n$,  and $z \in \partial B_m$. Let $W$ be the half-wedge
$$ W = \{ x \in \Lambda: l  \leq \abs{x} \leq n, - \theta_1 \leq \arg(x) - \arg(z) \leq \theta_2 \}.$$ 
\begin{enumerate}
\item Let $r = \min \{m \sin \theta_1, m \sin \theta_2, m - l \}$. Then for any $K \subset B_m$,
$$\condPro{z}{S[0,\sigma_n] \subset W }{ \sigma_n < \xi_{K}} \geq c \left( \frac{r}{n} \right)^{\alpha}.$$
\item  Let $r' = \min \{m \sin \theta_1, m \sin \theta_2, n - m \}$. There exists $\beta = \beta(\theta_1 + \theta_2, l/m)$ such that for any $K \subset \Lambda \setminus B_m$,
$$\condPro{z}{S[0,\xi_l] \subset W }{ \xi_l < \xi_{K}} \geq c \left( \frac{r'}{m} \right)^{\beta}.$$
\end{enumerate}
Notice that in both cases, the right hand side depends only on $\theta_1$, $\theta_2$, and the ratios $l/n$ and $m/n$.
\end{cor}

\begin{figure}[htp]
\centerline{
\epsfxsize=4in
\epsfysize=3in
\epsfbox{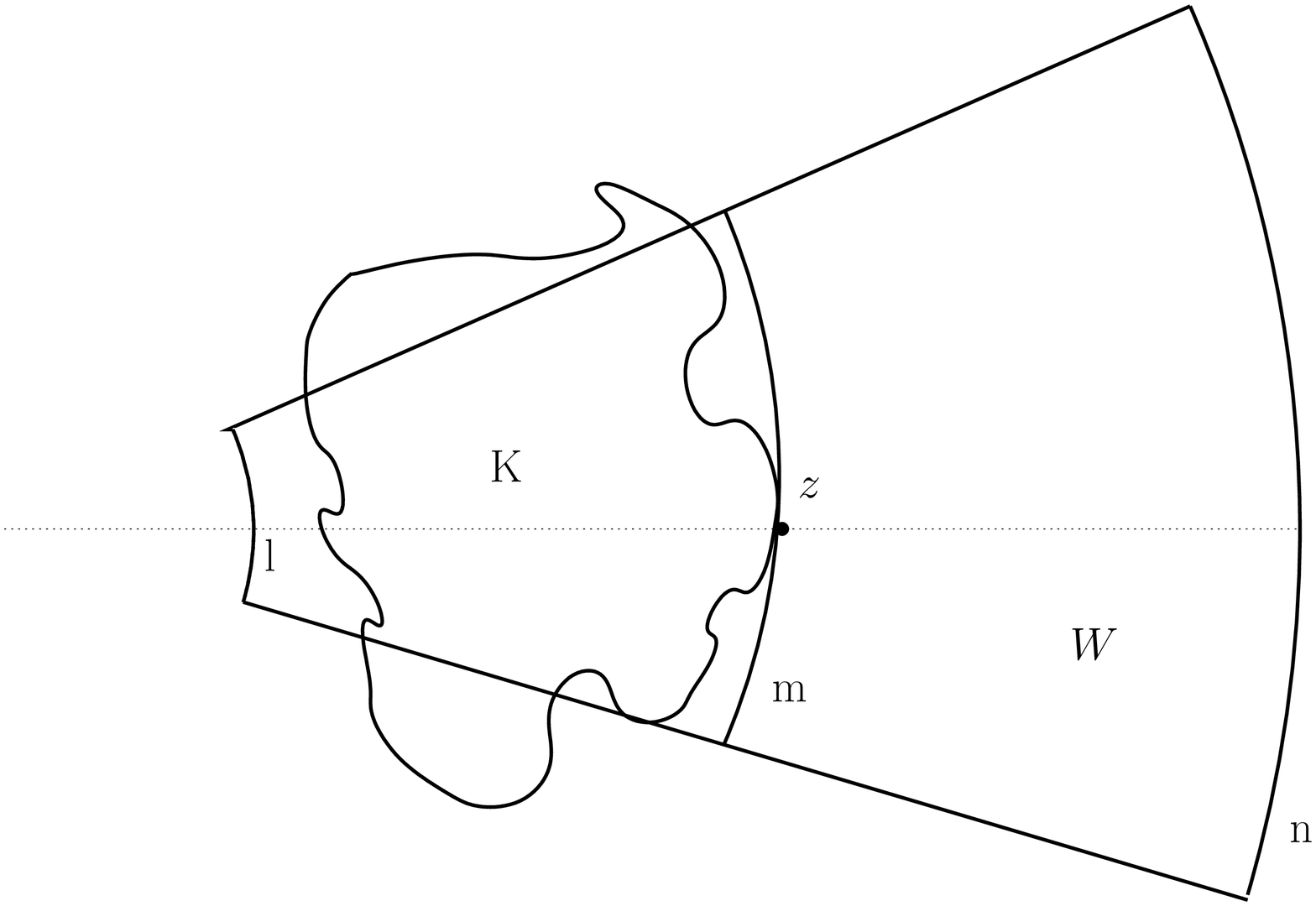}
}
\caption{The setup for part 1 of Corollary \ref{dfr}}
\end{figure}

\begin{proof}
Both parts of the corollary are proved similarly. We prove 1 in detail, and indicate the modifications needed to prove 2.

Without loss of generality, assume that $\arg(z) = 0$. The quantity $r$ defined in the statement of the corollary is the radius of the largest ball with center $z$ whose closure is contained in the half-infinite wedge 
$$ \wh{W} = \{ x \in \Lambda: \abs{x} \geq l, - \theta_1 \leq \arg(x) \leq \theta_2 \}.$$ We can apply Proposition \ref{dirichlet} to the ball $B = B(z,r)$, to obtain that there exists a constant $c > 0$ such that
$$ \condPro{z}{\abs{\arg(S(\sigma_B) - z)} \leq \frac{\pi}{4}}{ \sigma_{B} < \xi_K} \geq c.$$

Let $y$ be any point on $\partial B$ such that $\abs{\arg(y -z)} \leq \pi/4$, and let $\wt{B} = B(y, r/2)$. Note that $\wt{B} \subset \wh{W} \setminus B_m$. There exists a point $w \in \partial \wt{B}$ such that $y$ is on the bisector of the angle formed from the lines joining $w$ to the two outermost "corners" of $W$, $x_1 = n e^{-i\theta_1}$ and $x_2 = n e^{i \theta_2}$. Let
$$s = \max \{\abs{x_1 - w}, \abs{x_2 - w}\} \leq 2n,$$ and let $\wt{W}$ be the wedge with vertex $w$, radius $s$ and such that $x_1$ and $x_2$ are on $\partial \wt{W}$. The wedge $\wt{W}$ will have aperture $\theta \geq (\theta_1 + \theta_2)/2$ and $y$ will be on the axis of symmetry of $\wt{W}$. Therefore, by Lemma \ref{wedge},
\begin{eqnarray*}
\Pro{y}{S[0, \sigma_n] \subset W; \sigma_n < \xi_K} &\geq& \Pro{y}{\sigma_{\wt{W}} = \sigma_{B(w,s)}} \\
&\geq& c(\theta) \left( \frac{r}{s} \right)^{\alpha(\theta)} \\
&\geq& c(\theta_1 + \theta_2) \left( \frac{r}{n} \right)^{\alpha(\theta_1 + \theta_2)}.
\end{eqnarray*}

To finish the proof of 1, let
$$ c^* = c \left( \frac{r}{n} \right)^{\alpha}.$$
Then,
\begin{eqnarray*}
& & \Pro{z}{S[0,\sigma_n] \subset W; \sigma_n < \xi_{K}} \\
&\geq& \sum_{\substack{ y \in \partial B(z,R) \\ \abs{\arg(y - z)} \leq \pi/4}} \Pro{y}{S[0, \sigma_n] \subset W; \sigma_n < \xi_K} \Pro{z}{\sigma_B < \xi_K; S(\sigma_B) = y} \\
&\geq& c^* \Pro{z}{\abs{\arg(S(\sigma_B) - z)} \leq \frac{\pi}{4}; \sigma_{B} < \xi_K} \\
&\geq& c^* \Pro{z}{\sigma_B < \xi_K} \\
&\geq& c^* \Pro{z}{\sigma_n < \xi_K}.
\end{eqnarray*}

The proof of 2 is similar. In this case, the angle $\theta$ of the wedge $\wt{W}$ will be such that
$$ \theta \geq \tan^{-1}\left( \frac{2l \sin((\theta_1 + \theta_2)/2)}{m}\right).$$
This is why $\beta$ will also depend on $l/m$. Besides this observation, the proof of 2 is identical to the proof of 1.

\end{proof}

The following corollary is similar to the previous one, except that here we are conditioning to avoid sets that are on either side of the half-wedge. 

\begin{cor} \label{wedgecond} Let $\theta \in (0, 2\pi]$ and $0 < a < 1 < b < \infty$. Then there exists $N$ and $c > 0$ depending on $a$, $b$ and $\theta$ such that for $n > N$ the following holds. Let $W$ be the half-wedge
$$ W = \{ x : an  \leq \abs{x} \leq 4bn, \abs{\arg(x)} \leq \theta \}.$$  Suppose that $K_1 \subset B_n$ contains a path connecting $\partial B_{an}$ to $\partial_i B_n$, and $K_2 \subset \Lambda \setminus B_{4n}$ contains a path connecting $\partial B_{4n}$ to $\partial B_{4bn}$. Let $K = K_1 \cup K_2$. Then for any $z \in \partial B_n$, $y \in \partial B_{4n}$ with $\abs{\arg(z)} < \theta/2$, $\abs{\arg(y)} < \theta/2$,
$$ \condPro{z}{S[0,\xi_{y}] \subset W }{ \xi_{y} < \xi_K} \geq c.$$
\end{cor}

\begin{figure}[htp]
\centerline{
\epsfxsize=4.2in
\epsfysize=3in
\epsfbox{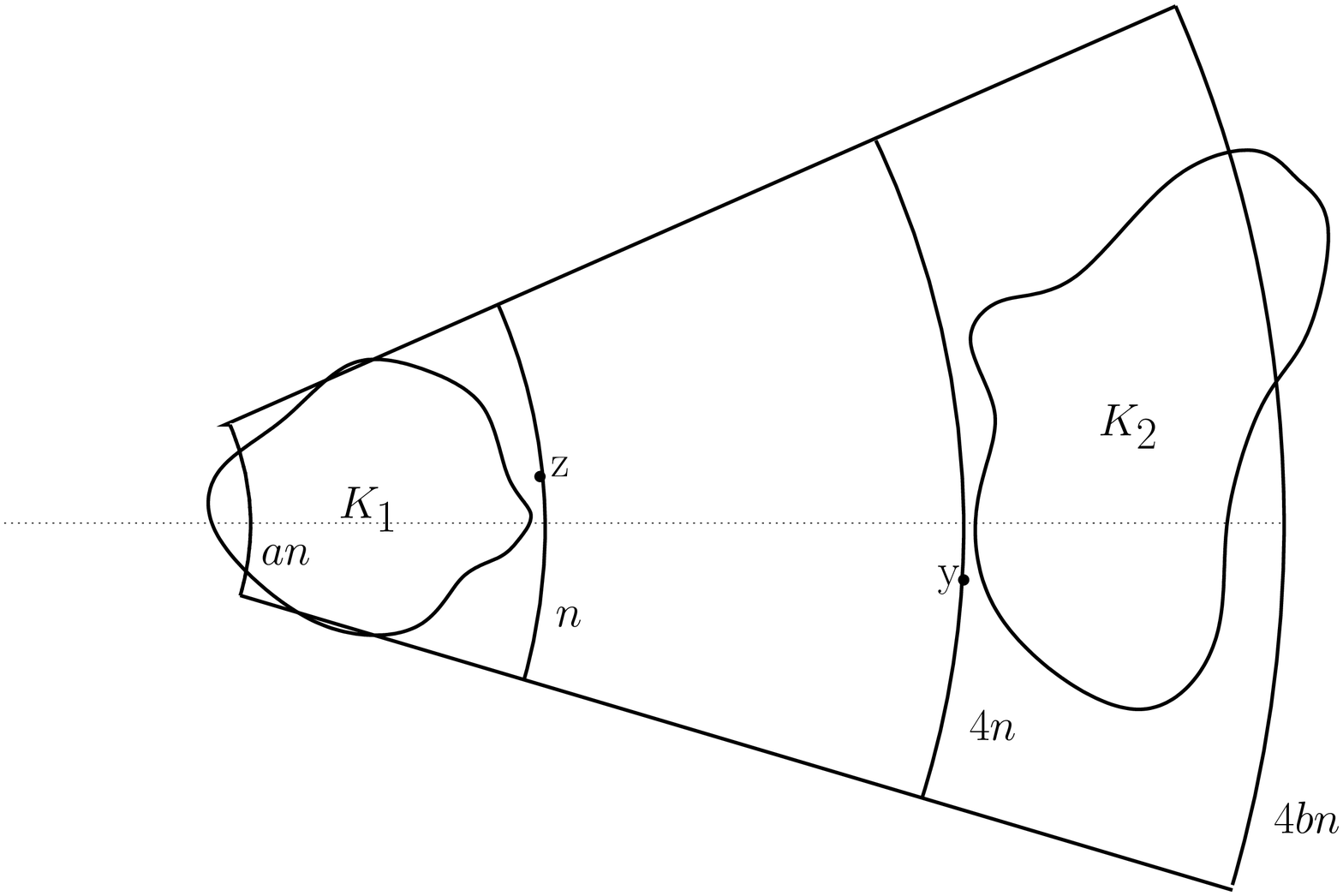}
}
\caption{The setup for Corollary \ref{wedgecond}}
\end{figure}

\begin{proof}
Applying part 1 of Corollary \ref{dfr} to the half-wedge
$$ W^* = \{ x : an  \leq \abs{x} \leq 2n, \abs{\arg(x) - \arg(z)} \leq \frac{\theta}{4} \},$$
one obtains that there exists a constant $c = c(a, \theta)$ such that
$$ \condPro{z}{S[0,\sigma_{2n}] \subset W^*}{ \sigma_{2n} < \xi_K} \geq c.$$
Now suppose that $w \in \partial B_{2n} \cap W^*$. Then by Lemma \ref{rwdecomp}, 
\begin{eqnarray*}
& &
\condPro{w}{S[0,\xi_{y}] \subset W }{ \xi_{y} < \xi_K} \\
&=& \frac{\Pro{w}{\xi_{y} < \sigma_W \wedge \xi_K}}{\Pro{w}{\xi_{y} < \xi_K}} \\
&=& \frac{G(w ; W \setminus (\{y\} \cup K))\Pro{y}{\xi_{w} < \sigma_W \wedge \xi_K \wedge \xi_y}}{G(w; \Lambda \setminus (\{y\} \cup K)) \Pro{y}{\xi_{w} < \xi_{K} \wedge \xi_y}}.
\end{eqnarray*}
However, for $w \in \partial B_{2n} \cap W^*$, 
$$ G(w ; W \setminus (\{y\} \cup K)) \geq G(w; B_{c(\theta)n}(w)),$$ and therefore by Lemma \ref{green},
$$ \frac{G(w ; W \setminus (\{y\} \cup K))}{G(w; \Lambda \setminus (\{y\} \cup K))} \geq c(\theta, b).$$
Thus,
$$ \condPro{w}{S[0,\xi_{y}] \subset W}{ \xi_{y} < \sigma_K} \geq c(\theta,b) \frac{\Pro{y}{\xi_{w} < \sigma_W \wedge \xi_K \wedge \xi_y}}{\Pro{y}{\xi_{w} < \xi_{K} \wedge \xi_y}}.$$

By the strong Markov property,
\begin{eqnarray*}
& & \Pro{y}{\xi_{w} < \sigma_W \wedge \xi_K \wedge \xi_y} \\
&\geq& \sum_{x \in \partial_i B_{3n} \cap W^*} \Pro{x}{\xi_w < \sigma_W \wedge \xi_K \wedge \xi_y} \Pro{y}{S(\xi_{3n}) = x ; \xi_{3n} < \sigma_{W^*} \wedge \xi_K \wedge \xi_y}.
\end{eqnarray*}
However, by Lemma \ref{rebv}, there exists $c(\theta,b)$ such that for $x \in \partial_i B_{3n} \cap W^*$,
$$ \Pro{x}{\xi_w < \sigma_W \wedge \xi_K \wedge \xi_y} \geq c(\theta,b) \Pro{x}{\xi_w < \xi_K \wedge \xi_y}.$$
Furthermore, by the discrete Harnack inequality, there exists $c > 0$ such that for all $x, x' \in \partial_i B_{3n}$, 
$$ \Pro{x}{\xi_w < \xi_K \wedge \xi_y} \geq c \Pro{x'}{\xi_w < \xi_K \wedge \xi_y}.$$
Thus, fixing $x' \in \partial_i B_{3n}$,
$$ \Pro{y}{\xi_{w} < \sigma_W \wedge \xi_K \wedge \xi_y} \geq c(\theta,b) \Pro{x'}{\xi_w < \xi_K \wedge \xi_y} \Pro{y}{\xi_{3n} < \sigma_{W^*} \wedge \xi_K \wedge \xi_y}.$$
Similarly,
$$ \Pro{y}{\xi_{w} < \xi_K \wedge \xi_y} \leq c \Pro{x'}{\xi_w < \xi_K \wedge \xi_y} \Pro{y}{\xi_{3n} < \xi_K \wedge \xi_y}.$$
Therefore using part 2 of Corollary \ref{dfr},
\begin{eqnarray*} 
\condPro{w}{S[0,\xi_{y}] \subset W}{\xi_{y} < \sigma_K} &\geq& c(\theta,b) \frac{\Pro{y}{\xi_{w} < \sigma_W \wedge \xi_K \wedge \xi_y}}{\Pro{y}{\xi_{w} < \xi_{K} \wedge \xi_y}}\\
&\geq& c(\theta,b) \frac{\Pro{y}{\xi_{3n} < \sigma_{W^*} \wedge \xi_K \wedge \xi_y}}{\Pro{y}{\xi_{3n} < \xi_K \wedge \xi_y}} \\
&=& c(\theta,b) \condPro{y}{\xi_{3n} < \sigma_{W^*} }{ \xi_{3n} < \xi_K \wedge \xi_y} \\
&\geq& c(\theta, b).
\end{eqnarray*}

\end{proof}

\subsection{Random walk approximations to hitting probabilities of curves by Brownian motion} \label{rwapproxsec}

Given a random walk $S$ on a discrete lattice $\Lambda$, we can make $S$ into a continuous curve $S_t$ by linear interpolation and define $S_t^{(n)} = n^{-1} S_{n^2t}$. Now fix a continuous curve $\alpha:[0,t_\alpha] \to \overline{\Disc}$. In this section, we will compare the probability that a Brownian motion $W_t$ started at the origin leaves the unit disk before hitting $\alpha$ to the probability that $S^{(n)}_t$ started at the origin leaves the unit disk before hitting $\alpha$. By the invariance principle, one can show that as $n$ tends to infinity, the latter probability approaches the former. What is more difficult is to show that this convergence is uniform in $\alpha$ as long as the diameter of $\alpha$ is sufficiently large. This is Proposition \ref{a} and is the main result of the section.

For $0 < \delta < 1$, let $A_\delta$ denote the annulus
$$ A_\delta = \{ z \in \Complex : 1 - \delta < \abs{z} < 1 \}.$$
Given a curve $\alpha:[0,t_\alpha] \to \overline{\Disc}$, let
$$ C_\delta(\alpha) =  \{ e^{i\theta} : \theta = \arg(z) \text{ for some $z \in \alpha[0,t_\alpha] \cap A_\delta$} \},$$
and
$$ \wt{C}_\delta(\alpha) = \{ z \in \partial \Disc : \operatorname{dist}(z,C_\delta) < \delta \}.$$
Recall that $D_\delta(z)$ is the disk of radius $\delta$ centered at $z$.

\begin{lemma} \label{jordan}
For all $0 < \delta < 1$ and all simple continuous curves $\alpha: [0,t_\alpha] \to \Disc$, there exist two points $z_1, z_2 \in \partial \Disc$ such that $D_\delta(z_1) \cup \alpha[0, t_\alpha] \cup D_\delta(z_2)$ disconnects $0$ from $\wt{C}_\delta(\alpha)$ in $\Disc$.
\end{lemma}

\begin{proof}
If $0 \in \alpha[0, t_\alpha]$, then the result is immediate, so we will assume that $0 \notin \alpha[0, t_\alpha]$. In that case, $\Arg(\alpha(t))$, the continuous argument of $\alpha(t)$ is well defined for all $t \in [0, t_\alpha]$, where we let $\Arg(\alpha(0)) = \arg(\alpha(0)) \in (-\pi, \pi]$. Let
$$ \theta_1 = \inf \{ \Arg(\alpha(t)): \alpha(t) \in A_\delta \} \hskip20pt \text{ and } \hskip20pt \theta_2 = \sup \{ \Arg(\alpha(t)): \alpha(t) \in A_\delta \}.$$ We consider two cases: $\theta_2 - \theta_1 < 2\pi$, and $\theta_2 - \theta_1 \geq 2\pi$.

Suppose first that $\theta_2 - \theta_1 < 2\pi$. For $k=1,2$, let $z_k = e^{i\theta_k}$, $r_k = \sup \{ r : re^{i\theta_k} \in \alpha[0, t_\alpha] \}$ and $t_k$ be such that $\alpha(t_k) = r_k e^{i \theta_k}$ (the $t_k$ exist because $\alpha[0, t_\alpha]$ is compact). See Figure \ref{pic7}.

\begin{figure}[htp]
\centerline{
\epsfxsize=3.5in
\epsfysize=3.5in
\epsfbox{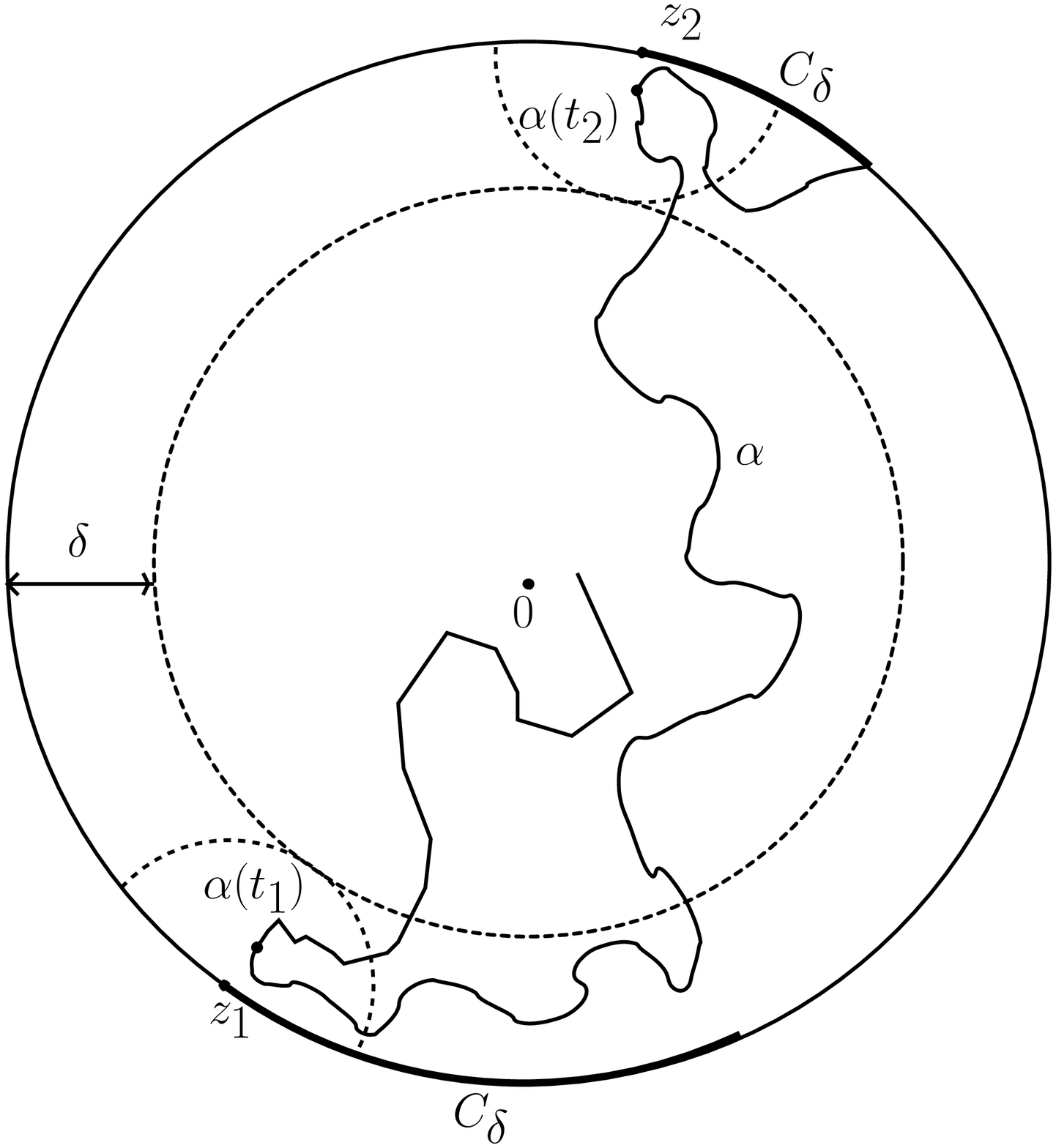}
}
\caption{The set $C_\delta$ and the points $z_1$, $z_2$ and $\alpha(t_1)$, $\alpha(t_2)$}
\label{pic7}
\end{figure}

 We construct a continuous curve $\omega$ as follows.  Given any $z = e^{i \theta} \in \partial \Disc$, let
$$ r_\delta(z) = \{ r e^{i\theta} : 1 - \delta \leq r \leq 1\}.$$We let $\omega(0) = \alpha(t_1)$, then follow the curve $\alpha$ from $\alpha(t_1)$ to $\alpha(t_2)$ (we might be following the curve backwards), then the ray $r_\delta(z_2)$ from $\alpha(t_2)$ to $z_2$, then $\partial \Disc$ clockwise from $z_2$ to $z_1$, and finally the ray $r_\delta(z_1)$ back to $\alpha(t_1)$. See Figure \ref{pic8}.

\begin{figure}[htp]
\centerline{
\epsfxsize=3.5in
\epsfysize=3.5in
\epsfbox{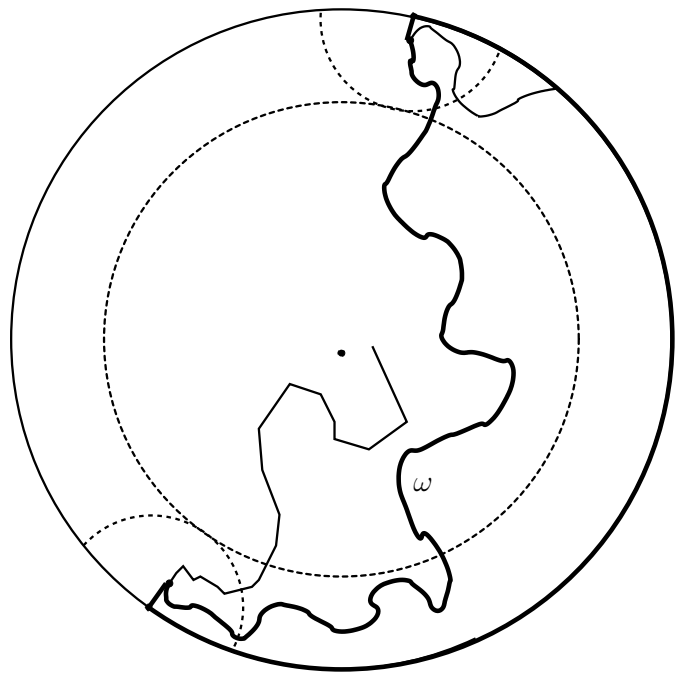}
}
\caption{The curve $\omega$ in the case $\theta_2 - \theta_1 < 2 \pi$}
\label{pic8}
\end{figure}

From the definition of the $t_k$ and $z_k$, and the fact that $\alpha$ is simple, $\omega$ is a closed simple curve. Therefore, by the Jordan curve theorem, $\omega$ separates the plane into two disjoint connected components. Furthermore, because $\theta_2 - \theta_1 < 2\pi$, the winding number of $\omega$ with respect to $0$ is $0$. Therefore, $0$ is in the unbounded component defined by $\omega$. Furthermore, $C_\delta \subset [z_1,z_2]$, where $[z_1, z_2]$ is the arc $\{e^{i\theta} : \theta_1 \leq \theta \leq \theta_2\}$. Therefore, $D_\delta(z_1) \cup \alpha[0,t_\alpha] \cup D_\delta(z_2)$ disconnects $0$ from $\wt{C}_\delta$ in $\Disc$.

Now suppose that $\theta_2 - \theta_1 \geq 2\pi$. Let $z_1 = e^{i \theta_1}$ and $z_2 = e^{i \theta_2}$. In order to prove the lemma for this case, we claim that it suffices to show that either $r_\delta(z_1) \cap \alpha[0, t_\alpha]$ or $r_\delta(z_2) \cap \alpha[0, t_\alpha]$ contains two points whose Argument differs by a nonzero multiple of $2\pi$. For suppose that $w_1 = \abs{w_1}e^{i\theta_1} = \alpha(s_1)$ and $w_2 = \abs{w_2} e^{i\theta_1} = \alpha(s_2)$ are such that $\Arg(w_2) - \Arg(w_1) = 2k\pi$, $k \neq 0$, and $w_1$ and $w_2$ are chosen so that $\arg(\alpha(t)) \neq \theta_1$ for $t$ between $s_1$ and $s_2$. Also, without loss of generality, $\abs{w_1} < \abs{w_2}$. Then we can consider the curve $\omega$ that starts at $w_1$, follows $\alpha$ from $w_1$ to $w_2$, and then returns to $w_1$ along the ray $r_\delta(z_1)$ (see Figure \ref{pic9}). By construction, $\omega$ is a closed simple curve whose winding number is nonzero. Therefore, $\omega$ contains $0$, and since $\omega \subset \Disc$, it disconnects $0$ from $\partial \Disc$. This shows that $r_\delta(z_1) \cup \alpha[0, t_\alpha] \cup r_\delta(z_2)$ disconnects $0$ from $\partial \Disc$, from which the lemma follows. 

\begin{figure}[htp]
\centerline{
\epsfxsize=3.5in
\epsfysize=3.5in
\epsfbox{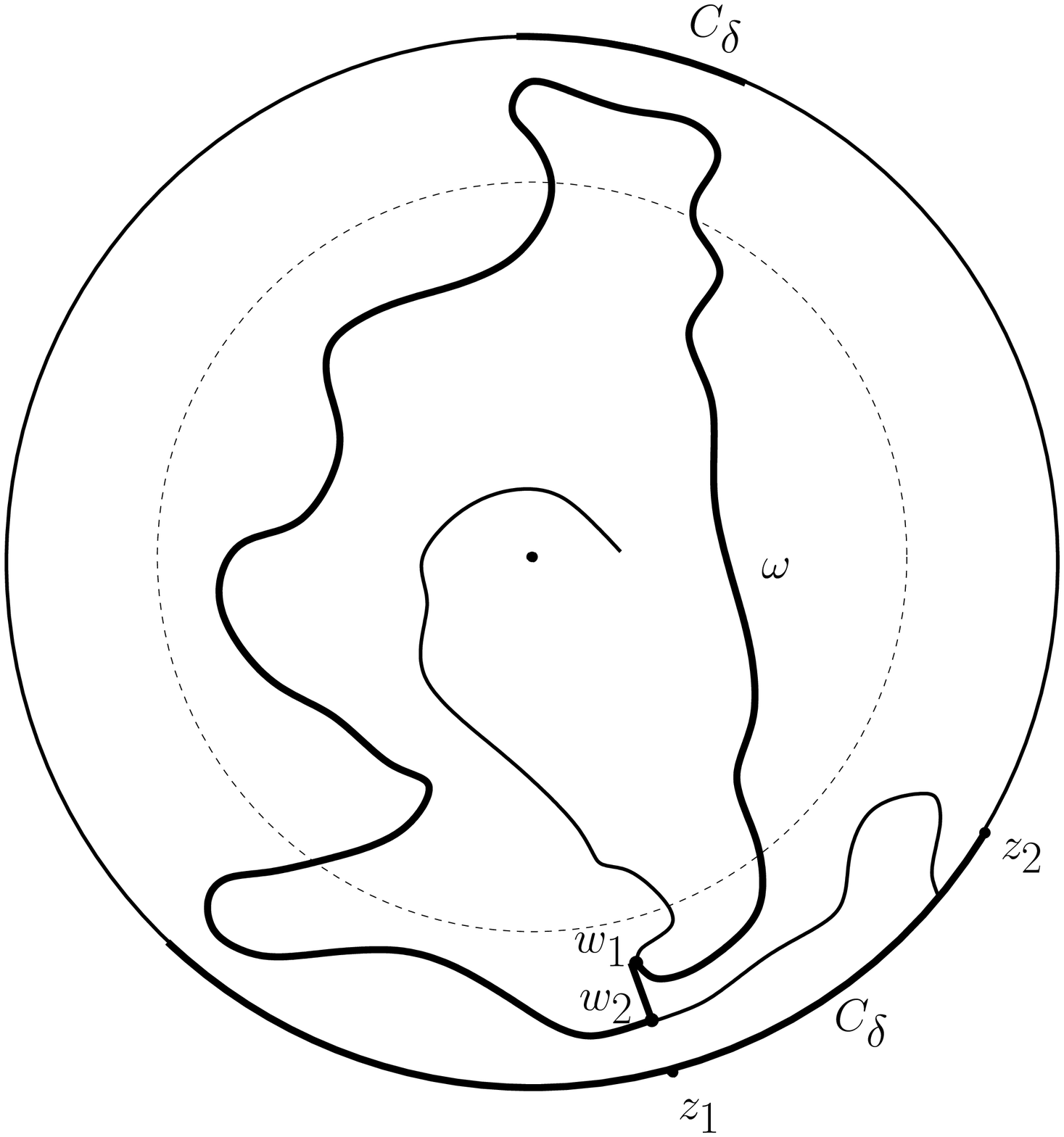}
}
\caption{The curve $\omega$ in the case $\theta_2 - \theta_1 > 2 \pi$}
\label{pic9}
\end{figure}

In order to show that either $r_\delta(z_1) \cap \alpha[0, t_\alpha]$ or $r_\delta(z_2) \cap \alpha[0, t_\alpha]$ contains two points whose Argument differs by a nonzero multiple of $2\pi$, we let $r_k$ and $t_k$ be such that $r_k = \sup \{ \abs{\alpha(t)} : \Arg(\alpha(t)) = \theta_k \}$ and $\alpha(t_k) = r_k e^{i \theta_k}$, $k=1,2$. We assume to the contrary that both $\{ r e^{i\theta_1} : r_1 < r \leq 1 \} \cap \alpha[0, t_\alpha] = \emptyset$ and $\{ r e^{i\theta_2} : r_2 < r \leq 1 \} \cap \alpha[0, t_\alpha] = \emptyset$. Then we can define two curves $\omega_1$ and $\omega_2$ as follows. $\omega_1$ starts at $r_1 e^{i \theta_1}$, travels along $\alpha$ to $r_2 e^{i \theta_2}$, follows the ray $r_\delta(z_2)$ to $z_2$, then travels along $\partial \Disc$ clockwise to $z_1$, and finally returns to $r_1 e^{i \theta_1}$ along $r_\delta(z_1)$. We define $\omega_2$ in the same way except that we travel along $\partial \Disc$ clockwise. Then by our assumptions, $\omega_1$ and $\omega_2$ are both closed simple curves with nonzero winding number, and hence both contain the origin, a contradiction.
\end{proof}

\begin{lemma} \label{skor}
For all $r> 0$ and $\epsilon > 0$, there exists $N$ such that for all $n \geq N$, the following holds. There exists a Brownian motion $W$ and a random walk $S$ defined on the same probability space such that if $S^{(n)}_t = n^{-1}S_{n^2t}$, then for all continuous curves $\alpha: [0, t_\alpha] \to \overline{\Disc}$, with $\operatorname{diam}(\alpha[0,t_\alpha]) \geq r$,
$$ \Pro{0}{\abs{W(\tau_\alpha \wedge \tau_\Disc) - S^{(n)}(\xi_\alpha \wedge \sigma_\Disc)} > r} < \epsilon.$$
\end{lemma}

\begin{proof} Let $T$ be large enough so that $\Pro{0}{\tau_\Disc > T} < \epsilon$. By the strong approximation of Brownian motion by random walk \cite[Theorem 3.5.1]{LL08}, there exists a sequence $\wt{S}^{n}$ of random walks defined on the same probability space as $W$ so that if $S^{(n)}_t = n^{-1} \wt{S}^{n}_{n^2t}$ is defined as above, then almost surely,
$$ \lim_{n \to \infty} \sup_{0 \leq t \leq T} \abs{S^{(n)}_t - W_t} = 0.$$
Therefore, we can let $N$ be such that for $n \geq N$,
$$ \Pro{0}{ \sup_{0 \leq t \leq T} \abs{S^{(n)}_t - W_t} > \frac{\epsilon^2 r}{C^2}} < \epsilon,$$
where $C$ is the larger of the constants in the Beurling estimates (Theorem \ref{beurling}).

Now fix $n \geq N$ and let $\tau^* = \tau_\alpha \wedge \tau_\Disc$ and $\sigma^* = \xi_\alpha \wedge \sigma_\Disc$. Suppose first that $\tau^* < \sigma^*$, and let $W_{\tau^*} = w$, $S^{(n)}_{\tau^*} = z$. Suppose further that $\tau_\Disc < T$ and that
$$ \sup_{0 \leq t \leq T} \abs{S^{(n)}_t - W_t} \leq \frac{\epsilon^2 r}{C^2}.$$
Then on this event, $\abs{z - w} \leq C^{-2} \epsilon^2 r$. Since both $\alpha$ and $\partial \Disc$ are continuous curves, by the Beurling estimates, letting $D = D_r(w)$,
\begin{eqnarray*}
\Pro{z}{ \abs{S^{(n)}_{\sigma^*} - w} > r} &\leq& \Pro{z}{ S^{(n)}[0,\sigma_D] \cap (\alpha \cup \partial \Disc) = \emptyset } \\
&\leq& C \left(\frac{\epsilon^2 r / C^2}{r}\right)^{1/2} = \epsilon.
\end{eqnarray*}
The case where $\sigma^* < \tau^*$ is proved in the same way, using the Beurling estimates for Brownian motion.
\end{proof}

\begin{lemma} \label{ckappa} Recall the definition of $\wt{C}_\delta$ from the beginning of the section.
\begin{enumerate}
\item \label{ckappa1} For all $\epsilon > 0$, there exists $\delta > 0$ such that for all $\alpha: [0, t_\alpha] \to \Disc$,
$$ \Pro{0}{W[0,\tau_\Disc] \cap \alpha[0,t_\alpha] = \emptyset; W(\tau_\Disc) \in \wt{C}_\delta} < \epsilon.$$
\item \label{ckappa2} For all $\epsilon > 0$, there exists $\delta > 0$ and $N$ such that for all $n \geq N$ and $\alpha: [0, t_\alpha] \to \Disc$,
$$ \Pro{0}{S^{(n)}[0,\sigma_\Disc] \cap \alpha[0,t_\alpha] = \emptyset; S^{(n)}(\sigma_\Disc) \in \wt{C}_\delta} < \epsilon $$
where $S^{(n)}_t = n^{-1} S_{n^2t}$.
\end{enumerate}
\end{lemma}

\begin{proof}
By Lemma \ref{jordan}, there exist $z_1, z_2 \in \partial \Disc$ such that
\begin{eqnarray*}
& & \Pro{0}{W[0,\tau_\Disc] \cap \alpha[0,t_\alpha] = \emptyset, W(\tau_\Disc) \in \wt{C}_\delta} \\ 
&\leq& \Pro{0}{W[0,\tau_\Disc] \cap D_\delta(z_1) \neq \emptyset} + \Pro{0}{W[0,\tau_\Disc] \cap D_\delta(z_2) \neq \emptyset}.
\end{eqnarray*}
However, by rotational symmetry of Brownian motion,
$$ \Pro{0}{W[0,\tau_\Disc] \cap D_\delta(z) \neq \emptyset}$$
is the same for all $z \in \partial \Disc$. Since, $D_\delta(z)$ shrinks to a single point as $\delta$ tends to $0$, the right-hand side above can be made to be less than $\epsilon$ by making $\delta$ small enough. This proves \ref{ckappa1}.

The proof of \ref{ckappa2} is the same as \ref{ckappa1}, except that we cannot use any sort of rotational symmetry. Therefore, we must show that there exists $\delta > 0$ and $N$ such that for all $n \geq N$ and $z \in \partial \Disc$,
$$ \Pro{0}{S^{(n)}[0, \sigma_\Disc] \cap D_\delta(z) \neq \emptyset} < \epsilon.$$
Let $\delta > 0$ be small enough so that for all $z \in \partial \Disc$,
$$ \Pro{}{W[0,\tau_2] \cap D_\delta(z) \neq \emptyset} < \epsilon,$$
where $\tau_2$ is the hitting time of the circle of radius $2$ by the Brownian motion $W$. We now apply Lemma \ref{skor} to obtain that there exists $N$ such that for all $n \geq N$, there exists a simple random walk $S$, defined on the same probability space as $W$ such that
$$ \Pro{0}{\abs{W(\tau_{D_\delta(z)} \wedge \tau_2) - S^{(n)}(\xi_{D_\delta(z)} \wedge \sigma_2)} > \delta } < \epsilon.$$
This implies that
\begin{eqnarray*}
& &
\Pro{0}{S^{(n)}[0, \sigma_\Disc] \cap D_\delta(z) \neq \emptyset} \\
&\leq& \Pro{0}{\xi_{D_\delta(z)} < \sigma_2} \\
&=& \Pro{0}{\xi_{D_\delta(z)} < \sigma_2; \tau_{D_\delta(z)} < \tau_2} + \Pro{}{\xi_{D_\delta(z)} < \sigma_2; \tau_{D_\delta(z)} > \tau_2} \\
&\leq& \Pro{0}{\tau_{D_\delta(z)} < \tau_2} + \Pro{0}{\abs{W(\tau_{D_\delta(z)} \wedge \tau_2) - S^{(n)}(\xi_{D_\delta(z)} \wedge \sigma_2)} > \delta} \\
&<& 2\epsilon.
\end{eqnarray*}
\end{proof}

\begin{prop} \label{a}
For all $\epsilon > 0$ and $r > 0$, there exists $N$ such that for all $n \geq N$ and all continuous curves $\alpha: [0,t_\alpha] \to \Disc$ with $\operatorname{diam}(\alpha[0,t_\alpha]) \geq r$, 
$$ \abs{\Pro{0}{W[0,\tau_\Disc] \cap \alpha[0,t_\alpha] = \emptyset} - \Pro{0}{S^{(n)}[0,\sigma_\Disc] \cap \alpha[0,t_\alpha] = \emptyset}} < \epsilon,$$
where $S^{(n)}_t = n^{-1} S_{n^2t}$.
\end{prop}

\begin{proof}
By Lemmas \ref{skor} and \ref{ckappa} , there exists $\delta > 0$ and $N$ such that for all $n \geq N$, the following holds. There exists a Brownian motion $W$ and a random walk $S$ defined on the same probability space such that for all continuous curves $\alpha: [0,t_\alpha] \to \Disc$ with $\operatorname{diam}(\alpha[0,t_\alpha]) > r$,
\begin{eqnarray}
\Pro{0}{W[0,\tau_\Disc] \cap \alpha[0,t_\alpha] = \emptyset; W(\tau_\Disc) \in \wt{C}_\delta} < \epsilon; \label{a1} \\
\Pro{0}{S^{(n)}[0,\sigma_\Disc] \cap \alpha[0,t_\alpha] = \emptyset; S^{(n)}(\sigma_\Disc) \in \wt{C}_\delta} < \epsilon; \label{a2}\\
\Pro{0}{\abs{W(\tau^*) - S^{(n)}(\sigma^*)} > \delta } < \epsilon, \label{a3}
\end{eqnarray}
where $\tau^* = \tau_\alpha \wedge \tau_\Disc$ and $\sigma^* = \xi_\alpha \wedge \sigma_\Disc$.

We will show that the proposition holds with this choice of $N$. Note that
\begin{eqnarray*}
& & \abs{\Pro{0}{W[0,\tau_\Disc] \cap \alpha[0,t_\alpha] = \emptyset} - \Pro{0}{S^{(n)}[0,\sigma_\Disc] \cap \alpha[0,t_\alpha] = \emptyset}} \\
&\leq& \Prob{0}{E} + \Prob{0}{F},
\end{eqnarray*}
where
$$ E = \{ W(\tau^*) \in \partial \Disc ; S^{(n)}(\sigma^*) \in \alpha[0, t_\alpha] \},$$
$$ F = \{ W(\tau^*) \in  \alpha[0, t_\alpha] ; S^{(n)}(\sigma^*) \in \partial \Disc \}.$$
We will show that $\Prob{0}{E} < \epsilon$. The proof that $\Prob{0}{F} < \epsilon$ is entirely similar. Recall that $D_{1 - \delta}$ denotes the ball of radius $1 - \delta$ and that $A_\delta$ denotes the annulus $\Disc \setminus D_{1 - \delta}$. Then,
$$ \Prob{0}{E} \leq \Prob{0}{E_1} + \Prob{0}{E_2} + \Prob{0}{E_3},$$
where
$$ E_1 = \{ W(\tau^*) \in \wt{C}_\delta ; S^{(n)}(\sigma^*) \in \alpha[0,t_\alpha] \};$$
$$ E_2 = \{ W(\tau^*) \in \partial \Disc ; S^{(n)}(\sigma^*) \in \alpha[0,t_\alpha] \cap D_{1 - \delta} \};$$
$$ E_3 = \{ W(\tau^*) \in \partial \Disc \setminus \wt{C}_\delta ; S^{(n)}(\sigma^*) \in \alpha[0,t_\alpha] \cap A_\delta \}.$$
However by (\ref{a1}), $\Prob{0}{E_1} < \epsilon$, and by (\ref{a3}), $\Prob{0}{E_2} < \epsilon$ and $\Prob{0}{E_3} < \epsilon$.
\end{proof}

\section{Some results for loop-erased random walks}

\label{LERWsec}

\subsection{Up to constant independence of the initial and terminal parts of a LERW path} \label{measures}

\textit{For this section only}, we no longer restrict our random walks to be two-dimensional. When it is necessary to specify what dimension we are in, we will denote the dimension by $d$.

Although we have avoided using it up to now, it will be convenient to use ``big-O'' notation in this section. Recall that $f(n) = O(a(n))$ if there exists $C < \infty$ such that 
$$ f(n) \leq Ca(n).$$
Here, $C$ can depend on the dimension but on no other quantity. 
We will also write 
$$ f(n) = \left[1 + O(a(n))\right] g(n), $$
if there exists $C < \infty$ such that  
$$ 1 - Ca(n) \leq \frac{f(n)}{g(n)} \leq 1 + C a(n).$$

Recall that for a natural number $l$, $\Omega_l$ denotes the set of paths 
$$ \omega = [0, \omega_1, \ldots, \omega_k]$$ such that $\omega_j \in B_l$, $j=0,1, \ldots, k-1$ and $\omega_k \in \partial B_l$.

Given a set $K$ such that $B_l \subset K$, and such that
$$ \Pro{0}{\sigma_K < \infty} = 1,$$
we define $\mu_{l,K}$ on $\Omega_l$ to be the measure obtained by running a random walk up to the first exit time $\sigma_K$ of $K$, loop-erasing and restricting to $B_l$. More precisely, for $\omega \in \Omega_l$,
$$ \mu_{l,K}(\omega) = \Pro{0}{\widehat{S}^K[0, \widehat{\sigma}_l] = \omega}.$$

If $B_l \subset K_1$ and $B_l \subset K_2$ are such that 
$$ \Pro{0}{\sigma_{K_i} < \infty} = 1$$
for either $i=1$ or $i=2$, we define a measure $\mu_{l,K_1,K_2}$ on $\Omega_l$ as follows. Let $X$ denote random walk conditioned to leave $K_1$ before $K_2$ (as long as this has positive probability; if not, $\mu_{l,K_1, K_2}$ is not defined). Then for $\omega \in \Omega_l$, we let
$$ \mu_{l,K_1,K_2}(\omega) = \Pro{0}{\widehat{X}^{K_1}[0, \widehat{\sigma}_l] = \omega}.$$
This is the measure on $\Omega_l$ obtained by running $X$ up to $\sigma^X_{K_1}$, loop-erasing and restricting to $B_l$.
Note that $\mu_{l,K}$ is equal to $\mu_{l, K , \Lambda}$.

In this section, we establish some relations between the measures defined above. In fact we will show that for $n \geq 4$ and any $K_1$ and $K_2$ such that $B_{nl} \subset K_1$ and $B_{nl} \subset K_2$,
\begin{eqnarray*}
\mu_{l,K_1, K_2}(\omega) &=& \left\{ \begin{array}{ll}
\left[1 + O(\frac{1}{\log n})\right]\mu_{l,K_1 \cap K_2}(\omega) & d = 2; \\
\left[1 + O(n^{-1})\right]\mu_{l,K_1 \cap K_2}(\omega) & d \geq 3. 
\end{array} \right.
\hskip10pt \text{ (Proposition
\ref{aw})} \\
\mu_{l, K_1}(\omega) &=& \left\{ \begin{array}{ll}
\left[1 + O(\frac{1}{\log n})\right]\mu_{l,K_2}(\omega) & d = 2; \\
\left[1 + O(n^{2-d})\right]\mu_{l,K_2}(\omega) & d \geq 3. 
\end{array} \right. \hskip10pt \text{ (Proposition \ref{ax})}
\end{eqnarray*}
This implies that if $B_{4l} \subset K_1$ and $B_{4l} \subset K_2$ then
\begin{eqnarray*} 
\mu_{l,K_1,K_2}(\omega) &\asymp& \mu_{l,K_1 \cap K_2}(\omega) \\
\mu_{l,K_1}(\omega) &\asymp& \mu_{l,K_2}(\omega)
\end{eqnarray*}
(recall that the symbol $\asymp$ means that each side is bounded by a constant multiple of the other side, the constant depending on the random walk $S$ and on \emph{nothing else}). We use these facts to prove that for a LERW $\wh{S}^n$, $\eta^1_l(\wh{S}^n)$ and $\eta^2_{4l,n}(\wh{S}^n)$ (see the definitions in section \ref{sets}) are independent up to constants (Proposition \ref{indep}).

\begin{lemma} \label{avoidfar}
Let $d=2$. For $\omega \in \Omega_l$ and $y \in \partial B_{l}$,
$$ \Pro{y}{\sigma_{nl} < \xi_\omega} = O(\frac{1}{\log n}).$$
\end{lemma}

\begin{proof}
Let $y_0$ be such that
$$ \Pro{y_0}{\sigma_{nl} < \xi_\omega} = \max_{y \in \partial B_l} \Pro{y}{\sigma_{nl} < \xi_\omega}.$$
We will show that 
$$  \Pro{y_0}{\sigma_{nl} < \xi_\omega} \leq \frac{C}{\log n}$$ 
which will clearly imply the result for all $y \in \partial B_l$. 

\begin{eqnarray*}
\Pro{y_0}{\sigma_{nl} < \xi_\omega} &=& \Pro{y_0}{\sigma_{2l} < \xi_\omega} \sum_{z \in \partial B_{2l}} \Pro{z}{\sigma_{nl} < \xi_\omega} \condPro{y}{S(\sigma_{nl}) = z}{\sigma_{2l} < \xi_\omega}.
\end{eqnarray*}
By part (1) of Lemma \ref{disc}, there exists $c > 0$ such that
$$ \Pro{y_0}{\sigma_{2l} < \xi_\omega} \leq 1 - c.$$
Furthermore, for $z \in \partial B_{2l}$,
\begin{eqnarray*}
\Pro{z}{\sigma_{nl} < \xi_\omega} &=& \Pro{z}{\sigma_{nl} < \xi_l} + \sum_{y \in \partial B_l} \Pro{z}{S(\xi_l)=y; \xi_l < \sigma_{nl}} \Pro{y}{\sigma_{nl} < \xi_\omega} \\
&\leq& \Pro{z}{\sigma_{nl} < \xi_l} + \Pro{y_0}{\sigma_{nl} < \xi_\omega}.
\end{eqnarray*}
By \cite[Proposition 6.4.1]{LL08}, 
$$ \Pro{z}{\sigma_{nl} < \xi_l} \asymp (\log n)^{-1}.$$
Therefore,
\begin{eqnarray*}
\Pro{y_0}{\sigma_{nl} < \xi_\omega} &\leq& (1-c)\left[C(\log n)^{-1}+ \Pro{y_0}{\sigma_{nl} < \xi_\omega}\right]
\end{eqnarray*}
which gives the desired result.
\end{proof}

\begin{prop} \label{aw}
Suppose that $n \geq 4$, $K_1$ and $K_2$ are such that $B_{nl} \subset K_1$ and $B_{nl} \subset K_2$, and that $\omega = [0, \omega_1, \ldots, \omega_k] \in \Omega_l$. Then,
\begin{eqnarray*}
\mu_{l,K_1, K_2}(\omega) &=& \left\{ \begin{array}{ll}
\left[1 + O(\frac{1}{\log n})\right]\mu_{l,K_1 \cap K_2}(\omega) & d = 2; \\
\left[1 + O(n^{-1})\right]\mu_{l,K_1 \cap K_2}(\omega) & d \geq 3. 
\end{array} \right.
\end{eqnarray*}
In particular, if  $B_{4l} \subset K_1$ and $B_{4l} \subset K_2$ then
$$ \mu_{l,K_1,K_2}(\omega) \asymp \mu_{l,K_1 \cap K_2}(\omega).$$
\end{prop}

\begin{proof}
Let $K = K_1 \cap K_2$. Let $X$ be a random walk conditioned to exit $K_1$ before $K_2$. Then by formula (\ref{lerwf}),
$$ \mu_{l,K}(\omega) = p(\omega) G_K(\omega) \Pro{\omega_k}{\sigma_K < \xi_{\omega}},$$
and
$$ \mu_{l,K_1,K_2}(\omega) = p^X(\omega) G^X_K(\omega) \condPro{\omega_k}{\sigma_{K_1} < \xi_{\omega} }{ \sigma_{K_1} < \sigma_{K_2}}.$$

By Lemma \ref{greencondit}, $G^X_K(\omega) = G_K(\omega)$. Furthermore, if we let $h(z) = \Pro{z}{\sigma_{K_1} < \sigma_{K_2}}$, then by (\ref{htransform}),
$$ p^X(\omega) = \frac{h(\omega_k)}{h(0)} p(\omega).$$
The function $h$ is harmonic in $B_{nl}$ and $\omega_k \in B_l$. Therefore, by the difference estimates for harmonic functions \cite[Theorem 6.3.8]{LL08},
$$ h(\omega_k) = \left[1 + O(n^{-1})\right]h(0),$$
and thus, 
$$  p^X(\omega) = \left[1 + O(n^{-1})\right]p(\omega).$$

Hence, it suffices to show that 

$$\Pro{\omega_k}{\sigma_{K_1} < \xi_{\omega} \wedge \sigma_{K_2}} = \left\{ \begin{array}{ll}
\left[1 + O(\frac{1}{\log n})\right] \Pro{\omega_k}{\sigma_{K_1} < \sigma_{K_2}} \Pro{\omega_k}{\sigma_K < \xi_{\omega}}  & d = 2; \\
\left[1 + O(n^{-1})\right] \Pro{\omega_k}{\sigma_{K_1} < \sigma_{K_2}} \Pro{\omega_k}{\sigma_K < \xi_{\omega}} & d \geq 3. 
\end{array} \right. $$

Let $y_0 \in \omega$ be such that
$$ \Pro{y_0}{\sigma_{K_1} < \sigma_{K_2}} = \max_{y \in \omega} \Pro{y}{\sigma_{K_1} < \sigma_{K_2}}.$$
Then
\begin{eqnarray*}
\Pro{y_0}{\sigma_{K_1} < \sigma_{K_2}} &=&  \condPro{y_0}{\sigma_{K_1} < \sigma_{K_2}}{\sigma_K < \xi_\omega} \Pro{y_0}{\sigma_K < \xi_\omega} \\
&+& \condPro{y_0}{\sigma_{K_1} < \sigma_{K_2}}{\xi_\omega < \sigma_K} \Pro{y_0}{\xi_\omega < \sigma_K} \\
&\leq& \condPro{y_0}{\sigma_{K_1} < \sigma_{K_2}}{\sigma_K < \xi_\omega} \Pro{y_0}{\sigma_K < \xi_\omega} \\
&+& \Pro{y_0}{\sigma_{K_1} < \sigma_{K_2}}\Pro{y_0}{\xi_\omega < \sigma_K}.
\end{eqnarray*}
Therefore,
$$ \Pro{y_0}{\sigma_{K_1} < \sigma_{K_2}} \leq \condPro{y_0}{\sigma_{K_1} < \sigma_{K_2}}{\sigma_K < \xi_\omega} = \frac{\Pro{y_0}{\sigma_{K_1} < \sigma_{K_2} \wedge \xi_\omega}}{\Pro{y_0}{\sigma_K < \xi_\omega}},$$ 
and hence
\begin{eqnarray} \label{okr1}
\Pro{\omega_k}{\sigma_{K_1} < \sigma_{K_2}} \leq \max_{y \in \omega} \frac{\Pro{y}{\sigma_{K_1} < \sigma_{K_2} \wedge \xi_\omega}}{\Pro{y}{\sigma_K < \xi_\omega}}.
\end{eqnarray}
A similar argument shows that
\begin{eqnarray} \label{okr2}
\Pro{\omega_k}{\sigma_{K_1} < \sigma_{K_2}} \geq \min_{y \in \omega} \frac{\Pro{y}{\sigma_{K_1} < \sigma_{K_2} \wedge \xi_\omega}}{\Pro{y}{\sigma_K < \xi_\omega}}.
\end{eqnarray}

Now let $y$ be any point on the path $\omega$. Then since $B_{nl}$ is a subset of both $K_1$ and $K_2$,
\begin{eqnarray*}
& & \Pro{y}{\sigma_{K_1} < \sigma_{K_2} \wedge \xi_\omega} \\
&=& \Pro{y}{\sigma_{nl} < \xi_\omega} \sum_{z \in \partial B_{nl}} \Pro{z}{\sigma_{K_1} < \sigma_{K_2} \wedge \xi_\omega} \condPro{y}{S(\sigma_{nl}) = z}{\sigma_{nl} < \xi_{\omega}},
\end{eqnarray*}
and
\begin{eqnarray*}
\Pro{y}{\sigma_{K} < \xi_\omega} 
= \Pro{y}{\sigma_{nl} < \xi_\omega} \sum_{z \in \partial B_{nl}} \Pro{z}{\sigma_{K} < \xi_\omega} \condPro{y}{S(\sigma_{nl}) = z}{\sigma_{nl} < \xi_{\omega}}.
\end{eqnarray*}

However \cite[Lemma 2.1.2]{Law91},
\begin{eqnarray*}
\condPro{y}{S(\sigma_{nl}) = z}{\sigma_{nl} < \xi_{\omega}} &=& \left\{ \begin{array}{ll} \left[1 +  O(\frac{\log n}{n})\right] \Pro{0}{S(\sigma_{nl}) = z} & d = 2; \\
\left[1 +  O(n^{-1})\right] \Pro{0}{S(\sigma_{nl}) = z} & d \geq 3;
\end{array} \right.
\end{eqnarray*}

Let $y_1$ be such that 
$$ \frac{\Pro{y_1}{\sigma_{K_1} < \sigma_{K_2} \wedge \xi_\omega}}{\Pro{y_1}{\sigma_K < \xi_\omega}} = \max_{y \in \omega} \frac{\Pro{y}{\sigma_{K_1} < \sigma_{K_2} \wedge \xi_\omega}}{\Pro{y}{\sigma_K < \xi_\omega}},$$
and $y_2$ be such that
$$ \frac{\Pro{y_2}{\sigma_{K_1} < \sigma_{K_2} \wedge \xi_\omega}}{\Pro{y_2}{\sigma_K < \xi_\omega}} = \min_{y \in \omega} \frac{\Pro{y}{\sigma_{K_1} < \sigma_{K_2} \wedge \xi_\omega}}{\Pro{y}{\sigma_K < \xi_\omega}}.$$  
Then by (\ref{okr1}) and (\ref{okr2}), if $d = 2$,
\begin{eqnarray*}
\frac{\Pro{\omega_k}{\sigma_{K_1} < \sigma_{K_2} \wedge  \xi_{\omega}}}{\Pro{\omega_k}{\sigma_{K_1} < \sigma_{K_2}} \Pro{\omega_k}{\sigma_K < \xi_{\omega}}} 
&\leq& \frac{\Pro{y_1}{\sigma_{K_1} < \sigma_{K_2} \wedge \xi_\omega} \Pro{y_2}{\sigma_K < \xi_\omega}}{\Pro{y_1}{\sigma_K < \xi_\omega} \Pro{y_2}{\sigma_{K_1} < \sigma_{K_2} \wedge \xi_\omega}} \\
&\leq& \frac{(1+ C \log n / n)^2}{(1 - C \log n / n)^2} \\
&\leq& 1 + \frac{C' \log n}{n}.
\end{eqnarray*}
The lower bound and the case $d \geq 3$ follows in the same way.
\end{proof}

We now define a measure on unrooted loops in $\Lambda$. See \cite[Chapter 9]{LL08} for more details.

A rooted loop $\eta = [\eta_0, \eta_1, \ldots, \eta_k]$ is a path in $\Lambda$ such that $\eta_0 = \eta_k$; $\eta_0$ is called the root of the loop. We say that two rooted loops $\eta$ and $\eta'$ are equivalent if $\eta' = [\eta_j, \eta_{j+1}, \ldots, \eta_{k-1}, \eta_0, \ldots,  \eta_j]$ for some $j$. We call the equivalence classes under this relation unrooted loops. We will denote by $\widetilde{\eta}$ the unrooted loop corresponding to the rooted loop $\eta$. Recall the notation
$$ p(\eta) := \prod_{i=1}^k p(\eta_{i-1},\eta_i) = \Pro{\eta_0}{S_i = \eta_i, i=0, \ldots k}.$$
Notice that this does not depend on the root of $\eta$ and therefore $p(\wt{\eta})$ is well defined for unrooted loops $\wt{\eta}$.

We define a measure $m$ on the set of unrooted loops as follows. Given an unrooted loop $\widetilde{\eta}$, let $\alpha(\widetilde{\eta})$ be the number of distinct rooted representatives of $\widetilde{\eta}$. Then we define
$$ m(\widetilde{\eta}) = \frac{\alpha(\widetilde{\eta})p(\wt{\eta})}{|\wt{\eta}|},$$
where $|\wt{\eta}|$ denotes the number of steps of a representative of $\wt{\eta}$. Any two representatives of $\widetilde{\eta}$ have the same number of steps so that $m$ is well defined.

Recall the definition of $G_K(\omega)$ given in (\ref{GK}). The following lemma allows us to express $G_K(\omega)$ in terms of the unrooted loop measure. 

\begin{lemma} \label{aa}
$$ G_K(\omega) = \exp \left\{   m(\wt{\eta} : \wt{\eta} \subset K; \wt{\eta} \cap \omega \neq \emptyset) \right\}.$$
\end{lemma}

\begin{proof}
We will first show that for any $z \in \Lambda$,
\begin{eqnarray}
G_K(z) = \exp \{  m(\wt{\eta} : \wt{\eta} \subset K ; z \in \wt{\eta}) \}. \label{fd}
\end{eqnarray}
Let $\rho = \Pro{z}{\xi_z < \sigma_K}$. Then by the strong Markov property for random walk, 
$$G_K(z) = 1 + \rho G_K(z),$$ 
and thus,
$$ G_K(z) = \frac{1}{1 - \rho} = e^{ - \ln (1 - \rho)} = \exp \left\{ \sum_{j=1}^\infty \frac{\rho^j}{j} \right\}.$$

Given an unrooted loop $\wt{\eta}$, let $\kappa(\wt{\eta})$ be the largest integer $m$ that divides $\abs{\wt{\eta}} = n$ such that $\eta_k = \eta_{k + (n/m)}$ for all $k \leq n - (n/m)$. Then $\kappa(\wt{\eta}) = \abs{\wt{\eta}} / \alpha(\wt{\eta})$ and therefore
$$ m(\wt{\eta}) = \frac{p(\eta)}{\kappa(\wt{\eta})} .$$
Let
$$ \beta(\wt{\eta}) = \beta_z(\wt{\eta}) := \# \{j: \eta_j = z \}$$ be the number of times that $\wt{\eta}$ hits $z$. Then the number of representatives of $\wt{\eta}$ that are rooted at $z$ is $\beta(\wt{\eta})/ \kappa(\wt{\eta})$. Hence,
\begin{eqnarray*}
\rho^j &=& \sum_{\substack{\text{ $\eta$ rooted at $z$} \\ \eta \subset K  \\ \beta(\eta) = j}} p(\eta) \\
&=& \sum_{\substack{\text{ $\wt{\eta}$ unrooted} \\ z \in \wt{\eta}, \: \wt{\eta} \subset K  \\ \beta(\wt{\eta}) = j}} \sum_{\text{ $\eta$ rooted at $z$}} p(\eta) \\
&=& \sum_{\substack{\text{ $\wt{\eta}$ unrooted} \\ z \in \wt{\eta}, \: \wt{\eta} \subset K  \\ \beta(\wt{\eta}) = j}} \frac{\beta(\wt{\eta}) p(\wt{\eta})}{\kappa(\wt{\eta})} \\
&=& j \cdot m(z \in \wt{\eta}; \wt{\eta} \subset K ; \beta(\wt{\eta}) = j).
\end{eqnarray*}
Therefore,
$$\sum_{j=1}^\infty \frac{\rho^j}{j} = \sum_{j=1}^\infty m(z \in \wt{\eta}; \wt{\eta} \subset K ; \beta(\wt{\eta}) = j) = m(z \in \wt{\eta}; \wt{\eta} \subset K),$$
which proves (\ref{fd}).

Let $E_j$, $j=0, \ldots, k$, be the set of unrooted loops $\wt{\eta}$ such that $\omega_j \in \wt{\eta}$ and $\wt{\eta} \subset K \setminus \{\omega_0, \ldots, \omega_{j-1}\}$. Then the $E_j$ are disjoint and their union is the set of all unrooted loops $\wt{\eta}$ such that $\wt{\eta} \cap \omega \neq \emptyset$ and $\wt{\eta} \subset K$. This observation along with (\ref{fd}) finishes the proof.
\end{proof}

\begin{prop} \label{ax}
Suppose that $n \geq 4$, $K_1$ and $K_2$ are such that $B_{nl} \subset K_1$ and $B_{nl} \subset K_2$, and that $\omega \in \Omega_l$. Then 
\begin{eqnarray*}
\mu_{l, K_1}(\omega) &=& \left\{ \begin{array}{ll}
\left[1 + O(\frac{1}{\log n})\right]\mu_{l,K_2}(\omega) & d = 2; \\
\left[1 + O(n^{2-d})\right]\mu_{l,K_2}(\omega) & d \geq 3. 
\end{array} \right.
\end{eqnarray*}
In particular, if $B_{4l} \subset K_1$ and $B_{4l} \subset K_2$ then
$$ \mu_{l,K_1}(\omega) \asymp \mu_{l,K_2}(\omega).$$
\end{prop}

\begin{proof}
By Formula (\ref{lerwf}), for any $\omega \in \Omega_l$,
$$ \mu_{l,K_i}(\omega) = p(\omega) G_{K_i}(\omega) \Pro{\omega_k}{\sigma_{K_i} < \xi_{\omega}} \hskip15pt i = 1,2.$$

Let $e(n) = (\log n)^{-1}$ if $d=2$ and $e(n) = n^{2-d}$ if $d \geq 3$. Let $\omega' = [\omega'_0, \ldots, \omega'_{k'}]$ be any other path in $\Omega_l$. We will show that
\begin{eqnarray}
\frac{\Pro{\omega_k}{\sigma_{K_1} < \xi_{\omega}}}{\Pro{\omega_k}{\sigma_{K_2} < \xi_{\omega}}} \leq \left[1 + Ce(n)\right] \frac{\Pro{\omega'_{k'}}{\sigma_{K_1} < \xi_{\omega'}}}{\Pro{\omega'_{k'}}{\sigma_{K_2} < \xi_{\omega'}}} \label{az}
\end{eqnarray}
and
\begin{eqnarray}
\frac{G_{K_1}(\omega)}{G_{K_2}(\omega)} \leq \left[1 + Ce(n)\right] \frac{G_{K_1}(\omega')}{G_{K_2}(\omega')}.  \label{ay}
\end{eqnarray}
For this will imply that
$$ \frac{\mu_{l,K_1}(\omega')}{\mu_{l,K_2}(\omega')} \leq \left[1 + Ce(n)\right] \frac{\mu_{l,K_1}(\omega)}{\mu_{l,K_2}(\omega)}$$
and
\begin{eqnarray*}
1 &=& \sum_{\omega' \in \Omega_l} \mu_{l, K_1}(\omega') \\
&=& \sum_{\omega' \in \Omega_l} \frac{\mu_{l, K_1}(\omega')}{\mu_{l, K_2}(\omega')} \mu_{l, K_2}(\omega') \\
&\leq& \left[1 + Ce(n)\right] \frac{\mu_{l, K_1}(\omega)}{\mu_{l, K_2}(\omega)} \sum_{\omega' \in \Omega_l} \mu_{l, K_2}(\omega') \\
&=& \left[1 + Ce(n)\right] \frac{\mu_{l, K_1}(\omega)}{\mu_{l, K_2}(\omega)}.
\end{eqnarray*}
One then gets the other bound by reversing the roles of $K_1$ and $K_2$.

We first show (\ref{az}). Since $B_{nl} \subset K_1$
$$ \Pro{\omega_k}{\sigma_{K_1} < \xi_{\omega}} = \Pro{\omega_k}{\sigma_{nl} < \xi_{\omega}} \sum_{z \in \partial B_{nl}} \Pro{z}{\sigma_{K_1} < \xi_{\omega}} \condPro{\omega_k}{S(\sigma_{nl}) = z}{\sigma_{nl} < \xi_\omega}.$$
If $d \geq 3$, then \cite[Proposition 6.4.2]{LL08} for $z \in \partial B_{nl}$,
$$ \Pro{z}{\sigma_{K_1} < \xi_{\omega}} \geq \Pro{z}{\xi_{l} = \infty} \geq 1 - Cn^{2-d}.$$
Therefore, if $d \geq 3$,
$$ \left[1 - Cn^{2-d}\right] \Pro{\omega_k}{\sigma_{nl} < \xi_\omega} \leq \Pro{\omega_k}{\sigma_{K_1} < \xi_{\omega}} \leq \Pro{\omega_k}{\sigma_{nl} < \xi_\omega}.$$
One gets a similar formula with $K_2$ replacing $K_1$ and $\omega'$ replacing $\omega$, from which (\ref{az}) follows for the case $d \geq 3$.

To prove (\ref{az}) for the case $d = 2$, we first note that \cite[Lemma 2.1.2]{Law91}
$$ \condPro{\omega_k}{S(\sigma_{nl}) = z}{\sigma_{nl} < \xi_\omega} = \left[1 + O \left(\frac{\log n}{n}\right)\right] \Pro{0}{S(\sigma_{nl}) = z}.$$
Furthermore, for $z \in \partial B_{nl}$,
\begin{eqnarray*}
\Pro{z}{\sigma_{K_1} < \xi_{\omega}} &=& \Pro{z}{\sigma_{K_1} < \xi_l} + \sum_{y \in \partial_i B_l} \Pro{y}{\sigma_{K_1} < \xi_{\omega}} \Pro{z}{S(\xi_l) = y ; \xi_l < \sigma_{K_1}}.
\end{eqnarray*}
By applying Lemma \ref{avoidfar} and \cite[Lemma 2.1.2]{Law91} again we get that for $y \in \partial_i B_l$, 
\begin{eqnarray*}
\Pro{y}{\sigma_{K_1} < \xi_{\omega}} &=& \Pro{y}{\sigma_{nl} < \xi_\omega} \sum_{w \in \partial B_{nl}} \Pro{w}{\sigma_{K_1} < \xi_\omega} \condPro{y}{S(\sigma_{nl}) = z}{\sigma_{nl} < \xi_\omega} \\
&\leq& \frac{C}{\log n} \sum_{w \in \partial B_{nl}} \Pro{w}{\sigma_{K_1} < \xi_\omega} \Pro{0}{S(\sigma_{nl}) = w} \\
&\leq& \frac{C}{\log n} \frac{\Pro{\omega_k}{\sigma_{K_1} < \xi_\omega}}{\Pro{\omega_k}{\sigma_{nl} < \xi_\omega}}.
\end{eqnarray*}
Thus,
\begin{eqnarray*}
\Pro{z}{\sigma_{K_1} < \xi_\omega} &\leq& \Pro{z}{\sigma_{K_1} < \xi_l} + \frac{C}{\log n} \frac{\Pro{\omega_k}{\sigma_{K_1} < \xi_\omega}}{\Pro{\omega_k}{\sigma_{nl} < \xi_\omega}} \Pro{z}{\xi_l < \sigma_{K_1}} \\
&\leq& \Pro{z}{\sigma_{K_1} < \xi_l} + \frac{C}{\log n} \frac{\Pro{\omega_k}{\sigma_{K_1} < \xi_\omega}}{\Pro{\omega_k}{\sigma_{nl} < \xi_\omega}}.
\end{eqnarray*}
Therefore, 
\begin{eqnarray*}
\Pro{\omega_k}{\sigma_{K_1} < \xi_\omega} 
&=& \Pro{\omega_k}{\sigma_{nl} < \xi_{\omega}} \sum_{z \in \partial B_{nl}} \Pro{z}{\sigma_{K_1} < \xi_{\omega}} \condPro{\omega_k}{S(\sigma_{nl}) = z}{\sigma_{nl} < \xi_\omega} \\ 
&\leq& \left[1 + \frac{C \log n}{n}\right] \Pro{\omega_k}{\sigma_{nl} < \xi_\omega} \sum_{z \in \partial B_{nl}} \Pro{z}{\sigma_{K_1} < \xi_{\omega}}\Pro{0}{S(\sigma_{nl}) = z} \\
&\leq& \left[1 + \frac{C \log n}{n}\right] \Pro{\omega_k}{\sigma_{nl} < \xi_\omega} \sum_{z \in \partial B_{nl}} \Pro{z}{\sigma_{K_1} < \xi_l}\Pro{0}{S(\sigma_{nl}) = z} \\ &+& \frac{C}{\log n} \Pro{\omega_k}{\sigma_{K_1} < \xi_\omega},
\end{eqnarray*}
and hence,
$$ \Pro{\omega_k}{\sigma_{K_1} < \xi_\omega} \leq \left[1 + \frac{C}{\log n}\right] \Pro{\omega_k}{\sigma_{nl} < \xi_\omega} \sum_{z \in \partial B_{nl}} \Pro{z}{\sigma_{K_1} < \xi_l}\Pro{0}{S(\sigma_{nl}) = z},$$
with a similar lower bound. We get similar bounds with $\omega'$ replacing $\omega$ and $K_2$ replacing $K_1$ from which (\ref{az}) follows.

\vskip20pt

We now show (\ref{ay}). By Lemma \ref{aa},
\begin{eqnarray*}
& &
\frac{G_{K_1}(\omega)G_{K_2}(\omega')}{G_{K_2}(\omega)G_{K_1}(\omega')} \\
&=& \exp \{  m(\widetilde{\eta} \cap \omega \neq \emptyset; \wt{\eta} \cap \omega' = \emptyset; \widetilde{\eta} \subset K_1; \widetilde{\eta} \cap (\Lambda \setminus K_2) \neq \emptyset) \\
&+& m(\widetilde{\eta} \cap \omega = \emptyset; \wt{\eta} \cap \omega' \neq \emptyset; \widetilde{\eta} \cap (\Lambda \setminus K_1) \neq \emptyset; \widetilde{\eta} \subset K_2) \\
&-& m(\widetilde{\eta} \cap \omega \neq \emptyset; \wt{\eta} \cap \omega' = \emptyset; \widetilde{\eta} \cap (\Lambda \setminus K_1) \neq \emptyset; \widetilde{\eta} \subset K_2)  \\
&-& m(\widetilde{\eta} \cap \omega = \emptyset; \wt{\eta} \cap \omega' \neq \emptyset; \widetilde{\eta} \subset K_1; \widetilde{\eta} \cap (\Lambda \setminus K_2) \neq \emptyset) \} \\
&\leq& \exp \{ m(\wt{\eta} \cap \omega' = \emptyset; \wt{\eta} \cap (\Lambda \setminus K_2) \neq \emptyset) + m(\wt{\eta} \cap \omega = \emptyset; \wt{\eta} \cap (\Lambda \setminus K_1) \neq \emptyset) \}.
\end{eqnarray*}
We also get a similar lower bound by exchanging the roles of $K_1$ and $K_2$. Therefore, it suffices to show that there exists $C < \infty$ such that for all $l$ and all $\omega \in \Omega_l$,
$$  m(\wt{\eta} \cap B_l \neq \emptyset; \wt{\eta} \cap \omega = \emptyset; \wt{\eta} \cap (\Lambda \setminus B_{nl}) \neq \emptyset) \leq Ce(n).$$

Given an unrooted loop $\widetilde{\eta}$ with representative $\eta = [\eta_0, \ldots, \eta_k]$, let
$$ <\widetilde{\eta}> = \min \{\abs{\eta_i} : i=0, \ldots, k \}.$$ 
Let $\widetilde{\eta}^{*}$ be such that $\abs{\widetilde{\eta}^{*}} = < \wt{\eta} >$ and such that
$$\arg(\widetilde{\eta}^{*}) = \min \{ \arg (\eta_i) : \abs{\eta_i} = <\wt{\eta}> \}.$$

Suppose first that $d = 2$. Then for $j \leq l/2$ and $z \in \partial B_j$,
\begin{eqnarray*}
& & m( \widetilde{\eta} \cap \omega = \emptyset; \widetilde{\eta} \cap (\Lambda \setminus B_{nl}) \neq \emptyset; \lceil <\widetilde{\eta}> \rceil = j; \widetilde{\eta}^{*}= z) \\
&\leq& \Pro{z}{ \sigma_{2j} < \xi_{j-1}} \left( \max_{y \in \partial B_{2j}} \Pro{y}{\sigma_l < \xi_{\omega}} \right) \\
&\times& \left( \max_{v \in \partial B_l} \Pro{v}{\sigma_{nl} < \sigma_{j-1}} \right)\left( \max_{w \in \partial B_{nl}} \Pro{w}{\xi_{z} < \xi_{j-1}} \right)\\
&\leq& C \frac{1}{j} \left( \frac{2j}{l} \right)^{\frac{1}{2}} \left( \frac{\log l - \log (j-1)}{\log{nl} - \log l} \right) \frac{1}{j} \\
 &\leq& C l^{-\frac{1}{2}}(\log n)^{-1}j^{-\frac{3}{2}}\log \frac{l}{j},
\end{eqnarray*}
where the exponent $1/2$ comes from the Beurling estimates (Theorem \ref{beurling}).

If $l/2 < j \leq l$, then
\begin{eqnarray*}
& & m( \widetilde{\eta} \cap \omega = \emptyset, \widetilde{\eta} \cap (\Lambda \setminus B_{nl}) \neq \emptyset, \lceil <\widetilde{\eta}> \rceil = j, \widetilde{\eta}^{*}= z) \\ 
&\leq&  \Pro{z}{ \sigma_{nl} < \xi_{j-1}} \left( \max_{w \in \partial B_{nl}} \Pro{y}{\xi_z < \xi_{j-1}} \right) \\
&\leq& C \left(\frac{\log j - \log (j-1)}{\log{nl} - \log l} \right)j^{-1}
\\
&\leq& C j^{-2}(\log n)^{-1}.
\end{eqnarray*}
Therefore,
\begin{eqnarray*}
& & 
m(\wt{\eta} \cap B_l \neq \emptyset; \wt{\eta} \cap \omega = \emptyset; \wt{\eta} \cap (\Lambda \setminus B_{nl}) \neq \emptyset) \\
&=&  \sum_{j=1}^{l/2} \sum_{z \in \partial B_j} m( \widetilde{\eta} \cap \omega = \emptyset; \widetilde{\eta} \cap (\Lambda \setminus B_{nl}) \neq \emptyset; <\widetilde{\eta}> = j; \widetilde{\eta}_{*}= z) \\
&+& \sum_{j = l/2}^l \sum_{z \in \partial B_j} m( \widetilde{\eta} \cap \omega = \emptyset; \widetilde{\eta} \cap (\Lambda \setminus B_{nl}) \neq \emptyset; <\widetilde{\eta}> = j; \widetilde{\eta}_{*}= z) \\
&\leq& \frac{C}{\log n} \left[ l^{-\frac{1}{2}} \sum_{j=1}^{l/2} j^{-\frac{1}{2}} \log \frac{l}{j} + \sum_{j=l/2}^l \frac{1}{j} \right] \\
&\leq& \frac{C}{\log n}.
\end{eqnarray*}

The case $d \geq 3$ is easier. In this case,
\begin{eqnarray*}
& & m( \widetilde{\eta} \cap \omega = \emptyset, \widetilde{\eta} \cap (\Lambda \setminus B_{nl}) \neq \emptyset, \lceil <\widetilde{\eta}> \rceil = j) \\ 
&\leq&  \left( \max_{z \in \partial B_j} \Pro{z}{ \sigma_{nl} < \xi_{j-1}} \right) \left( \max_{w \in \partial B_{nl}} \Pro{w}{\xi_{j} < \infty} \right) \\
&\leq& C \left(\frac{j^{2-d} - (j-1)^{2-d}}{(nl)^{2-d} - l^{2-d}} \right)\left(\frac{(nl)^{2-d}}{l^{2-d}} \right)
\\
&\leq& C \left(\frac{j^{1-d}}{l^{2-d} - (nl)^{2-d}} \right)n^{2-d}
\\
&\leq& C n^{2-d}l^{d-2}j^{1-d}.
\end{eqnarray*}
Thus,
\begin{eqnarray*}
& & 
m(\wt{\eta} \cap B_l \neq \emptyset; \wt{\eta} \cap \omega = \emptyset; \wt{\eta} \cap (\Lambda \setminus B_{nl}) \neq \emptyset) \\
&=& \sum_{j=1}^{l} m( \widetilde{\eta} \cap \omega = \emptyset, \widetilde{\eta} \cap (\Lambda \setminus B_{nl}) \neq \emptyset, \lceil <\widetilde{\eta}> \rceil = j) \\
&\leq& C n^{2-d} l^{d-2} \sum_{j=1}^l j^{1-d} \\
&\leq& C n^{2-d}.
\end{eqnarray*}
\end{proof}

\begin{cor} \label{infdist} Recall that $\wh{S}$ denotes an infinite LERW. Suppose that $n \geq 4$, $K$ is such that $B_{nl} \subset K$, and $\omega \in \Omega_l$. Then,
\begin{eqnarray*}
\Pro{}{\widehat{S}[0,\widehat{\sigma}_l] = \omega} &=& \left\{ \begin{array}{ll}
\left[1 + O(\frac{1}{\log n})\right] \Pro{}{\widehat{S}^K[0,\widehat{\sigma}_l] = \omega} & d = 2; \\
\left[1 + O(n^{2-d})\right]\Pro{}{\widehat{S}^K[0,\widehat{\sigma}_l] = \omega} & d \geq 3. 
\end{array} \right.
\end{eqnarray*}
In particular,
$$ \Pro{}{\widehat{S}[0,\widehat{\sigma}_l] = \omega} \asymp \Pro{}{\widehat{S}^{4l}[0,\widehat{\sigma}_l] = \omega}.$$
\end{cor}

\begin{proof}
This follows immediately from  Proposition \ref{ax} and the definition of the infinite LERW $\wh{S}$:
$$ \Pro{}{\widehat{S}[0,\widehat{\sigma}_l] = \omega}  = \lim_{m \to \infty} \Pro{}{\widehat{S}^m[0,\widehat{\sigma}_l] = \omega} = \lim_{m \to \infty} \mu_{l,B_m}(\omega).$$
\end{proof}

We conclude this section with the proof that $\eta_1$ and $\eta_2$ are independent (up to constants) for the LERW $\wh{S}^n$.

\begin{prop} \label{indep} Let $4l \leq m \leq n$. Then for any $\omega \in \Omega_l$, $\lambda \in \widetilde{\Omega}_{m,n}$,
\begin{eqnarray*}
& & \Pro{}{\eta^1_l\left(\widehat{S}^n\right) = \omega; \eta^2_{m,n}\left(\widehat{S}^n\right) = \lambda} \\
 &=& \left\{ \begin{array}{ll}
\left[1 + O(\log (m/l)^{-1})\right] \Pro{}{\eta^1_l\left(\widehat{S}^n\right) = \omega} \Pro{}{\eta^2_{m,n}\left(\widehat{S}^n\right) = \lambda} & d = 2; \\
\left[1 + O(\frac{l}{m})\right] \Pro{}{\eta^1_l\left(\widehat{S}^n\right) = \omega} \Pro{}{\eta^2_{m,n}\left(\widehat{S}^n\right) = \lambda} & d \geq 3.  
\end{array} \right.
\end{eqnarray*}
In particular,
$$ \Pro{}{\eta^1_l\left(\widehat{S}^n\right) = \omega; \eta^2_{m,n}\left(\widehat{S}^n\right) = \lambda} \asymp \Pro{}{\eta^1_l\left(\widehat{S}^n\right) = \omega} \Pro{}{\eta^2_{m,n}\left(\widehat{S}^n\right) = \lambda},$$
i.e. $\eta^1$ and $\eta^2$ are ``independent up to constants''.
\end{prop}

\begin{proof}
We fix $l$, $m$ and $n$ throughout and let $\eta^1 = \eta^1_l$, $\eta^2 = \eta^2_{m,n}$.

Let $X$ be a random walk started at $0$ conditioned to leave $B_n$ before returning to $0$. Then $\wh{X}$ and $\wh{S}^n$ have the same distribution. Let $Y$ be a random walk started on $\partial B_n$ according to harmonic measure from $0$ and conditioned to hit $0$ before returning to $\partial B_n$. By reversing paths, for all $z \in \partial B_n$,
$$ \Pro{0}{S(\xi_0 \wedge \sigma_n) = z} = \Pro{z}{S(\xi_0 \wedge \sigma_n) = 0}.$$
Therefore, $X$ and $Y^R$ (the time-reversal of $Y$) have the same distribution. Recall that one obtains the same distribution on LERW by erasing loops from random walks forwards or backwards. Therefore, if $\omega$ and $\lambda$ are as above,
$$ \condPro{}{\eta^1\left(\widehat{S}^n\right) = \omega }{ \eta^2\left(\widehat{S}^n\right) = \lambda} = \condPro{}{\eta^1\left(\widehat{Y}[0,\wh{\xi}_0]\right) = \omega^R }{ \eta^2\left(\widehat{Y}[0, \wh{\xi}_0]\right) = \lambda^R}.$$

Now let $Z$ be a random walk starting at $\lambda_0$, conditioned to hit $0$ before leaving $B_n \setminus \lambda$. Then by the domain Markov property for LERW (Lemma \ref{condit}),
$$ \condPro{}{\eta^1\left(\widehat{Y}[0,\wh{\xi}_0]\right) = \omega^R }{ \eta^2\left(\widehat{Y}[0, \wh{\xi}_0]\right) = \lambda^R}  = \Pro{}{\eta^1(\widehat{Z}[0,\wh{\xi}_{0}]) = \omega^R}.$$

However, by again reversing paths as above, and noting that the loop-erasure of a random walk starting at $0$ and conditioned to avoid $0$ after the first step has the same distribution as the loop-erasure of an unconditioned random walk,
$$ \Pro{}{\eta^1\left(\widehat{Z}[0,\wh{\xi}_{0}]\right) = \omega^R} = \mu_{l,\Lambda \setminus \{\lambda_0\},K}(\omega),$$
where $K = B_n \setminus \{ \lambda_1, \ldots, \lambda_k \}$. 

Let $k = m/l$, $e(k) = (\log k)^{-1}$ if $d = 2$ and $e(k) = k^{-1}$ if $d \geq 3$. Since $k \geq 4$, $B_m \subset K$ and $B_m \subset \Lambda \setminus \{ \lambda_0 \}$, we can apply Propositions \ref{aw} and \ref{ax} to conclude that
\begin{eqnarray*}
\mu_{l,\Lambda \setminus \{\lambda_0\},K}(\omega) &=& [1 + O(e(k))] \mu_{l, B_n \setminus \lambda}(\omega) \\
&=& [1 + O(e(k))] \mu_{l,n}(\omega) \\
&=& \left[1 + O(e(k))\right] \Pro{}{\eta^1\left(\widehat{S}^n\right) = \omega }.
\end{eqnarray*}
\end{proof}

\subsection{The separation lemma} \label{separation}

Throughout this section $S$ will be a random walk and $\wh{S}$ will be an independent infinite LERW. Let $\mathcal{F}_k$ denote the $\sigma$-algebra generated by
$$ \{ S_n : n \leq \sigma_k \} \cup \{\wh{S}_n: n \leq \wh{\sigma}_k \}.$$
For positive integers $j$ and $k$, let $A_k$ be the event
$$ A_k = \{ S[1, \sigma_k] \cap \wh{S}[0, \wh{\sigma}_k] = \emptyset \},$$
$D_k$ be the random variable
$$ D_k = k^{-1} \min \{\operatorname{dist}(S(\sigma_k), \wh{S}[0, \wh{\sigma}_k]), \operatorname{dist}(\wh{S}(\wh{\sigma}_k), S[0, \sigma_k]) \},$$
and $T_j^k$ be the integer valued random variable
$$ T_j^k = \min \{l \geq k  : D_l \geq 2^{-j} \}.$$

The goal of this section is to prove the following separation lemma which states that, conditioned on the event $A_k$ that the random walk $S$ and the infinite LERW $\wh{S}$ do not intersect up to the circle of radius $k$, the probability that they are further than some fixed distance apart from each other at the circle of radius $k$ ($D_k \geq c_1$) is bounded from below by a constant $c_2 > 0$. 

\begin{thm} [Separation Lemma] \label{sep}
There exist constants $c_1, c_2 > 0$ such that for all $k$,
$$ \condPro{}{D_k \geq c_1}{A_k} \geq c_2.$$
\end{thm}

The proof of Theorem \ref{sep} depends on two lemmas. Lemma \ref{ope} roughly states that the probability that $S$ and $\wh{S}$ stay close together without intersecting each other is very small. More precisely, the probability that $T_{j-1} \geq (1+c j^2 2^{-j})T_j$ and that the paths don't intersect is less than $2^{- \beta j^2}$. Lemma \ref{lfg} states that if $S$ and $\wh{S}$ are separated, then there is a substantial probability that they stay separated and don't intersect. To wit, if $\{ T_j > k \}$ and $A_{T_j}$ hold, then the probability that $A_{2k}$ and $\{ D_{2k} \geq 2^{-j} \}$ hold is greater than $2^{-\alpha j}$. The proof of the separation lemma then combines the two lemmas to show that 
$$ \condPro{}{T_{j-1} \leq (1+c j^2 2^{-j})T_j}{A_{2k}} \geq 1 - 2^{\alpha j - \beta j^2}.$$ 
Since 
$$ \prod_{j=1}^\infty (1+cj^2 2^{-j}) < \infty,$$ 
and 
$$ \prod_{j=j_0}^\infty (1 - 2^{\alpha j - \beta j^2}) > 0,$$
then conditioned on $A_{2k}$, there is a probability bounded below that $S$ and $\wh{S}$ separate to some fixed distance before leaving the ball of radius $2k$ no matter how close the two paths were upon leaving the ball of radius $k$. 

\begin{lemma} \label{ope}
For all $c > 0$, there exists $\beta = \beta(c) > 0$ and $j_0(c)$ such that for all $j \geq j_0$ and all $k$,
$$ \condPro{}{T_{j-1}^{k} \geq (1 + c j^2 2^{-j})T_{j}^k; A_{2k} }{ \mathcal{F}_{T_j^{k}}} \ind \left\{ T_j^{k} \leq \frac{3k}{2} \right\}\leq 2^{-\beta j^2}.$$
\end{lemma}

\begin{proof}
We let $j_0$ be such that for all $j \geq j_0$, $cj^2 2^{-j} < 1/2$.

Since $k$ is fixed we will write $T_j$ for $T_j^k$ from now on. We suppose that $\wh{S}[0, \wh{\sigma}(T_j)]$ and $S[0, \sigma(T_j)]$ are any paths such that $T_j \leq \frac{3k}{2}$ holds. We also assume that $D_{T_j} < 2^{-j + 1}$ or else there is nothing to prove.

Now consider $K := \wh{S}[0, \wh{\sigma}((1 + c j^2 2^{-j})T_j)]$ and let
$$ \rho = \inf \{ n \geq \sigma(T_j) : \operatorname{dist}(S_n, K) \leq 2^{-j + 1} \abs{S_n} \}.$$
Notice that even though we assume that $D_{T_j} < 2^{-j +1}$, $\rho$ is not necessarily equal to $\sigma(T_j)$.

If $\rho > \sigma((1 + 4 \cdot 2^{-j})T_j)$ then this means that $T_{j - 1} < (1 + 4 \cdot 2^{-j})T_j$.
However, if $\rho \leq \sigma((1 + 4 \cdot 2^{-j})T_j)$, then by the Beurling estimates for random walk (Theorem \ref{beurling}), there exists $c' < 1$ such that
$$\Pro{}{S[\rho, \sigma((1 + 8 \cdot 2^{-j})T_j)] \cap K = \emptyset} \leq c'.$$

The same estimate will hold starting at $T_j + 8k2^{-j}$, $k = 0, 1, \ldots, \lfloor c j^2/8 \rfloor$. Therefore,
\begin{eqnarray*}
& & \condPro{}{T_{j-1} \geq (1 + c j^2 2^{-j})T_{j}; A_{2k} }{ \mathcal{F}_{T_j}} \ind \left\{T_j \leq \frac{3k}{2} \right\} \\
&\leq& \condPro{}{T_{j-1} \geq (1 + c j^2 2^{-j})T_{j}; A_{T_j + cj^2 2^{-j}} }{ \mathcal{F}_{T_j}} \ind \left\{T_j \leq \frac{3k}{2} \right\} \\
&\leq& (c')^{\lfloor c j^2/8 \rfloor} = 2^{-\beta j^2}.
\end{eqnarray*}

\end{proof}

\begin{lemma} \label{lfg}
There exists $ \alpha < \infty$ and $c > 0$ such that for all $j$ and $k$,
$$ \condPro{}{A_{2k}; D_{2k} \geq c2^{-j} }{ \mathcal{F}_{T_j^{k}}} \geq c 2^{-\alpha j} \ind(A_{T_j^{k}}).$$
\end{lemma}

\begin{proof}
Since $k$ is fixed, we will omit the superscript $k$ from now on. Let $z_1 = S(\sigma_{T_j})$ and $z_2 = \wh{S}(\wh{\sigma}_{T_j})$. Without loss of generality, we may assume that $T_j < 2k$ (or else there is nothing to prove) and also that $\arg(z_2) < \arg(z_1)$. Note that $\abs{z_1} = \abs{z_2} = T_j$ and $k \leq T_j \leq 2k$.

Suppose that $A_{T_j}$ holds. By definition of $T_j$, there exists $c > 0$ and half-wedges
$$W_1 = \{ z : (1- c 2^{-j})T_j \leq \abs{z} \leq (1 + c 2^{-j}) T_j, -c2^{-j} \leq \arg(z) - \arg(z_1) \leq c2^{-j} \} $$
and
$$ W_2 = \{ z : (1 - c 2^{-j})T_j \leq \abs{z} \leq (1 + c 2^{-j})T_j, -c2^{-j} \leq \arg(z) - \arg(z_2) \leq c2^{-j} \}$$
such that $S[0,\sigma_{T_j}] \cap W_2 = \emptyset$, $\wh{S}[0, \wh{\sigma}_{T_j}] \cap W_1 = \emptyset$, and $\dist(W_1,W_2) \geq c2^{-j}T_j$.

Using Lemma \ref{condit} and Proposition \ref{dirichlet}, it is easy to verify that there exists a global constant $c'$ such that
$$ \Pro{z_1} { \abs{S(\sigma_{W_1})} \geq (1 + c2^{-j})T_j} \geq c',$$
and
$$ \Pro{z_2}{\abs{\wh{S}(\wh{\sigma}_{W_2})} \geq (1 + c2^{-j})T_j} \geq c'.$$

Now consider the half-wedges
$$ W_1' = \{ z : T_j \leq \abs{z} \leq 2k, -\frac{3c}{2}2^{-j} \leq \arg(z) - \arg(z_1) \leq \frac{\pi}{6} \} $$
and
$$ W_2' = \{ z : T_j \leq \abs{z} \leq 2k,\frac{\pi}{6} \leq \arg(z) - \arg(z_1) \leq  \frac{3c}{2}2^{-j} \}.$$
Applying Lemma \ref{condit} and Corollary \ref{dfr} to $W_1'$ and $W_2'$,
one obtains that for any $z_1' \in \partial W_1$ such that $\abs{z_1'} \geq (1 + c2^{-j})T_j$
$$ \Pro{z_1'} { \abs{S(\sigma_{W_1'})} = 2k} \geq c'2^{-\alpha j},$$
and for any $z_2' \in \partial W_2$ such that $\abs{z_2'} \geq (1 + c2^{-j})T_j$
$$ \Pro{z_2'} { \abs{\wh{S}(\wh{\sigma}_{W_2'})} = 2k} \geq c'2^{- \alpha j},$$
The result then follows since $W_1'$ and $W_2'$ are distance $c2^{-j}T_j$ apart and $S$ and $\wh{S}$ are independent.

\end{proof}

\begin{proof} [Proof of Theorem \ref{sep}]
We again fix $k$ and let $T_j = T_j^{k/2}$. Let 
$$s = \prod_{j=1}^\infty (1 + c j^2 2^{-j})$$ 
where $c$ is chosen so that $s \leq 3/2$. We also let $j_0$ be such that for $j \geq j_0$, $2^{-\beta j^2 + \alpha j} < 1$, where $\alpha$ and $\beta = \beta(c)$ are as in Lemmas \ref{ope} and \ref{lfg}.

To prove the theorem, it suffices to show that for all $m$,
$$\condPro{}{D_k \geq c_1}{ A_k ; 2^{-m} \leq D_{k/2} < 2^{-m + 1} } \geq c_2.$$
By Lemma \ref{lfg}, it is enough to find a constant $c_2'$ such that
$$ \condPro{}{T_{j_0} \leq 3k/4 }{ A_k ; 2^{-m} \leq D_{k/2} < 2^{-m + 1} } \geq c_2'.$$
In fact, we will show that
$$ \condPro{}{ T_{j_0} \leq ks/2 }{ A_k ; 2^{-m} \leq D_{k/2} < 2^{-m + 1} } \geq c_2'.$$

Let 
$$B_m = \{ 2^{-m} \leq D_{k/2} < 2^{-m + 1} \},$$
and
$$C_{j} = \{T_{j-1} \leq (1 + c j^2 2^{-j})T_j \}.$$  Then,
\begin{eqnarray*}
\condPro{}{T_{j_0} \leq 3k/4}{A_k ; B_m} &\geq& \condProb{}{ \bigcap_{j=j_0 + 1}^m C_j }{  A_k; B_m   }  \\
&=& \prod_{j=j_0 + 1}^m \condProb{}{ C_j }{  A_k ; B_m ; \bigcap_{l=j + 1}^m C_l } \\
&=& \prod_{j=j_0 + 1}^m \condProb{}{ C_j }{  A_k; B_m; A_{T_j}; \bigcap_{l=j + 1}^m C_l } \\
&=& \prod_{j=j_0 + 1}^m \left[1 - \frac{\condProb{}{C_j^c ; A_k }{ B_m; A_{T_j}; \bigcap_{l=j + 1}^m C_l }}{\condProb{}{A_k }{ B_m ; A_{T_j} ; \bigcap_{l=j + 1}^m C_l }} \right].
\end{eqnarray*}

However,
$$  B_m \cap A_{T_j} \cap \bigcap_{l=j + 1}^m C_l \in \mathcal{F}_{T_j},$$
$$ B_m \cap A_{T_j} \cap \bigcap_{l=j + 1}^m C_l  \subset \{T_j \leq \frac{3k}{2} \}, $$
and
$$ B_m \cap A_{T_j} \cap \bigcap_{l=j + 1}^m C_l  \subset A_{T_j}.$$
Therefore, by Lemmas \ref{ope} and \ref{lfg},
$$ \condPro{}{T_{j_0} \leq 3k/4 }{ A_k ; B_m} \geq \prod_{j=j_0 + 1}^m (1 - 2^{-\beta j^2 + \alpha j}) \geq \prod_{j=j_0 + 1}^\infty (1 - 2^{-\beta j^2 + \alpha j}) = c_2 > 0.$$

\end{proof}

Using the same techniques, one can prove a ``reverse'' separation lemma. Let $\wt{S}$ be a random walk started uniformly on the circle $\partial B_n$ and conditioned to hit $0$ before leaving $B_n$.  Let $X$ be the time reversal of $\wh{S}^n$ (so that $X$ is also a process from $\partial B_n$ to $0$). As before, for $k \leq n$, let 
$$ \wt{A}_k = \{ \wt{S}[0,\wt{\sigma}_k] \cap X[0,\wh{\sigma}_k] = \emptyset \},$$
$$ \wt{D}_k = k^{-1} \min \{ \operatorname{dist}(\wt{S}(\wt{\sigma}_k), X[0,\wh{\sigma}_k]), \operatorname{dist}(X(\wh{\sigma}_k), \wt{S}[0, \wt{\sigma}_k]) \}.$$
Then,

\begin{thm} [Reverse Separation Lemma] \label{sep2}
There exists $c_1, c_2 > 0$ such that
$$\condPro{}{\wt{D}_k \geq  c_1}{ \wt{A}_k} \geq c_2.$$
\end{thm}

\section{The growth exponent} \label{intersection}

\subsection{Introduction}

Recall that $W_t$ denotes standard complex Brownian motion and $\gamma$ denotes radial $\SLE_2$ in $\Disc$ started uniformly on $\partial \Disc$.

In this chapter we will consider random walks and independent LERWs. We will view them as being defined on different probability spaces so that $\whPro{}{.}$ and $\whExp{}{.}$ denote probabilities and expectations with respect to the LERW, while $\Pro{}{.}$ and $\Exp{}{.}$ will denote probabilities and expectations with respect to the random walk. For $m \leq n$, we define $\Es(m,n)$, $\Es(n)$ and $\wt{\Es}(n)$ as follows.
\begin{eqnarray*}
\Es(m,n) &=& \whExp{}{\Pro{}{S[1,\sigma_n] \cap \eta^2_{m,n}(\wh{S}^{n}[0, \wh{\sigma}_n]) = \emptyset}},\\
\Es(n) &=&  \whExp{}{\Pro{}{S[1,\sigma_n] \cap \wh{S}^n[0, \wh{\sigma}_n] = \emptyset}} = \Es(0,n),\\
\wt{\Es}(n) &=& \whExp{}{\Pro{}{S[1,\sigma_n] \cap \wh{S}[0, \wh{\sigma}_n] = \emptyset}}.
\end{eqnarray*}
$\Es(m,n)$ is the probability that a random walk from the origin to $\partial B_n$ and the terminal part of an independent LERW from $m$ to $n$ do not intersect.
$\Es(n)$ is the probability that a random walk from the origin to $\partial B_n$ and the loop-erasure of an independent random walk from the origin to $\partial B_n$ do not intersect. $\wt{\Es}(n)$ is the probability that a random walk from the origin to $\partial B_n$ and an \emph{infinite} LERW from the origin to $\partial B_n$ do not intersect.

In section \ref{decompsec}, we prove that for $m < n$, $\Es(n)$ can be decomposed as
$$ \Es(n) \asymp \Es(m) \Es(m,n).$$
In section \ref{escapesec}, we use the convergence of LERW to $\SLE_2$ (Theorem \ref{LERWtoSLE}) and the intersection exponent $3/4$ for $\SLE_2$ (Theorem \ref{harmeasure}) to show that 
$$\Es(m,n) \asymp \left( \frac{m}{n} \right)^{3/4}.$$ We then combine these two results to show that $\Es(n) \approx n^{-3/4}$. Finally, in section \ref{growthsec}, we show how the fact that $\Es(n) \approx n^{-3/4}$ implies that $\Gr(n) \approx n^{5/4}$.

Before proceeding, we prove the following lemma which shows that $\wt{\Es}(n)$ and $\Es(4n)$ are on the same order of magnitude.

\begin{lemma} \label{2escapes}
$$ \wt{\Es}(n) \asymp \Es(4n).$$
\end{lemma}

\begin{proof}
By Corollary \ref{infdist}, it suffices to show that
$$ \whExp{}{\Pro{}{S[1,\sigma_n] \cap \wh{S}^{4n}[0, \wh{\sigma}_n] = \emptyset}} \asymp \whExp{}{\Pro{}{S[1,\sigma_{4n}] \cap \wh{S}^{4n}[0, \wh{\sigma}_{4n}] = \emptyset}}.$$
It is clear that the left hand side is greater than or equal to the right hand side. To prove the other direction, we will use the separation lemma (Theorem \ref{sep}). Given a point $z \in \partial B_n$, let $W(z)$ be the half-wedge
$$ W(z) = \{ w \in \Lambda : (1-c_1)n \leq \abs{w} \leq 4n, \abs{\arg(w) - \arg(z)} < c_1/2\},$$
where $c_1$ is as in the statement of the separation lemma. We also let
$$ A_n = \{S[1,\sigma_n] \cap \wh{S}^{4n}[0, \wh{\sigma}_n] = \emptyset\},$$
$$ z_0 = \wh{S}^{4n}(\wh{\sigma}_n),$$
and
$$ D_n = n^{-1} \min \{\operatorname{dist}(S(\sigma_n), \wh{S}^{4n}[0, \wh{\sigma}_n]), \operatorname{dist}(z_0, S[0, \sigma_n]) \}.$$

By the strong Markov property for random walk,
\begin{eqnarray*}
& & \whExp{}{\Pro{}{S[1,\sigma_{4n}] \cap \wh{S}^{4n}[0, \wh{\sigma}_{4n}] = \emptyset}} \\
&\geq& c \whExp{}{\ind \{\wh{S}^{4n}[\wh{\sigma}_n, \wh{\sigma}_{4n}] \subset W(z_0) \} \Pro{}{A_n; D_n \geq c_1}}.
\end{eqnarray*}
By Lemma \ref{condit} and Corollary \ref{dfr},
$$ \whExp{}{\ind \{\wh{S}^{4n}[\wh{\sigma}_n, \wh{\sigma}_{4n}] \subset W(z_0) \} \Pro{}{A_n; D_n \geq c_1}} \geq c \whExp{}{\Pro{}{A_n; D_n \geq c_1}}.$$
Finally, by the separation lemma,
$$ \whExp{}{\Pro{}{A_n; D_n \geq c_1}} \geq c \whExp{}{\Prob{}{A_n}},$$
and therefore,
$$ \whExp{}{\Pro{}{S[1,\sigma_{4n}] \cap \wh{S}^{4n}[0, \wh{\sigma}_{4n}] = \emptyset}} \geq c \whExp{}{\Pro{}{S[1,\sigma_n] \cap \wh{S}^{4n}[0, \wh{\sigma}_n] = \emptyset}}.$$
\end{proof}

\subsection{Proof that $\Es(n) \asymp \Es(m)\Es(m,n)$} 
\label{decompsec}

\begin{prop} \label{d} There exists $C < \infty$ such that for all $m$ and $n$ with $m \leq n$,
$$ \Es(n) \leq C \Es(m) \Es(m,n).$$
\end{prop}

\begin{proof}
Let $l = \lfloor m/4 \rfloor$ and fix $\eta^1= \eta^1_l$ and $\eta^2 = \eta^2_{m,n}$.

For any path $\eta$ in $\Omega_n$,
\begin{eqnarray*}
& & \Pro{}{S[1,\sigma_n] \cap \eta = \emptyset} \\
&\leq& \Pro{}{S[1,\sigma_l] \cap \eta^1(\eta) = \emptyset; S[1,\sigma_n] \cap \eta^2(\eta) = \emptyset} \\
&=& \sum_{z \in \partial B_l} \Pro{}{ S[1,\sigma_l] \cap \eta^1(\eta) = \emptyset; S(\sigma_l) = z}\Pro{z}{S[1,\sigma_n] \cap \eta^2(\eta) = \emptyset}.
\end{eqnarray*}
However, $\eta^2(\eta) \subset \Lambda \setminus B_{m}$, and thus by the discrete Harnack principle, for any $z,z' \in \partial B_l$,
$$ \Pro{z}{S[1,\sigma_n] \cap \eta^2(\eta) = \emptyset} \asymp \Pro{z'}{S[1,\sigma_n] \cap \eta^2(\eta) = \emptyset}.$$
Therefore,
$$ \Pro{}{S[1,\sigma_n] \cap \eta = \emptyset} \leq C \Pro{}{S[1,\sigma_{l}] \cap \eta^1(\eta) = \emptyset} \Pro{}{S[1,\sigma_n] \cap \eta^2(\eta) = \emptyset}.$$

We now let $\eta = \wh{S}^n[0,\wh{\sigma}_n]$. By Proposition \ref{indep}, for any $\omega \in \Omega_l$, $\lambda \in \widetilde{\Omega}_{m,n}$,
\begin{eqnarray*}
& & \Pro{}{\eta^1(\widehat{S}^n[0, \widehat{\sigma}_n]) = \omega; \eta^2(\widehat{S}^n[0, \widehat{\sigma}_n]) = \lambda} \\
&\asymp& \Pro{}{\eta^1(\widehat{S}^n[0, \widehat{\sigma}_n]) = \omega} \Pro{}{\eta^2(\widehat{S}^n[0, \widehat{\sigma}_n]) = \lambda}.
\end{eqnarray*}
Therefore,
\begin{eqnarray*}
\Es(n) 
&=& \whExp{}{\Pro{}{S[1,\sigma_n] \cap \wh{S}^n[0,\wh{\sigma}_n] = \emptyset}} \\
&\leq& C \whExp{}{\Pro{}{S[1,\sigma_{l}] \cap \eta^1(\wh{S}^n[0,\wh{\sigma}_n]) = \emptyset} \Pro{}{S[1,\sigma_n] \cap \eta^2(\wh{S}^n[0,\wh{\sigma}_n]) = \emptyset}}\\
&\leq& C \whExp{}{ \Pro{}{S[1,\sigma_l] \cap \wh{S}^n[0,\wh{\sigma}_l] = \emptyset}} \whExp{}{\Pro{}{S[1,\sigma_n] \cap \eta^2(\wh{S}^n[0,\wh{\sigma}_n]) = \emptyset}}\\
&=& C \whExp{}{ \Pro{}{S[1,\sigma_l] \cap \wh{S}^n[0,\wh{\sigma}_l] = \emptyset}} \Es(m,n).
\end{eqnarray*}
By corollary \ref{infdist}, since $4l \leq n$, 
$$ \whExp{}{ \Pro{}{S[1,\sigma_l] \cap \wh{S}^n[0,\wh{\sigma}_l] = \emptyset}} \asymp \whExp{}{ \Pro{}{S[1,\sigma_l] \cap \wh{S}[0,\wh{\sigma}_l] = \emptyset}} = \wt{\Es}(l).$$
Finally, by Lemma \ref{2escapes}, $\wt{\Es}(l) \asymp \Es(m)$, which finishes the proof of the proposition.

\end{proof}

\begin{prop} \label{e}
There exists $c > 0$ such that for all $m$ and $n$ with $m \leq n/2$,
$$ \Es(n) \geq c \Es(m) \Es(m, n).$$
\end{prop}

\begin{proof}
We will use the following abbreviations. Let $l = \lfloor m/4 \rfloor$ and let
\begin{eqnarray*}
\eta^1 &=& \eta^1_l(\wh{S}^n[0, \wh{\sigma}_n]);\\
\eta^2 &=& \eta^2_{m,n}(\wh{S}^n[0, \wh{\sigma}_n]);\\
\eta^* &=& \eta^*_{l,m,n}(\wh{S}^n[0, \wh{\sigma}_n]).
\end{eqnarray*}
Then $\wh{S}^n[0, \wh{\sigma}_n] = \eta^1 \oplus \eta^* \oplus \eta^2$. We also decompose $S[1, \sigma_n]$ into $S^1 = S[1,\sigma_{2l}]$ and $S^2 = S[\sigma_{2l}+1, \sigma_n]$.

Let $c_1$ be as in the statement of the separation lemma (Theorem \ref{sep}). Let $W$ and $W^*$ be the half-wedges
\begin{eqnarray*}
W = \{ z : (1 - \frac{c_1}{4}) l \leq \abs{z} \leq (1 + \frac{c_1}{4})m, \abs{\arg(z)} \leq \frac{c_1}{4} \}; \\
W^* = \{ z : (1 - \frac{c_1}{4}) l  \leq \abs{z} \leq (1 + \frac{c_1}{4})m, \abs{\arg(z)} \leq \frac{c_1}{2} \}.
\end{eqnarray*}
and let $A = B_l \cup W$.

Let $K_1$ be the set of $\eta^1$ such that
$$ \eta^1 \cap \partial B_l \subset \{ z : \arg(z) \in (-\frac{c_1}{8}, \frac{c_1}{8}) \},$$ and $K_2$ be the set of $\eta^2$ such that
$$ \eta^2 \cap \partial B_{m} \subset \{ z : \arg(z) \in (-\frac{c_1}{8}, \frac{c_1}{8}) \}.$$
Then,
\begin{eqnarray*}
& & 
\Es(n) \\
&=& \whExp{}{\Pro{}{S[1,\sigma_n] \cap \widehat{S}^n[0,\wh{\sigma}_n] = \emptyset}}\\
&=& \whExp{}{ \Pro{}{S^1 \cap \eta^1 \oplus \eta^* = \emptyset;  S^2 \cap \eta^1 \oplus \eta^* \oplus \eta^2 = \emptyset} } \\
&\geq& \whExp{}{ \ind_{\{\eta^1 \in K_1\}} \ind_{\{\eta^2 \in K_2\}} \ind_{\{\eta^* \subset W \}} \Pro{}{S^1 \cap (\eta^1 \cup W^*) = \emptyset;  S^2 \cap (\eta^2 \cup A) = \emptyset }}\\
&=& \whExp{} {\ind_{\{\eta^1 \in K_1\}}  \Pro{}{S^1 \cap (\eta^1 \cup W^*) = \emptyset} \right. \\ 
 &\times&  \left. \ind_{\{\eta^2 \in K_2\}} \condPro{}{S^2 \cap (\eta^2 \cup A) = \emptyset }{ S^1 \cap (\eta^1 \cup W^*) = \emptyset} \ind_{\{\eta^* \subset W\}} } 
\end{eqnarray*}
Therefore,
$$ \Es(n) \geq \whExp{}{ X(\eta^1) Y(\eta^2) \ind_{\{\eta^* \subset W\}} },$$
where
$$ X(\eta^1) = \ind_{\{\eta^1 \in K_1\}}  \Pro{}{S^1 \cap (\eta^1 \cup W^*) = \emptyset},$$
and
$$ Y(\eta^2) = \ind_{\{\eta^2 \in K_2\}} \inf_{z \in \partial B_{2l} \setminus W^*} \Pro{z}{S[1,\sigma_n] \cap (\eta^2 \cup A) = \emptyset}.$$

By Lemma \ref{condit} and Corollary \ref{wedgecond}, for any $\omega_1 \in K_1$ and $\omega_2 \in K_2$,
$$ \condwhPro{} { \eta^* \subset W }{ \eta^1 = \omega_1, \eta^2 = \omega_2 } \geq c,$$
and therefore
\begin{eqnarray*}
\Es(n) \geq  c \whExp{} { X(\eta^1) Y(\eta^2) }.
\end{eqnarray*}
However, by Proposition \ref{indep}, $\eta^1$ and $\eta^2$ are independent up to constants, and therefore,
$$ \Es(n) \geq c \whExp{}{ X(\eta^1)} \whExp {}{Y(\eta^2) }.$$
To prove the Proposition, it therefore suffices to show that
\begin{eqnarray}
\whExp{} {X(\eta^1)} \geq c \Es(m), \label{dsp}
\end{eqnarray}
and
\begin{eqnarray}
\whExp{}{ Y(\eta^2) } \geq c \Es(m,n). \label{opg}
\end{eqnarray}

To prove (\ref{dsp}), note that
\begin{eqnarray*}
\whExp{} {X(\eta^1)} &=& \whExp{} {\ind_{\{\eta^1 \in K_1\}}  \Pro{} {S^1 \cap (\eta^1 \cup W^*) = \emptyset }} \\
&\geq& c \whExp{} { \Pro{}{S^1 \cap (\eta^1 \cup W^*) = \emptyset}} \\
&\geq& c \whExp{} { \Pro{}{S[1,\sigma_l] \cap \eta^1 = \emptyset; \dist(S(\sigma_l), \eta^1) \geq c l; S[1,\sigma_l] \cap W^* = \emptyset}} \\
&\geq& c \whExp{} {\Pro{}{S[1,\sigma_l] \cap \eta^1 = \emptyset}},
\end{eqnarray*}
where the last inequality is justified by the separation lemma (Theorem \ref{sep}).
However, by Corollary \ref{infdist} and Lemma \ref{2escapes},
$$ \whExp{} {\Pro{}{S[1,\sigma_l] \cap \eta^1 = \emptyset}} \asymp \whExp{} {\Pro{}{S[1,\sigma_l] \cap \wh{S}[0,\wh{\sigma}_l] = \emptyset}} = \wt{\Es}(l) \asymp \Es(m).$$

We now prove (\ref{opg}). Since $\eta^2 \subset \Lambda \setminus B_{m}$, by the discrete Harnack inequality, for any $z_1, z_2 \in \partial B_{2l} \setminus W^*$,
$$ \Pro{z_1}{S[1,\sigma_n] \cap (\eta^2 \cup A) = \emptyset} \asymp \Pro{z_2}{S[1, \sigma_n] \cap (\eta^2 \cup A) = \emptyset}.$$
Therefore, fixing a $z \in \partial B_{2l} \setminus W^*$,
$$ Y(\eta^2) \geq c \ind_{\{\eta^2 \in K_2\}} \Pro{z}{S[1,\sigma_n] \cap (\eta^2 \cup A) = \emptyset}.$$

By Lemma \ref{rwdecomp},
\begin{eqnarray*}
& & 
\Pro{z}{S[1,\sigma_n] \cap (\eta^2 \cup A) = \emptyset} \\
&=& \frac{G(z; B_n \setminus (\eta^2 \cup A))}{G(z; B_n)} \sum_{y \in \partial B_n} \condPro{y}{\xi_{z} < \xi_A \wedge \xi_{\eta^2} }{ \xi_{z} < \sigma_n} \Pro{z}{S(\sigma_n) = y}.
\end{eqnarray*}
For any $y \in \partial B_n$,
\begin{eqnarray*}
& & \condPro{y}{\xi_{z} < \xi_A \wedge \xi_{\eta^2} }{ \xi_{z} < \sigma_n} \\
&\geq&  \sum_{w \in \partial B_{m} \setminus W^*} \condPro{w}{\xi_{ z } < \xi_A \wedge \xi_{\eta^2} }{ \xi_{z} < \sigma_n} \condPro{y}{\xi_{w} < \xi_{W^*} \wedge \xi_{\eta^2} }{ \xi_{z}< \sigma_n}
\end{eqnarray*}
For any $w \in \partial B_{m} \setminus W^*$,
$$ \condPro{w}{\xi_{B_{cl/4}(z)} < \xi_A \wedge \xi_{\eta^2} }{ \xi_{z} < \sigma_n} \geq c.$$
Furthermore, by Lemma \ref{rebv}, for any $u \in \partial B_{cl/4}(z)$,
$$ \frac{\condPro{u}{\xi_{z} < \xi_A \wedge \xi_{\eta^2} }{ \xi_{z} < \sigma_n}}{\condPro{u}{\xi_{z} < \xi_{\eta^2} }{ \xi_{z} < \sigma_n}} \geq \condPro{u}{\xi_z < \xi_{B_{cl/2}(z)} }{ \xi_{z} < \xi_{\eta^2} \wedge \sigma_n}  \geq c.$$
It also follows from (\ref{adhx}) in Lemma \ref{rebv} that for any two paths $\eta^2, \wt{\eta}^2 \in \wt{\Omega}_{m,n}$  and any $u \in \partial B_{cl/4}(z)$, 
$$ \condPro{u}{\xi_{z} < \xi_{\eta^2} }{ \xi_{z} < \sigma_n} \asymp \condPro{u}{\xi_{z} < \xi_{\wt{\eta}^2} }{ \xi_{z} < \sigma_n},$$
and therefore there exists $f(n,m)$ such that for all $\eta^2 \in \wt{\Omega}_{m,n}$ and $u \in \partial B_{cl/4}(z)$,
$$ \Pro{u}{\xi_{z} < \xi_{\eta^2} | \xi_{z} < \sigma_n} \asymp f(n,m).$$

Thus,
\begin{eqnarray*}
& & \Pro{z}{S[1,\sigma_n] \cap (\eta^2 \cup A) = \emptyset} \\ 
&\geq& c f(n,m) \frac{G(z; B_n \setminus (\eta^2 \cup A))}{G(z; B_n)}  \\
&\times& \sum_{y \in \partial B_n} \Pro{y}{\xi_m < \xi_{\eta^2} \wedge \xi_{W^*} | \xi_{z} < \sigma_n} \Pro{z}{S(\sigma_n) = y}.
\end{eqnarray*}

Let $r_1 = \dist(z, \eta^2 \cup \partial B_n)$ and $r_2 = \dist(z, \eta^2 \cup A \cup \partial B_n)$. Then $r_2 > c_1 l > c_1 r_1$. Therefore,
$$G(z; B_n \setminus (\eta^2 \cup A)) \geq G(z; B_{r_2}(z)) \geq c G(z; B_l(z))$$
and by Lemma \ref{green} applied to the ball $B(z, r_1)$,
$$ G(z; B_l(z)) \geq c G(z; B_{r_1}(z)) \geq c G(z; B_n \setminus \eta^2).$$

Finally, by the reverse separation lemma (Theorem \ref{sep2}),
$$ \whExp{}{\Pro{y}{\xi_m < \xi_{\eta^2} \wedge \xi_{W^*} | \xi_{z} < \sigma_n}} \geq c \whExp{}{\Pro{y}{\xi_m < \xi_{\eta^2} | \xi_{z} < \sigma_n}},$$
and thus
\begin{eqnarray*}
& & \whExp{}{Y(\eta^2)} \\
&\geq& c \whExp{}{ \ind_{\{\eta^2 \in K_2\}} \Pro{z}{S[1,\sigma_n] \cap (\eta^2 \cup A) = \emptyset}} \\
&\geq& c \whExp{}{\Pro{z}{S[1,\sigma_n] \cap (\eta^2 \cup A) = \emptyset}} \\
&\geq& c \frac{G(z; B_l)}{G(z; B_n)} f(n,m) \whExp{}{ \sum_{y \in \partial B_n} \Pro{y}{\xi_{m} < \xi_{\eta^2} \wedge \xi_{W^*} | \xi_{z} < \sigma_n} \Pro{z}{S(\sigma_n) = y}} \\
&\geq& c \whExp{}{ \frac{G(z; B_n \setminus \eta^2)}{G(z; B_n)} f(n,m)  \sum_{y \in \partial B_n} \Pro{y}{\xi_{m} < \xi_{\eta^2} | \xi_{z} < \sigma_n} \Pro{z}{S(\sigma_n) = y} } \\
&\geq& c \whExp{}{\frac{G(z; B_n \setminus \eta^2)}{G(z; B_n)} \sum_{y \in \partial B_n} \Pro{y}{\xi_z < \xi_{\eta^2} | \xi_{z} < \sigma_n} \Pro{z}{S(\sigma_n) = y}}.
\end{eqnarray*}
However, by applying Lemma \ref{rwdecomp} again,
\begin{eqnarray*}
& & \frac{G(z; B_n \setminus \eta^2)}{G(z; B_n)} \sum_{y \in \partial B_n} \Pro{y}{\xi_z < \xi_{\eta^2} | \xi_{z} < \sigma_n} \Pro{z}{S(\sigma_n) = y} \\
&=& \Pro{z}{S[1, \sigma_n] \cap \eta^2 = \emptyset}.
\end{eqnarray*}
and thus
$$ \whExp{}{Y(\eta^2)} \geq c \Es(m,n).$$
\end{proof}

\subsection{Intersection exponents for $\SLE_2$ and LERW} \label{escapesec}
In this section, we use the convergence of LERW to $\SLE_2$ to show that for $0 < r < 1$, $\Es(rn,n) \asymp r^{3/4}$. We combine this result with the decomposition 
$$\Es(n) \asymp \Es(rn)\Es(rn,n)$$ from the previous section to obtain that $\Es(n) \approx n^{-3/4}$.

We recall the notation introduced in Section \ref{sSLE}. Let $\Gamma$ denote the set of continuous curves $\alpha:[0,t_\alpha] \to \overline{\Disc}$ (we allow $t_\alpha$ to be $\infty$) such that $\alpha(0) \in \partial \Disc$, $\alpha(0,t_\alpha] \subset \Disc$ and $\alpha(t_\alpha) = 0$. We can make $\Gamma$ into a metric space as follows. If $\alpha, \beta \in \Gamma$, we let
$$ d(\alpha, \beta) =  \inf \sup_{0 \leq t \leq t_{\alpha}}  \abs{\alpha(t) - \beta(\theta(t))},$$
where the infimum is taken over all continuous, increasing bijections $\theta:[0,t_{\alpha}] \to [0,t_{\beta}]$. Note that $d$ is a pseudo-metric on $\Gamma$, and is a metric if we consider two curves to be equivalent if they are the same up to reparametrization.

Recall (Theorem \ref{LERWtoSLE}) that LERW converges weakly to $\SLE_2$ on the space $(\Gamma, d)$. We want to apply this result to the functions $f_r$ defined as follows. Given $0 < r < 1$ and $\alpha \in \Gamma$, we let
$$ f_r(\alpha) = \Pro{0}{W[0,\tau_\Disc] \cap \alpha[0,\rho_r] = \emptyset},$$
where
$$ \rho_r = \inf \{ t : \abs{\alpha(t)} = r \}.$$
We also define $f_r$ to be identically $1$ for $r \geq 1$ (think of $\rho_r = 0$ in that case, so that the above probability is $1$). Recall that Theorem \ref{harmeasure} states that if $\gamma$ is $\SLE_2$ then
$$ \Exp{}{f_r(\gamma)} \asymp r^{3/4}.$$

Unfortunately, the $f_r$ are not continuous on the space $(\Gamma, d)$. However, the following lemma shows that they can be approximated by continuous functions.

\begin{lemma} \label{contin}
For all $0 < r < 1$, there exists a function $\wt{f}_r$ that is uniformly continuous on the space $(\Gamma, d)$ such that for all $\alpha \in \Gamma$ $$ f_{r/2}(\alpha) \leq \wt{f}_r(\alpha) \leq f_{2r}(\alpha).$$
\end{lemma}

\begin{proof}
We define 
$$ \wt{f}_r(\alpha) = \frac{2}{3r} \int_{r/2}^{2r} f_s(\alpha) \,ds.$$
Note that for a fixed $\alpha$, $f_s(\alpha)$ is increasing, and therefore $f_s(\alpha)$ is integrable. In addition, the second assertion in the statement of the lemma follows immediately. It remains to show that $\wt{f}_r$ is uniformly continuous.

First of all, we claim that for all $\epsilon > 0$, there exists $\delta > 0$ such that for all $0 < r < 1$ and all $\alpha$, $\beta$ with $d(\alpha, \beta) < \delta$,
\begin{eqnarray}
 f_r(\alpha) \leq f_{r+\delta}(\beta) + \epsilon. \label{contina}
\end{eqnarray}

To prove this note that
$$ f_r(\alpha) - f_{r + \delta}(\beta) \leq \Pro{0}{W[0,\tau_\Disc] \cap \beta[0,\rho_{r+\delta}] \neq \emptyset; W[0,\tau_\Disc] \cap \alpha[0,\rho_{r}] = \emptyset}.$$  
By Lemma \ref{ckappa}, there exists $\nu > 0$ depending only on $\epsilon$ such that 
$$ \Pro{0}{ W[0,\tau_\Disc] \cap \alpha[0,\rho_{r}] = \emptyset; W(\tau_\Disc) \in \wt{C}_\nu(\alpha)} < \epsilon.$$
Furthermore, if $d(\alpha, \beta) < \delta$, then for every $z \in \beta[0,\rho_{r+\delta}]$, there exists $y \in \alpha[0,\rho_r]$ such that $\abs{z - y} < \delta$. Therefore, by the Beurling estimates (Theorem \ref{beurling}), one can make $\delta$ small enough so that 
$$ \Pro{0}{\abs{W(\tau^*_\alpha) - W(\tau^*_\beta)} > \nu} < \epsilon $$
where $\tau^*_\alpha = \tau_{\alpha[0,\rho_r]} \wedge \tau_\Disc$ and $\tau^*_\beta = \tau_{\beta[0,\rho_{r+\delta}]} \wedge \tau_\Disc$.

Therefore, for such a $\delta$,

\begin{eqnarray*}
& & \Pro{0}{W[0,\tau_\Disc] \cap \beta[0,\rho_{r+\delta}] \neq \emptyset; W[0,\tau_\Disc] \cap \alpha[0,\rho_{r}] = \emptyset} \\
&\leq& \Prob{}{E_1} + \Prob{}{E_2} + \Prob{}{E_3},
\end{eqnarray*} 
where
\begin{eqnarray*}
E_1 &=& \{ W(\tau^*_\alpha) \in C_\nu(\alpha) \} \\
E_2 &=& \{ W[0,\tau_\Disc] \cap \alpha[0,\rho_{r}] = \emptyset; W(\tau^*_\beta) \in D_{1-\nu} \} \\
E_3 &=& \{ W(\tau^*_\alpha) \in \partial \Disc \setminus C_\nu(\alpha); W(\tau^*_\beta) \in A_\nu \}
\end{eqnarray*} 
(recall that $A_\nu$ denotes the annulus $\Disc \setminus D_{1-\nu}$).

By our choice of $\nu$, $\Prob{}{E_1} < \epsilon$. Provided we take $\delta < \nu/2$, the events $E_2$ and $E_3$ are subsets of the event  
$$\left\{ \abs{W(\tau^*_\alpha) - W(\tau^*_\beta)} > \frac{\nu}{2} \right\},$$
 and therefore $\Prob{}{E_2}$ and $\Prob{}{E_3}$ can be made less than $\epsilon$. This proves the claim (\ref{contina}). 

Fix $0 < r < 1$. Given $\epsilon > 0$, let $\delta > 0$ be such that (\ref{contina}) holds (recall that $\delta$ depends only on $\epsilon$ and not on $r$) and suppose that $d(\alpha, \beta) < \delta$. Then since $f_s(\beta) \leq 1$ for all $s$ and $\beta$,
\begin{eqnarray*}
\wt{f}_r(\alpha) - \wt{f}_r(\beta) &=& \frac{2}{3r} \int_{r/2}^{2r} f_s(\alpha) \,ds - \frac{2}{3r} \int_{r/2}^{2r} f_s(\beta) \,ds \\
&\leq& \frac{2}{3r} \int_{r/2}^{2r} f_{s + \delta}(\beta) \,ds + \epsilon - \frac{2}{3r} \int_{r/2}^{2r} f_s(\beta) \,ds \\
&\leq& \frac{2}{3r}(\delta + \delta) + \epsilon.
\end{eqnarray*}
The latter can be made arbitrarily small by choosing $\delta$ small enough. By reversing the roles of $\alpha$ and $\beta$, one gets a similar lower bound, proving that $\wt{f}_r$ is uniformly continuous.
\end{proof}

\begin{lemma} \label{contconv}
There exists $C < \infty$ such that the following holds. Given a random walk $S$ and an independent LERW $\wh{S}^n$, we extend them to continuous curves $S_t$ and $\wh{S}^n_t$ by linear interpolation. Then for all $0 < r < 1$, there exists $N = N(r)$ such that for $n \geq N$,
$$ \frac{1}{C} r^{3/4} \leq \whExp{}{\Pro{}{S[0,\sigma_n] \cap \eta^2_{rn,n}(\wh{S}^n) = \emptyset}} \leq C r^{3/4}.$$
\end{lemma}

\begin{proof}
We'll prove the upper bound. The lower bound is proved in exactly the same fashion.

Fix $0 < r <1$. Recall that $S^{(n)}_t = n^{-1}S_{n^2t}$. By Proposition \ref{a}, there exists $N_1$ such that for $n \geq N_1$, and any realization of $\wh{S}_t^n$,
\begin{eqnarray*}
\Pro{}{S[0,\sigma_n] \cap \eta^2_{rn, n}(\wh{S}^n) = \emptyset} &=& \Pro{}{S^{(n)}[0,\sigma_\Disc] \cap \left( n^{-1} \eta^2_{rn, n}(\wh{S}^n) \right) = \emptyset} \\
&\leq& \Pro{}{W[0,\tau_\Disc] \cap \left( n^{-1} \eta^2_{rn, n}(\wh{S}^n) \right) = \emptyset} + r^{3/4} \\
&=& f_r(n^{-1} \wh{S}^n) + r^{3/4}.
\end{eqnarray*}

By Lemma \ref{contin}, 
$$ f_r(n^{-1} \wh{S}^n) \leq \wt{f}_{2r}(n^{-1} \wh{S}^n),$$
and $\wt{f}_{2r}$ is continuous in the metric $(\Gamma, d)$. Therefore, by the weak convergence of LERW to $\SLE_2$ described at the beginning of this section, there exists $N_2$ such that for $n \geq N_2$,
$$ \whExp{}{\wt{f}_{2r}(n^{-1} \wh{S}^n)} \leq \Exp{}{\wt{f}_{2r}(\gamma)} + r^{3/4}$$
where $\gamma$ denotes $\SLE_2$. Furthermore, applying first Lemma \ref{contin}, and then Theorem \ref{harmeasure},
$$  \Exp{}{\wt{f}_{2r}(\gamma)} \leq  \Exp{}{f_{4r}(\gamma)} \asymp r^{3/4}.$$
Therefore, the upper bound holds for $N = \max \{N_1, N_2 \}$. The lower bound is proved in the same fashion.

\end{proof}

We now prove the analogue of the previous lemma for the case where $S$ and $\wh{S}^n$ are discrete processes. The reason why the discrete case does not follow immediately from the continuous case is that we allow random walks that ``jump'', and therefore it's possible for two realizations of $S$ and $\wh{S}^n$ to avoid each other on the lattice $\Lambda$ but to intersect after they are made continuous curves by linear interpolation.

\begin{thm} \label{conv}
There exists a constant $C$ such that the following holds. For all $0 < r < 1$, there exists $N = N(r)$ such that for $n \geq N$,
$$ \frac{1}{C} r^{3/4} \leq  \Es(r n, n) \leq C r^{3/4}.$$
\end{thm}

\begin{proof}
Fix $0 < r < 1$. The lower bound follows immediately from Lemma \ref{contconv} and the fact that if the discrete processes intersect each other so too will the continuous curves. 

To prove the upper bound we introduce some notation that will be used only in this proof. Let
$$ S[0, \ldots, \sigma_n]$$
denote the discrete set of points in $\Lambda$ visited by $S$ between $S_0$ and $S(\sigma_n)$. We will write
$$ S[0,\sigma_n]$$
to denote the continuous set of points in $\Complex$ visited by the continuous curve $S_t$ from $S_0$ to $S(\sigma_n)$. We use similar notation for $\wh{S}^n$. In addition, we let 
$$ \eta^2 = \eta^2_{rn,n}\left(\wh{S}^n[0, \ldots, \wh{\sigma}_n] \right)$$
be the terminal part of the discrete LERW curve and
$$ \wt{\eta}^2 =  \eta^2_{rn,n}\left(\wh{S}^n[0, \wh{\sigma}_n] \right)$$
be the terminal part of the continuous LERW curve.

As in the proof of Lemma \ref{ckappa}, one can choose $\delta > 0$ small enough so that for all $n$ sufficiently large, and for all $z \in \partial B_n$, 
$$ \Pro{0}{S[0, \sigma_n] \cap B_{\delta n}(z) \neq \emptyset} < r^{3/4}.$$
Furthermore, given such a $\delta$, we can choose $\epsilon > 0$ and $N$ such that for all $n \geq N$, and all $z \in \partial B_n$, the following holds. Let  $y \in \Lambda$ be the closest point to $(1-\epsilon)z$. Then,
$$ \Pro{y}{S[0,\sigma_n] \subset B_{\delta n}(z)} > 1 - r^{3/4}.$$
Since the LERW path $\wh{S}^n$ is a subset of a random walk path, one can combine the previous two observations to show that there exists $\epsilon > 0$ and $N$ such that for all $n \geq N$, 
$$ \whExp{}{\Pro{}{S[0, \sigma_n] \cap \left(\wh{S}^n[0,\wh{\sigma}_n] \cap A_{\epsilon n}\right) \neq \emptyset}} < 2 r^{3/4},$$   
where $A_{\epsilon n}$ denotes the annulus $B_n \setminus B_{(1-\epsilon)n}$.

By the Beurling estimates (Theorem \ref{beurling}), if $R$ is the range of $S$, then for any realization of $\wh{S}^n$,
\begin{eqnarray*}
\Pro{0}{S[0,\sigma_n] \cap \left( \wt{\eta}^2 \cap B_{(1 - \epsilon)n} \right) \neq \emptyset; S[0, \ldots, \sigma_n] \cap \eta^2 = \emptyset} &\leq& C \left(\frac{R}{\epsilon n} \right)^{1/2}.
\end{eqnarray*}
Therefore, we can select $N$ large enough so that for all $n \geq N$,
\begin{eqnarray*}
\Es(rn,n) &=& \whExp{}{\Pro{}{S[0, \ldots, \sigma_n] \cap \eta^2 = \emptyset}} \\
&=& \whExp{}{\Pro{}{S[0, \sigma_n] \cap \wt{\eta}^2 = \emptyset}} + \whExp{}{\Pro{}{S[0, \sigma_n] \cap \left( \wt{\eta}^2 \cap A_{\epsilon n} \right) \neq \emptyset}} \\
&+& \whExp{}{\Pro{}{S[0,\sigma_n] \cap \left( \wt{\eta}^2 \cap B_{(1 - \epsilon)n} \right) \neq \emptyset; S[0, \ldots, \sigma_n] \cap \eta^2 = \emptyset}} \\
&\leq& Cr^{3/4}.
\end{eqnarray*}
\end{proof}

\begin{thm} \label{escape}
$$ \Es(n) \approx n^{-3/4}.$$
\end{thm}

\begin{proof}
We prove the upper bound using Proposition \ref{d}. One gets the lower bound in exactly the same way using Proposition \ref{e}. Let $\delta > 0$ be given. Let $C$ denote the larger of the constants in Theorem \ref{conv} and Proposition \ref{d}. Let $0 < r < 1/4$ be small enough so that
$$ \frac{\ln C}{\ln (1/ r)} < \delta.$$
By Theorem \ref{conv} and our choice of $r$, there exists $N$ such that for $n \geq N$,
$$ \abs{\frac{\ln {\Es(rn,n)}}{\ln{1/r}} + \frac{3}{4}} < \delta.$$

Any $n \geq N$ can be written uniquely as $n = r^{-j} s$ for some $j=0,1,2,\ldots$ and $N \leq s < r^{-1}N$. Therefore, by Proposition \ref{d},
\begin{eqnarray*}
\ln{\Es(n)} &=& \ln{\Es(r^{-j}s)} \\
&\leq& \ln{\left[C^j \Es(s) \prod_{k=1}^{j} \Es(r^{k  -j}s, r^{k - 1 -j}s) \right]} \\
&\leq& j\ln{C} + \ln{\Es(s)} + j(-\frac{3}{4} + \delta)\ln (1/r).
\end{eqnarray*}
Thus,
\begin{eqnarray*}
\frac{\ln{\Es(n)}}{\ln{n}} &\leq& \frac{j\ln{C} + \ln{\Es(s)} + j(-\frac{3}{4} + \delta)\ln{1/r}}{j \ln{1/r} + \ln{s}} \\
&\leq& \frac{j\ln{C} + \ln{\Es(N)} + j(-\frac{3}{4} + \delta)\ln{1/r}}{j \ln{1/r} + \ln{N}}.
\end{eqnarray*}
Therefore,
\begin{eqnarray*}
\limsup_{n \to \infty} \frac{\ln{\Es(n)}}{\ln{n}} &\leq& \limsup_{j \to \infty} \frac{j\ln{C} + \ln{\Es(N)} + j(-\frac{3}{4} + \delta)\ln{1/r}}{j \ln{1/r} + \ln{N}} \\
&=& \frac{\ln{C}}{\ln{1/r}} + -\frac{3}{4} + \delta \\
&\leq& -\frac{3}{4} + 2 \delta.
\end{eqnarray*}
This proves the upper bound, since $\delta$ was arbitrary. As mentioned before, an identical proof will work for the lower bound.
\end{proof}

\subsection{Deriving the growth exponent from the intersection exponent} \label{growthsec}
In this section, we show that $\Gr(n) \approx n^{5/4}$. We first prove a lemma which relates the probability that a point $z$ is on the LERW path to an intersection probability for a LERW.

\begin{lemma} \label{lerwformula} Fix $z \in B_n$. Let $S$ be a random walk and  let $X$ be an independent random walk started at $z$ conditioned to hit $0$ before leaving $B_n$. Then
$$ \Pro{}{ z \in \wh{S}^n[0, \wh{\sigma}_n] } = G_n(0,z) \Pro{z} { \operatorname{L}(X[0, \xi^X_0]) \cap S[1, \sigma_n] = \emptyset }.$$
\end{lemma}

\begin{proof}
Let
$$ \wt{\sigma}_z = \left\{ \begin{array}{ll}
         \max \{k \leq \sigma_n : S_k = z \} & \mbox{if $\xi_z < \sigma_n$};\\
         \sigma_n & \mbox{if $\sigma_n < \xi_z$},\end{array} \right.$$
and let $\wt{\sigma}_0^z = \max \{k \leq \wt{\sigma}_z : S_k = 0 \}$.

By the definition of the loop-erasing procedure, 
$$ \Pro{}{z \in \wh{S}^n[0, \wh{\sigma}_n]} = \Pro{} { \xi_z < \sigma_n; \operatorname{L}(S[\wt{\sigma}_0^z, \wt{\sigma}_z]) \cap S[\wt{\sigma}_z + 1, \sigma_n] = \emptyset }.$$

Conditioned on the event $\{\xi_z < \sigma_n\}$, $S[\wt{\sigma}_0^z,\wt{\sigma}_z]$ and $S[\wt{\sigma}_z, \sigma_n]$ are independent. $S[\wt{\sigma}_z, \sigma_n]$  has the same distribution as a random walk started at $z$, conditioned to leave $B_n$ before returning to $z$. $S[\wt{\sigma}_0^z,\wt{\sigma}_z]$ has the same distribution as the time reversal of $X[0,\xi^X_0]$ and therefore $\operatorname{L}(S[\wt{\sigma}_0^z, \wt{\sigma}_z])$ has the same distribution as the time reversal of $\operatorname{L}(X[0, \xi^X_0])$  Thus,
\begin{eqnarray*} 
& & \Pro{} {\xi_z < \sigma_n; \operatorname{L}(S[\wt{\sigma}_0^z, \wt{\sigma}_z]) \cap S[\wt{\sigma}_z + 1, \sigma_n] = \emptyset } \\
&=& \Pro{}{\xi_z < \sigma_n} \condPro{z}{\operatorname{L}(X[0, \xi^X_0]) \cap S[1, \sigma_n] = \emptyset}{\sigma_n < \xi_z} \\
&=& \Pro{}{\xi_z < \sigma_n}\Pro{z}{\sigma_n < \xi_z}^{-1} \Pro{z}{\operatorname{L}(X[0, \xi^X_0]) \cap S[1, \sigma_n] = \emptyset}.
\end{eqnarray*} 

The lemma now follows from the fact that
$$ \Pro{}{\xi_z < \sigma_n} \Pro{z}{\sigma_n < \xi_z}^{-1} = \Pro{}{\xi_z < \sigma_n} G_n(z) = G_n(0,z).$$
\end{proof}

We finally have all the tools needed to prove our main theorem.

\begin{proof}[Proof of Theorem \ref{growth}]
We first prove the upper bound. For $z \in B_n$, let
$$r = r(z) = \frac{1}{4} \min \{ \abs{z}, n - \abs{z} \}.$$
Then by Lemma \ref{lerwformula},
\begin{eqnarray*}
\Pro{}{ z \in \wh{S}^n[0, \wh{\sigma}_n] } &=& G_n(0,z) \Pro{z} { \operatorname{L}(X[0, \xi^X_0]) \cap S[1, \sigma_n] = \emptyset } \\
&\leq& G_n(0,z) \Pro{z} { \wh{X}[0, \wh{\sigma}^X_{B_r(z)}] \cap S[1, \sigma_{B_r(z)}] = \emptyset }.
\end{eqnarray*}
By Propositions \ref{aw} and \ref{ax}, $\wh{X}[0, \wh{\sigma}^X_{B_r(z)}]$ has the same distribution up to a constant as $\wh{S}[0, \wh{\sigma}_{B_r(z)}]$. Therefore, by Lemma \ref{2escapes},
\begin{eqnarray*}
\Pro{}{ z \in \wh{S}^n[0, \wh{\sigma}_n] } &\leq& C G_n(0,z) \Pro{z} { \wh{S}[0, \wh{\sigma}_{B_r(z)}] \cap S[1, \sigma_{B_r(z)}] = \emptyset } \\
&=& C G_n(0,z) \wt{\Es}(r) \\
&\leq& C G_n(0,z) \Es(4r).
\end{eqnarray*}
By Theorem \ref{escape}, for all $\epsilon > 0$, there exists $M$ such that for all $k > M/4$,
$$ \Es(k) \leq k^{-3/4 + \epsilon}.$$
Also, by \cite[Proposition 6.3.5]{LL08}, for all $z \in B_n \setminus \{0\}$, 
$$G_n(0,z) \asymp \ln \frac{n}{\abs{z}}.$$    
Thus, for $n > 3M$,
\begin{eqnarray*}
\Gr(n) &=& \sum_{z \in B_n} \Pro{} { z \in \wh{S}^n[0, \wh{\sigma}_n] } \\
&\leq& C \left[ M^2 + \sum_{k = M}^{n/2} k \ln \frac{n}{k} \Es(4k) + \sum_{k=n/2}^{n-M} k \ln \frac{n}{k} \Es(4(n - k)) + M n \right] \\
&\leq& C \left[ M^2 + \sum_{k = M}^{n/2} k \ln \frac{n}{k} k^{-3/4 + \epsilon} + \sum_{k=n/2}^{n-M} k \ln \frac{n}{k} (n - k)^{-3/4 + \epsilon} + M n \right] \\
&\leq& C \left[ M^2 + n^{5/4 + \epsilon} + M n \right].
\end{eqnarray*}
Since $M$ does not depend on $n$, for all $n$ sufficiently large one gets that
$$ \Gr(n) \leq C n^{5/4 + \epsilon}$$
which gives the upper bound since $\epsilon > 0$ was arbitrary.

We now prove the lower bound. As before, let
$$r = r(z) = \frac{1}{4} \min \{ \abs{z}, n - \abs{z} \},$$
and suppose that $r > n/16$ so that $n/4 < \abs{z} < 3n/4$.

Let $c_1$ be as in the statement of the separation lemma (Theorem \ref{sep}) and let
$$ D_r := r^{-1} \min \{\operatorname{dist}(S(\sigma_{B_r(z)}), \wh{X}[0, \wh{\sigma}^X_{B_r(z)}]), \operatorname{dist}(\wh{X}(\wh{\sigma}^X_{B_r(z)}), S[0, \sigma_{B_r(z)}]) \}.$$
Then using a similar argument to the one in the proof of \ref{2escapes},
$$ \condPro{z}{S[1, \sigma_n] \cap (\wh{X}[0,\wh{\xi}^X_{n/16}] \cup B_{n/8}) = \emptyset }{ S[1, \sigma_{B_r(z)}] \cap \wh{X}[0,\wh{\sigma}_{B_r(z)}] = \emptyset; D_r > c_1 } $$
can be bounded below by a constant $c > 0$.
Furthermore, by \cite[Proposition 1.6.7]{Law91}, for any $w \in \partial B_{n/16}$,
$$\condPro{w}{\xi_{0} < \sigma_{n/8} }{ \xi_{0} < \sigma_n} = \frac{\Pro{w}{\xi_{0} < \sigma_{n/8}}}{\Pro{w}{\xi_{0} < \sigma_{n}}} \geq c.$$
Hence,
$$ \condPro{z}{S[1, \sigma_n] \cap \wh{X}[0,\wh{\xi}^X_{0}] = \emptyset }{ S[1, \sigma_{B_r(z)}] \cap \wh{X}[0,\wh{\sigma}_{B_r(z)}] = \emptyset; D_r > c_1 } \geq c.$$
Therefore, by the separation lemma (Theorem \ref{sep}),
\begin{eqnarray*}
& & \Pro{z}{S[1, \sigma_n] \cap \wh{X}[0,\wh{\xi}^X_{0}] = \emptyset} \\
&\geq& \Pro{z}{S[1, \sigma_n] \cap \wh{X}[0,\wh{\xi}^X_{0}] = \emptyset ; S[1, \sigma_{B_r(z)}] \cap \wh{X}[0,\wh{\sigma}_{B_r(z)}] = \emptyset ; D_r \geq c_1} \\
&\geq& c \Pro{z}{S[1, \sigma_{B_r(z)}] \cap \wh{X}[0,\wh{\sigma}_{B_r(z)}] = \emptyset; D_r > c_1 } \\
&\geq& c \Pro{z}{S[1, \sigma_{B_r(z)}] \cap \wh{X}[0,\wh{\sigma}_{B_r(z)}] = \emptyset}.
\end{eqnarray*}
As before, the last quantity is comparable to $\Es(r)$. Therefore, for all $z$ such that $n/4 \leq \abs{z} \leq 3n/4$,
$$ \Pro{}{ z \in \wh{S}^n[0, \wh{\sigma}_n] } \geq c G_n(0,z) \Es(r).$$

Now let $\epsilon > 0$. By Theorem \ref{escape}, there exists $M$ such that for all $k > M$,
$$ \Es(k) \geq k^{-3/4 - \epsilon}.$$
Therefore, for $n > 16M$,
\begin{eqnarray*}
\Gr(n) &=& \sum_{z \in B_n} \Pro{0} { z \in \wh{S}^n[0, \wh{\sigma}_n] } \\
&\geq& \sum_{n/4 \leq \abs{z} \leq 3n/4} \Pro{0} { z \in \wh{S}^n[0, \wh{\sigma}_n] } \\
&\geq& c \left[ \sum_{k = n/4}^{n/2} k \ln \frac{n}{k} \Es(k) + \sum_{k=n/2}^{3n/4} k \ln \frac{n}{k} \Es(n - k) \right] \\
&\geq& c \left[ \sum_{k = n/4}^{n/2} k \ln \frac{n}{k} k^{-3/4 - \epsilon} + \sum_{k=n/2}^{3n/4} k \ln \frac{n}{k} (n - k)^{-3/4 - \epsilon} \right] \\
&\geq& c n^{5/4 - \epsilon}.
\end{eqnarray*}
This proves the lower bound since $\epsilon$ was arbitrary.
\end{proof}

\end{document}